\documentclass[10pt]{article}
\usepackage{amsmath,amsfonts,amssymb,amsthm}
\usepackage{enumerate}
\usepackage[utf8]{inputenc}
\usepackage[english]{babel}
\usepackage{graphicx}
\usepackage{color}
\usepackage{pgf,tikz}
\usepackage[colorlinks=true]{hyperref}
\hypersetup{urlcolor=blue, citecolor=blue}
\usepackage{tikz}
\usepackage{pgfplots}
\usepackage{cite}
\usepackage{mathrsfs}
\usepackage{amssymb}
\usepackage{eurosym}
\usepackage{mathabx}
\usepackage{enumitem}

\usepackage[colorinlistoftodos]{todonotes}

\usepackage{caption}
\pgfplotsset{
  log ticks with fixed point,
}

\newtheorem{deff}{Definition}[section]
\newtheorem{prop}[deff]{Proposition}
\newtheorem{thm}[deff]{Theorem}
\newtheorem{lem}[deff]{Lemma}
\newtheorem{cor}[deff]{Corollary}
\newtheorem{rmk}[deff]{Remark}

\newcounter{cst}
\newcommand{\ctel}[1]{C_{\refstepcounter{cst}\label{#1}\thecst}}
\newcommand{\cter}[1]{C_{\ref{#1}}}

\def\R{\mathbb{R}}

\def\O{\Omega}

\def\p{\partial}

\def\bs{{\boldsymbol{s}}}
\def\bdt{{\boldsymbol{\dt}}}

\def\div{{\text{div}}}

\def\ov#1{\overline{#1}}

\def\Hh{\mathcal{H}}

\def\grad{\nabla}
\def\d{{\rm d}}
\def\1{\mathbf{1}}
\def\be{\begin{equation}}
\def\ee{\end{equation}}

\def\0{{\bf 0}}
\def\sig{\sigma}

\def\Aa{\mathtt{A}}
\def\Bb{\mathtt{B}}

\def\Dd{\mathtt{D}}
\def\Ee{\mathscr{E}}

\def\Hh{\mathcal{H}}

\def\Jj{\mathfrak{J}}

\def\Rr{\mathtt{R}}
\def\Ss{\mathcal{S}}
\def\Tt{\mathscr{T}}

\def\Eee{\mathfrak{E}}

\def\bp{{\boldsymbol p}}

\def\bu{{\boldsymbol u}}

\def\wt#1{\widetilde{#1}}

\def\fs{\mathfrak{s}}
\def\fp{\mathfrak{p}}

\def\dt{{\Delta t}}
\def\dtn{{\Delta t^n}}

\def\dummys{\varsigma}
\def\dummyp{\pi}
\def\phead{\vartheta}
\def\size{h}
\def\regul{\zeta}
\def\measure{\nu}
\def\scheme{\Hh}
\def\fntest{\varphi}
\def\Energy{\text{\EUR{}}}
\def\kirsqr{\Theta}
\def\kircho{\Phi}
\def\kirmin{\Upsilon}

\def\wh#1{\widehat{#1}}

\def\iint{\int\!\!\!\int}

\usepackage{pgfplots}
\usepackage{pgfplotstable}
\pgfplotsset{compat=1.15}

\pgfplotscreateplotcyclelist{MyCyclelist}{%
{red, mark = asterisk, thick},
{blue, mark = o, thick},
{orange, mark = triangle, thick},
{green, mark = square, thick},
{magenta, mark = diamond, thick},
}
\pgfplotscreateplotcyclelist{MyCyclelist1}{%
{red, mark = asterisk, thick},
{blue, mark = o, thick},
{green, mark = triangle,thick},
{orange,mark = square, thick},
{magenta, mark = pentagon, thick},
{blue, mark = otimes, thick},
{pink, mark = oplus, thick},
{gray, mark = |, thick},
}

\numberwithin{equation}{section}

\begin{document}

\title{Upstream mobility Finite Volumes for the Richards equation in heterogenous domains}
\author{Sabrina \textsc{Bassetto}\thanks{IFP Energies nouvelles, 1 et 4 avenue de Bois Pr\'eau, 92852 Rueil-Malmaison Cedex, France. \texttt{sabrina.bassetto@ifpen.fr}, \texttt{guillaume.enchery@ifpen.fr}, \texttt{quang-huy.tran@ifpen.fr}}
\and
Cl\'ement \textsc{Canc\`es}\thanks{Inria, Univ. Lille, CNRS, UMR 8524 -- Laboratoire Paul Painlev\'e, 59000 Lille, France. \texttt{clement.cances@inria.fr}}
\and
Guillaume \textsc{Ench\'ery}${}^*$
\and
Quang-Huy \textsc{Tran}${}^*$
}

\date{}
\maketitle

\begin{abstract}
This paper is concerned with the Richards equation in a heterogeneous domain, each subdomain of which is homogeneous and represents a rocktype. Our first contribution is to rigorously prove convergence toward a weak solution of cell-centered finite-volume schemes with upstream mobility and without Kirchhoff's transform. Our second contribution is to numerically demonstrate the relevance of locally refining the grid at the interface between subregions, where discontinuities occur, in order to preserve an acceptable accuracy for the results computed with the schemes under consideration.
\end{abstract}
{\small
\begin{center}
\textbf{Keywords}\medskip\\
Richards' equation, heterogeneous domains, finite-volume schemes, mobility upwinding
\end{center}
\begin{center}
\textbf{Mathematics subject classification}\medskip\\
65M08, 65M12, 76S05
\end{center}
}

\section{Presentation of the continuous model}
The Richards equation \cite{Richards31} is one of the most well-known simplified models for water filtration in unsaturated soils. While it has been extensively studied in the case of a homogeneous domain, the heterogeneous case seems to have received less attention in the literature, at least from the numerical perspective. The purpose of this paper is to investigate a class of discretization scheme for a special instance of heterogeneous domains, namely, those with piecewise-uniform physical properties.

Before stating our objectives in a precise manner, a few prerequisites must be introduced regarding the model in \S\ref{sse:richards}--\S\ref{ssec:cont.stability} and the scheme in \S\ref{ssec:mesh}--\S\ref{ssec:scheme}. The goal of the paper is fully described in \S\ref{sse:goals}, in relation with other works. 
Practical aspects related to the numerical resolution are detailed in \S\ref{sec:numerics} and results on illustrative test cases are shown in \S\ref{sec:results}. 
A summary of our main results is provided in \S\ref{ssec:main}, together with the outline of the paper.

\subsection{Richards' equation in heterogeneous porous media}\label{sse:richards}
Let $\O \subset \R^d$, where $d\in \{2,3\}$, be a connected open polyhedral domain with Lipschitz boundary $\p\O$. A porous medium defined over the region $\O$ is characterized by
\begin{itemize}[noitemsep]
\item[--] the porosity $\phi : \O \rightarrow (0,1]$;
\item[--] the permeability $\lambda : \O \rightarrow \R_+^*$;
\item[--] the mobility function $\eta : [0,1]\times \O \rightarrow \R_+$;
\item[--] the capillary pressure law $\Ss : \R \times \O \rightarrow [0,1]$.
\end{itemize}
The conditions to be satisfied by $\phi$, $\lambda$, $\eta$ and $\Ss$ will be elaborated on later.
In a homogeneous medium, these physical properties are uniform over $\O$, i.e.,
\[
\phi(x) = \phi_0, \qquad \lambda(x) = \lambda_0, \qquad \eta(s,x)= \eta_0(s),
\qquad \Ss(p,x) = \Ss_0(p)
\]
for all $x\in \O$. In a heterogeneous medium, the dependence of $\phi$, $\lambda$, $\eta$ and $\Ss$ on $x$ must naturally be taken into account. The quantity $s$, called saturation, measures the relative volumic presence of water in the medium.  The quantity $p$ is the pressure.

Let $T>0$ be a finite time horizon. We designate by $Q_T = (0,T) \times \O$ the space-time domain of interest. Our task is to find the saturation field $s : Q_T \rightarrow [0,1]$ and the pressure field $p : Q_T \rightarrow \R$ so as to satisfy
\begin{itemize}[noitemsep]
\item the interior equations
\begin{subequations}\label{eq:bigsyst}
\begin{alignat}{2}
\phi(x) \, \p_t s + \div\, F & = 0 & \qquad & \text{in}\; Q_{T}, \label{eq:massbal}\\
F + \lambda(x) \,\eta(s,x) \, \grad (p - \varrho g\cdot x) & = 0 & \qquad & \text{in}\; Q_{T},\label{eq:darcymuskat}\\
s - \Ss(p,x) & = 0 & \qquad & \text{in}\;  Q_{T} ; \label{eq:spcapillar}
\end{alignat}
\item the boundary conditions
\begin{alignat}{2}
F\cdot n(x) & = 0 &\qquad & \text{on} \; (0,T) \times \Gamma^{\rm N}, \label{eq:cont.Neumann}\\
p(t,x) & = p^{\rm D}(x) & \qquad & \text{on}\; (0,T)\times \Gamma^{\rm D}; \label{eq:cont.Dirichlet}
\end{alignat}
\item the initial data
\begin{alignat}{2}
s(0,x) & = s^0(x) & \qquad & \text{in} \; \O. \label{eq:cont.init}
\end{alignat}
\end{subequations}
\end{itemize}

The partial differential equation \eqref{eq:massbal} expresses the water volume balance. The flux $F$ involved in this balance is given by the Darcy-Muskat law \eqref{eq:darcymuskat}, in which $g$ is the gravity vector and $\varrho$ is the known constant density of water, assumed to be incompressible. It is convenient to introduce
\begin{equation}
\psi = - \varrho g\cdot x, \qquad \phead = p+ \psi,
\end{equation}
referred to respectively as gravity potential and hydraulic head. In this way, the Darcy-Muskat law \eqref{eq:darcymuskat} can be rewritten as
\[
F + \lambda(x) \,\eta(s,x) \, \grad (p + \psi)  = F + \lambda(x) \,\eta(s,x) \, \grad \phead = 0 . 
\]
Equation \eqref{eq:spcapillar} connecting the saturation $s$ and the pressure $p$ is the capillary pressure relation. The boundary $\p\O$ is split into two non-overlapping parts, viz., 
\be
\p\O = \Gamma^{\rm N} \cup \Gamma^{\rm D}, \qquad
\Gamma^{\rm N} \cap \Gamma^{\rm D} = \emptyset,
\ee
where $\Gamma^{\rm N}$ is open and $\Gamma^{\rm D}$ is closed, the latter having a positive $(d-1)$-dimensional Hausdorff measure $\measure^{d-1}(\Gamma^{\rm D})>0$. The no-flux Neumann condition \eqref{eq:cont.Neumann} is prescribed on $(0,T)\times\Gamma^{\rm N}$, where $n(x)$ is the outward normal unit vector at $x\in \Gamma^{\rm N}$.
The Dirichlet condition \eqref{eq:cont.Dirichlet} with a known Lipschitz function $p^{\rm D}  \in W^{1,\infty}(\O)$ is imposed on $(0,T)\times\Gamma^{\rm D}$. Note that, in our theoretical development, the function $p^{\rm D}$ is assumed to be defined over the whole domain $\O$, which is stronger than a data $p^{\rm D}\in L^{\infty}(\Gamma^{\rm D})$ given only on the boundary. The assumption that $p^{\rm D}$ does not depend on time can be removed by following the lines of~\cite{CNV}, but we prefer here not to deal with time-dependent boundary data in order to keep the presentation as simple as possible. Finally, the initial data $s^0 \in L^{\infty}(\Omega ; [0,1])$ in \eqref{eq:cont.init} is also a given data.

In this work, we restrict ourselves to a specific type of heterogeneous media, defined as follows. We assume that the domain $\O$ can be partitioned into several connected polyhedral subdomains $\O_i$, $1 \leq i \leq I$. Technically, this means that if $\Gamma_{i,j}$ denotes the interface between $\Omega_i$ and $\Omega_j$ (which can be empty for some particular choices of $\{i,j\}$), then
\be\label{eq:geom}
\O_i \cap \O_j = \emptyset,  \quad \ov \O_i \cap \ov \O_j = \Gamma_{i,j},  \; \text{if}\; i \neq j,
\quad 
\O = \Big(\bigcup_{1\leq i\leq I} \O_i\Big) \cup \Gamma, 
\ee
with $\Gamma = \bigcup_{i\neq j} \Gamma_{i,j}$. Each of these subdomains corresponds to a distinctive rocktype. Inside each $\O_i$, the physical properties are homogeneous. In other words,
\[
\phi(x) = \phi_i, \qquad \lambda(x) = \lambda_i, \qquad \eta(s,x)= \eta_i(s),
\qquad \Ss(p,x) = \Ss_i(p)
\]
for all $x\in \O_i$. Therefore, system \eqref{eq:bigsyst} is associated with
\begin{subequations}\label{eq:subdomdata}
\begin{alignat}{2}
\phi(x) & = \sum_{1\leq i\leq I} \phi_i \, \1_{\O_i}(x), 
 & \qquad \eta(s,x) & = \sum_{1\leq i\leq I} \eta_i(s) \, \1_{\O_i}(x), \label{eq:phieta}\\
\lambda(x) & = \sum_{1\leq i\leq I} \lambda_i \, \1_{\O_i}(x),
 & \qquad \Ss(p,x) & = \sum_{1\leq i\leq I} \Ss_i(p) \, \1_{\O_i}(x), \label{eq:lambdaSs}
\end{alignat}
\end{subequations}
where $\1_{\O_i}$ stands for the characteristic function of $\O_i$. For all $i\in\{1,\ldots, I\}$, we assume that $\phi_i \in (0,1]$ and $\lambda_i >0$. Furthermore, we require that
\begin{subequations}\label{eq:EtaSs}
\be\label{eq:Eta.1}
\eta_i \text{ is increasing on } [0,1], \qquad
\eta_i(0)= 0, \qquad \eta_i(1) = \frac{1}{\mu},
\ee
where $\mu>0$ is the (known) viscosity of water. In addition to the assumption that $\Ss(\cdot,x)$, defined in \eqref{eq:lambdaSs}, is absolutley continuous and nondecreasing, the functions $\Ss_i$ are also subject to some generic requirements commonly verified the models available in the literature: for each $i\in\{1,\ldots , I\}$, there exists $\ov p_i \leq 0$ such that 
\be\label{eq:Ss.1}
\Ss_i \text{ is increasing on } (-\infty, \ov p_i], \qquad 
\lim_{p\to-\infty} \Ss_i(p)= 0, \qquad \Ss_i \equiv 1 \text{ on } [\ov p_i, +\infty).
\ee
\end{subequations}
This allows us to define an inverse $\Ss_i^{-1}: (0,1] \to (-\infty, \ov p_i]$ such that $\Ss_i \circ \Ss_i^{-1}(s) = s$ for all $s \in (0,1]$.
We further assume that for all $i \in \{1, \dots, I\}$
the function $\Ss_i$ is bounded in $L^1(\R_-)$, or equivalently, that 
$\Ss_i^{-1} \in L^1(0,1)$. 
It thus makes sense to consider the capillary energy density functions $\Energy_i: \R \times \O_i \to \R_+$ defined by 
\be\label{eq:defenergy}
\Energy_i(s,x) = \int_{\Ss_i(p^{\rm D}(x))}^{s} \phi_i (\Ss_i^{-1}(\dummys)- p^{\rm D}(x)) \, \d \dummys.
\ee
For all $x \in \O_i$, the function $\Energy_i(\cdot, x)$ is nonnegative, convex since $\Ss_i^{-1}$ is monotone, 
and bounded on $[0,1]$ as a consequence of the integrability of $\Ss_i$. For technical reasons that will 
appear clearly later on, we further assume that
\be\label{eq:hyp.xi}
\sqrt{\eta_i \circ \Ss_i} \in L^1(\R_-), \qquad \forall i \in \{1,\dots, I\}.
\ee

Let $Q_{i,T} = (0,T)\times \Omega_i$ be the space-time subdomains for $1\leq i\leq I$. The interior equations \eqref{eq:massbal}--\eqref{eq:spcapillar} then boil down to

\begin{subequations}\label{eq:locsyst}
\begin{alignat}{2}
\phi_i \, \p_t s + \div\, F & = 0 & \qquad & \text{in}\; Q_{i,T}, \label{eq:cont.mass}\\
F + \lambda_i \,\eta_i \, \grad (p + \psi) & = 0 & \qquad & \text{in}\; Q_{i,T},\label{eq:cont.flux}\\
s - \Ss_i(p) & = 0 & \qquad & \text{in}\;  Q_{i,T} . \label{eq:Ss}
\end{alignat}
\end{subequations}
At the interface $\Gamma_{i,j}$ between $\O_i$ and $\O_j$, $i\neq j$, any solution of \eqref{eq:massbal}--\eqref{eq:spcapillar} satisfies the matching conditions
\begin{subequations}\label{eq:matchcond}
\begin{alignat}{2}
F_i\cdot n_i + F_j\cdot n_j &= 0 & \qquad & \text{on}\; (0,T)\times \Gamma_{i,j}, \label{eq:cont.flux.interf}\\
p_i - p_j & = 0 & \qquad & \text{on}\; (0,T)\times \Gamma_{i,j} . \label{eq:cont.p.interf}
\end{alignat}
\end{subequations}
In the continuity of the normal fluxes \eqref{eq:cont.flux.interf}, which is enforced by the conservation of water volume, $n_i$ denotes the outward normal to $\p\Omega_i$ and  
$F_i\cdot n_i$ stands for the trace of the normal component of $F_{|_{Q_{i,T}}}$ on $(0,T)\times\p\O_i$.
In the continuity of pressure \eqref{eq:cont.p.interf}, which also results from \eqref{eq:massbal}--\eqref{eq:spcapillar}, $p_i$ denotes the trace on $(0,T)\times \p\O_i$ of the pressure $p_{|_{Q_{i,T}}}$ in the $i$-th domain. 

\subsection{Stability features and notion of weak solutions}\label{ssec:cont.stability}
We wish to give a proper sense to the notion of weak solution for problem \eqref{eq:bigsyst}. To achieve this purpose, we need a few mathematical transformations the definition of which crucially relies on a fundamental energy estimate at the continuous level. The calculations below are aimed at highlighting this energy estimate and will be carried out in a formal way, in constrast to those in the fully discrete setting.

Multiplying~\eqref{eq:cont.mass} by $p-p^{\rm D}$, invoking \eqref{eq:defenergy}, 
integrating over $\O_i$ and summing over $i$, we end up with
\be\label{eq:NRG.0}
\frac{\d}{\d t} \, \sum_{i=1}^I\int_{\O_i} \Energy_i(s,x) \, \d x + \sum_{i=1}^I \int_{\O_i}\div\, F (p- p^{\rm D} ) \, \d x = 0.
\ee
We now integrate by parts the second term. Thanks to the matching conditions 
\eqref{eq:matchcond} and the regularity of $p^{\rm D}$, we obtain 
\[
\Aa := \sum_{i=1}^I \int_{\O_i}\div\, F (p-p^{\rm D}) \, \d x = -  \sum_{i=1}^I \int_{\O_i} F \cdot \grad ( p - p^{\rm D} ) \, \d x.
\]
It follows from the flux value \eqref{eq:cont.flux} that
\begin{align*}
\Aa &  =  \sum_{i=1}^I \int_{\O_i} \lambda_i \eta_i(s) 
\grad ( p +\psi ) \cdot \grad ( p - p^{\rm D} ) \, \d x \\
& \, =  \sum_{i=1}^I \int_{\O_i} \lambda_i \eta_i(s)  |\grad p |^2 \, \d x 
-  \sum_{i=1}^I \int_{\O_i} \lambda_i \eta_i(s) \grad \psi \cdot \grad p^{\rm D} \, \d x \\
& \,+ \; \sum_{i=1}^I \int_{\O_i} \lambda_i \eta_i(s)  \grad p \cdot \grad (\psi - p^{\rm D} ) \, \d x.
\end{align*}
Young's inequality, combined with the boundedness of $\grad p^{\rm D}$, $\grad \psi$, $\lambda$ and $\eta$, yields
\[
\mathtt{A} \, \geq \, \frac12 \sum_{i=1}^I \int_{\O_i} \lambda_i \eta_i(s)  |\grad p |^2 \, \d x  - C
\]
for some $C\geq 0$ depending only on $\lambda$, $\eta$, $\psi$, $\mu$,  $\O$ and $p^{\rm D}$.

Let us define the energy $\Eee: [0,T] \to \R_+$ by
\[
\Eee(t) = \sum_{i=1}^I\int_{\O_i} \Energy_i(s(t,x),x) \, \d x, \qquad 0 \leq t \leq T .
\]
Integrating~\eqref{eq:NRG.0} w.r.t. time results in
\be\label{eq:NRG.1}
\Eee(T) + \frac12 \sum_{i=1}^I \iint_{Q_{i,T}} \lambda_i \eta_i(s)  |\grad p |^2 \, \d x \, \d t \leq \Eee(0) + CT.
\ee
Estimate~\eqref{eq:NRG.1} is the core of our analysis. However, it is difficult to use in its present form since $\eta_i(s) = \eta_i(\Ss_i(p))$ vanishes as $p$ tends to $-\infty$, so that the control of $\grad p$ degenerates. To circumvent this difficulty, we resort to the nonlinear functions (customarily referred to as the Kirchhoff transforms) $\kirsqr_i:\R \to \R$, 
$\kircho_i: \R \to \R$, 
and $\kirmin: \R \times \O \to \R$ respectively defined by
\begin{subequations}
\begin{alignat}{2}
\kirsqr_i(p) & = \int_{0}^p \sqrt{\lambda_i \eta_i \circ\Ss_i(\dummyp)} \, \d \dummyp, & \qquad p & \in \R, \label{eq:xi.def} \\
\kircho_i(p) & = \int_{0}^p {\lambda_i \eta_i \circ\Ss_i(\dummyp)} \, \d \dummyp, & \qquad p & \in \R,\\
\kirmin(p) & = \int_0^p  \min_{1\leq i \leq I}
\sqrt{\lambda_i \eta_i \circ \Ss_i(\dummyp)} \, \d \dummyp, & \qquad  p & \in \R, \label{eq:Upsilon}
\end{alignat}
\end{subequations}
the notion of $\kirmin$ being due to \cite{EEM06}.  
Bearing in mind that $\Eee(T) \geq 0$, estimate~\eqref{eq:NRG.1} implies that 
\be\label{eq:NRG.2}
\sum_{i=1}^I \iint_{Q_{i,T}} |\grad \kirsqr_i(p)|^2 \, \d x \, \d t \leq 2 (\Eee(0) + CT) < +\infty.
\ee
As $\kircho_i \circ \kirsqr_i^{-1}$ is Lipschitz continuous, this also gives rise to a $L^2(Q_{i,T})$-estimate on $\grad \kircho_i(p)$. The functions $\sum_i \kirsqr_i(p)\1_{\O_i}$ and $\sum_i \kircho_i(p)\1_{\O_i}$ are 
in general discontinuous across the interfaces $\Gamma_{i,j}$, 
unlike $\kirmin(p)$. Since the functions $\kirmin\circ \kirsqr_i^{-1}$ are Lipschitz continuous, we can readily infer from \eqref{eq:NRG.2} that
\be\label{eq:NRG.Upsilon}
\iint_{Q_T} |\grad \kirmin(p)|^2 \, \d x \leq C
\ee
for some $C$ depending on $T$, $\O$, $\lVert\grad p^{\rm D}\rVert_\infty$, the $\lVert\Ss_i\rVert_{L^1(\R_-)}$'s and
\[
\ov{\lambda} = \lVert\lambda\rVert_{L^{\infty}(\Omega)} = \max_{1\leq i\leq I}\lambda_i, \qquad
\ov{\eta} = \lVert\eta\rVert_{L^{\infty}(\Omega)} = \max_{1\leq i\leq I} \lVert\eta_i\rVert_{L^{\infty}(\Omega)} = \frac{1}{\mu},
\]
the last equality being due to \eqref{eq:Eta.1}.

Moreover, $\kirmin(p) - \kirmin(p^{\rm D})$ vanishes on $(0,T)\times \Gamma^{\rm D}$. 
Poincar\'e's inequality provides a $L^2(Q_T)$-estimate on $\kirmin(p)$ since 
$\Gamma^{\rm D}$ has positive measure and since $\kirmin(p^{\rm D})$ 
is bounded in $\O$. In view of assumption~\eqref{eq:hyp.xi}, 
the functions $\kirsqr_i$ and $\kirmin$ are bounded on $\R_-$. Besides, for $p\geq 0$, 
$\eta_i\circ \Ss_i(p) = 1/\mu$, so that
$
\kirsqr_i(p) = p \sqrt{\lambda_i/\mu} 
$
and $\kirmin(p) = \min_{1\leq i \leq I} p \sqrt{\lambda_i/\mu}$.
It finally comes that 
\be\label{eq:xi_Upsilon}
\kirsqr_i(p) \leq C (1+\kirmin(p)), \qquad \forall p \in \R, \; 1 \leq i \leq I, 
\ee
from which we infer a $L^2(Q_{i,T})$-estimate on $\kirsqr_i(p)$. Putting
\[
V = \big\{ u \in H^1(\O)\; | \; u_{|_{\Gamma^{\rm D}}} = 0\big\}, 
\]
the above estimates suggest the following notion of weak solution for our problem. 

\begin{deff}\label{def:weak}
A measurable function $p:Q_T \to \R$ is said to be a weak solution to the problem~\eqref{eq:cont.mass}--\eqref{eq:Ss} 
if 
\begin{subequations}
\begin{align}
\kirsqr_i(p) & \in L^2((0,T);H^1(\O_i)), \qquad \text{for } \; 1 \leq i \leq I,\\
\qquad \kirmin(p) - \kirmin(p^{\rm D}) & \in L^2((0,T);V)
\end{align}
and if for all $\fntest \in C^\infty_c([0,T)\times (\O\cup \Gamma^{\rm N}))$, there holds
\be\label{eq:weak.mass}
\iint_{Q_{T}} \phi \, \Ss(p,x) \p_t \fntest \, \d x \, \d t +  \int_{\O} 
\phi \, s^0 \fntest(\cdot, 0) \, \d x 
+ \iint_{Q_{T}} F \cdot \grad \fntest \, \d x \, \d t = 0,
\ee
with
\be\label{eq:weak.flux}
F = -  \grad \kircho_i(p) +  \lambda_i \eta_i(\Ss_i(p)) \, \varrho g \qquad \text{in }\; 
Q_{i,T}, \; 1 \leq i \leq I.
\ee
\end{subequations}
\end{deff}
The expression~\eqref{eq:weak.flux} is a reformulation of the original one~\eqref{eq:cont.flux}
in a quasilinear form which is suitable for analysis, even though the physical meaning
of the Kirchhoff transform $\kircho_i(p)$ is unclear. While the formulation~\eqref{eq:weak.mass} should be thought of as a weak form of~\eqref{eq:cont.mass}, \eqref{eq:cont.flux.interf},  \eqref{eq:cont.init}, and \eqref{eq:cont.Neumann}, the condition $\kirmin(p) - \kirmin(p^{\rm D}) \in L^2((0,T);V)$ contains~\eqref{eq:cont.p.interf} and~\eqref{eq:cont.Dirichlet}.

\subsection{Goal and positioning of the paper}\label{sse:goals}
We are now in a position to clearly state the two objectives of this paper.

The first objective is to put forward a rigorous proof that, for problem \eqref{eq:bigsyst} with heterogeneous data \eqref{eq:subdomdata}, cell-centered finite-volume schemes with upstream mobility such as described in \S\ref{ssec:scheme}, do converge towards a weak solution (in the sense of Definition \ref{def:weak}) as the discretization parameters tend to $0$. Such mathematically assessed convergence results are often dedicated to homogeneous cases: see for instance \cite{AWZ96, EGH99,RPK04} for schemes involving the Kirchhoff transforms for Richards' equation, \cite{Ahmed_M2AN} for a upstream mobility CVFE approximation of Richards' equation in anisotropic domains, \cite{CJ86, CE97, CE01} for schemes for two-phase flows involving the Kirchhoff transform, and \cite{EHM03,GRC_HAL} for upstream mobility schemes for two-phase porous media flows. For flows in highly heterogeneous porous media, rigorous mathematical results have been obtained for schemes involving the introduction of additional interface unknowns and Kirchhoff's transforms (see for instance~\cite{EEM06, NoDEA, FVbarriere, BCH13}), or under the non-physical assumption that the mobilities are strictly positive~\cite{EGHGH01, EGHM14_zamm}. It was established very recently in~\cite{BDMQ_HAL} that cell-centered finite-volumes with (hybrid) upwinding also converge for two-phase flows in heterogeneous domains, but with a specific treatment of the interfaces located at the heterogeneities. Here, the novelty lies in the fact that we do not consider any specific treatment of the interface in the design of the scheme. 

The second objective is of more practical nature. Even though our analysis still holds without any specific treatment of the interface, it is well-known that cell-centered upstream mobility finite-volumes can be inaccurate in the presence of heterogeneities. This observation motivated several contributions (see for instance \cite{EEN98, EEM06, HF08, EGHM14_zamm}) where skeletal (i.e., edge or vertex) unknowns where introduced in order to enforce the continuity of the pressures at the interfaces $\Gamma_{i,j}$. By means of extensive numerical simulations in \S\ref{sec:results}, we will show that without local refinement of the grid at the interface, the method still converges, but with a degraded order. Our ultimate motivation is to propose an approach which consists in adding very thin cells on both sides of the interface before using the cell centered scheme under study. Then the scheme appears to behave better, with first-order accuracy. Moreover, one can still make use of the parametrized cut-Newton method proposed in \cite{BCET_FVCA9} to compute the solution to the nonlinear system corresponding to the scheme. This method appears to be very efficient, while it avoids the possibly difficult construction of compatible parametrizations at the interfaces as in~\cite{BKJMP17, BMQ20, BDMQ_HAL}.

\section{Finite-volume discretization}
The scheme we consider in this paper is based on two-point flux approximation (TPFA) finite-volumes. Hence, it is subject to some restrictions on the mesh~\cite{Tipi, GK_Voronoi}. We first review the requirements on the mesh in \S\ref{ssec:mesh}. Next, we construct the upstream mobility finite-volume scheme for Richards' equation in \S\ref{ssec:scheme}. The main mathematical results of the paper, which are the well-posedness of the nonlinear system corresponding to the scheme and the convergence of the scheme, are then summarized in \S\ref{ssec:main}.

\subsection{Admissible discretization of $Q_T$}\label{ssec:mesh}
Let us start by discretizing w.r.t. space.  
\begin{deff}\label{def:meshAdmissible}
An \emph{admissible mesh of $\O$} is a triplet $(\Tt, \Ee, {(x_K)}_{K\in\Tt})$ such that the following conditions are fulfilled:
\begin{enumerate}[noitemsep]
\item[(i)] Each control volume (or cell) $K\in\Tt$ is non-empty, open, polyhedral and convex, with positive $d$-dimensional Lebesgue measure $m_K > 0$. We assume that 
\[
K \cap L = \emptyset \quad \text{if}\; K, L \in \Tt \; \text{with}\; K \neq L, 
\qquad \text{while}\quad \bigcup_{K\in\Tt}\ov K = \ov \O. 
\]
Moreover, we assume that the mesh is adapted to the heterogeneities of $\O$, in the sense that for all $K \in \Tt$, there exists $i\in \{1,\dots, I\}$ such that $K \subset \O_i$.
\item[(ii)] Each face $\sig \in \Ee$ is closed and is contained in a hyperplane of $\R^d$, with positive $(d-1)$-dimensional Hausdorff measure $\measure^{d-1}(\sig)= m_\sig  >0$. We assume that $\measure^{d-1}(\sig \cap \sig') = 0$ for $\sig, \sig' \in \Ee$ unless $\sig' = \sig$. For all $K \in \Tt$, we assume that there exists a subset $\Ee_K$ of $\Ee$ such that $\p K =  \bigcup_{\sig \in \Ee_K} \sig$. Moreover, we suppose that $\bigcup_{K\in\Tt} \Ee_K = \Ee$. Given two distinct control volumes $K,L\in\Tt$, the intersection $\ov K \cap \ov L$ either reduces to a single face $\sig  \in \Ee$ denoted by $K|L$, or its $(d-1)$-dimensional Hausdorff measure is $0$. 
\item[(iii)] The cell-centers $(x_K)_{K\in\Tt}$ are pairwise distinct with $x_K \in K$, and are such that, if $K, L \in \Tt$ share a face $K|L$, then the vector $x_L-x_K$ is orthogonal to $K|L$.
\item[(iv)] For the boundary faces $\sig \subset \p\O$, we assume that either $\sig \subset \Gamma^{\rm D}$ or $\sig \subset \ov \Gamma{}^{\rm N}$. For $\sig \subset \p\O$ with $\sig \in \Ee_K$ for some $K\in \Tt$, we assume additionally that there exists $x_\sig \in \sig$ such that $x_\sig - x_K$ is orthogonal to $\sig$.
\end{enumerate} 
\end{deff}

In our problem, the standard Definition \ref{def:meshAdmissible} must be supplemented by a compatibility property between the mesh and the subdomains. By ``compatbility'' we mean that each cell must lie entirely inside a single subregion. Put another way, 
\be\label{eq:indexK}
\forall K\in\Tt, \quad \exists ! \; i(K)\in \{ 1, \ldots, I\} \;\; | \;\; K\subset \O_{i(K)} .
\ee
This has two consequences. The first one is that, if we define 
\be
\Tt_i = \{K\in \Tt\; | \; K \subset \O_i\}, \qquad 1 \leq i \leq I, 
\ee
then $\Tt = \bigcup_{i=1}^I \Tt_i$. The second one is that the subdomain interfaces $\Gamma_{i,j}$ for $i\neq j$ coincide necessarily with some edges $\sig \in \Ee$. To express this more accurately, let $\Ee_\Gamma = \{\sig \in \Ee\; | \; \sig \subset \Gamma\}$ be the set of the interface edges,  $\Ee_{\rm ext}^{\rm D} =  \{\sig \in \Ee\; | \; \sig \subset \Gamma^{\rm D}\}$ be the set of Dirichlet boundary edges, and 
$\Ee_{\rm ext}^{\rm N} =  \{\sig \in \Ee\; | \; \sig \subset \ov \Gamma{}^{\rm N}\}$ be the set of Neumann boundary edges. Then,
$\Gamma = \bigcup_{\sig \in \Ee_\Gamma} \sig$, while 
$\Gamma^{\rm D}=  \bigcup_{\sig \in \Ee_{\rm ext}^{\rm D}} \sig$ and 
$\ov \Gamma{}^{\rm N}=  \bigcup_{\sig \in \Ee_{\rm ext}^{\rm N}} \sig$.
For later use, it is also convenient to introduce the subset $\Ee_i \subset\Ee$ consisting of those edges that correspond to cells in $\Tt_i$ only, i.e.,
\begin{subequations}
\be\label{eq:EdgesOfI}
\Ee_i = \bigg(\bigcup_{K\in\Tt_i} \Ee_K\bigg)\setminus \Ee_\Gamma, \qquad 1 \leq i \leq I, 
\ee
and the subset $\Ee_{\rm int}$ of the internal edges, i.e., 
\be
\Ee_{\rm int} = \Ee \setminus(\Ee_{\rm ext}^{\rm D} \cup \Ee_{\rm ext}^{\rm N}) = \bigcup_{K,L\in \Tt} \{ \sig = K|L\}.
\ee
\end{subequations}
Note that $\Ee_\Gamma \subset \Ee_{\rm int}$. 

To each edge $\sig\in\Ee$, we associate a distance $d_\sig$ by setting 
\be\label{eq:distance}
d_\sig = \begin{cases}
\, |x_K - x_L| & \text{if}\; \sig = K|L \in \Ee_{\rm int}, \\
\, |x_K - x_\sig| & \text{if}\; \sig \in \Ee_K \cap (\Ee_{\rm ext}^{\rm D} \cup \Ee_{\rm ext}^{\rm N}).
\end{cases}
\ee
We also define $d_{K\sig} = {\rm dist}(x_K,\sig)$ for all $K \in \Tt$ and $\sig \in \Ee_K$. The transmissivity of the edge $\sig\in \Ee$ is defined by
\be\label{eq:transmissive}
a_\sig = \frac{m_\sig}{d_\sig}.
\ee

Throughout the paper, many discrete quantities $\bu$ will be defined either in cells $K\in\Tt$ or on Dirichlet boundary edges $\sig \in \Ee_{\rm ext}^{\rm D}$, i.e. 
$\bu = ( \left(u_K\right)_{K\in\Tt}, \left(u_\sig\right)_{\sig \in \Ee_{\rm ext}^{\rm D}} ) \in \mathbb{X}^{\Tt\cup\Ee_{\rm ext}^{\rm D}}$,
where $\mathbb{X}$ can be either $\R^\ell$, $\ell \geq 1$, or a space of functions. 
Then for all $K\in\Tt$ and $\sig \in \Ee_K$, we define the mirror value $u_{K\sig}$ by 
\be\label{eq:mirror}
u_{K\sig} = \begin{cases}
\, u_L & \text{if}\; \sig = K|L \in \Ee_{\rm int}, \\
\, u_K &\text{if}\; \sig \in \Ee_K \cap \Ee_{\rm ext}^{\rm N}, \\ 
\, u_\sig &\text{if}\; \sig \in \Ee_K \cap \Ee_{\rm ext}^{\rm D}.
\end{cases}
\ee

The diamond cell $\Delta_\sig$ corresponding to the edge $\sig$ is defined as the convex hull of 
$\{x_K, x_{K\sig},\sig\}$ for $K$ such that $\sig \in \Ee_K$, while the half-diamond cell 
$\Delta_{K\sig}$ is defined as the convex hull of $\{x_K,\sig\}$. Denoting by $m_{\Delta_\sig}$ 
the Lebesgue measure of $\Delta_\sig$, the elementary geometrical relation 
$m_{\Delta_\sig} = d\, m_\sig d_\sig$ where $d$ stands for the dimension will be used many times in what follows.

Another notational shorthand is worth introducing now, since it will come in handy in the sequel. Let
\begin{subequations}\label{eq:fsubdom}
\be
f(\cdot,x) = \sum_{1\leq i\leq I} f_i(\cdot) \1_{\O_i}(x)
\ee
be a scalar quantity or a function whose dependence of $x\in\O$ is of the type \eqref{eq:subdomdata}. Then, for $K\in\Tt$, we slightly abuse the notations in writing
\be\label{eq:finK}
f_K(\cdot) := f(\cdot,x_K) = f_{i(K)}(\cdot),
\ee
\end{subequations}
where the index $i(K)$ is defined in \eqref{eq:indexK}. The last equality in the above equation holds by virtue of the compatibility property. For example, we will have not only
$\phi_K= \phi(x_K)$, $\lambda_K= \lambda(x_K)$, $\eta_K(s)= \eta(s,x_K)$, $\Ss_K(p) = \Ss(p,x_K)$ but also $\Energy_K(s)=\Energy(s,x_K)$. Likewise, we shall be writing $f_{K\sig}(\cdot) = f(\cdot,x_{K\sig})$ for the mirror cell without any ambiguity: if $\sig\in\Ee_{\rm int}\cup \Ee_{\rm ext}^{\rm N}$, then $x_{K\sig}$ is a cell-center;  if $\sig\in\Ee_{\rm ext}^{\rm D}$, then $x_{K\sig}$ lies on the boundary but does not belong to an interface between subdomains.

The size $\size_\Tt$ and the regularity $\regul_\Tt$ of the mesh are respectively defined by 
\be\label{eq:diamreg}
\size_\Tt = \max_{K\in\Tt} \, {\rm diam}(K), 
\qquad 
\regul_\Tt = \min_{K\in\Tt} \; {\bigg( \frac1{\operatorname{Card}\, \Ee_K}}\,  \min_{\sig \in \Ee_K} \frac{d_{K\sig}}{{\rm diam}(K)} {\bigg)}.
\ee

The time discretization is given by $\left(t^n\right)_{0\leq 1 \leq N}$ with 
$0 = t^0 < t^1 < \dots < t^N = T$. We denote by $\dtn = t^n - t^{n-1}$ for all $n \in \{1,\dots, N\}$ and 
by $\bdt = \left(\dtn\right)_{1\leq n \leq N}$.

\subsection{Upstream mobility TPFA Finite Volume scheme}\label{ssec:scheme}
Given a discrete saturation profile $(s_K^{n-1})_{K\in\Tt} \in [0,1]^\Tt$ at time $t^{n-1}$, $n\in \{1,\dots, N\}$, 
we seek for a discrete pressure profile $(p_K^{n})_{K\in\Tt} \in \R^\Tt$ at time $t^n$ solution to 
the following nonlinear system of equations. 
Taking advantage of the notational shorthand \eqref{eq:finK}, we define
\be\label{eq:scheme.capi}
s_K^n = \Ss_K(p_K^n), \qquad K\in\Tt, \; n\geq 1.
\ee 
The volume balance~\eqref{eq:cont.mass} is then discretized into
\be\label{eq:scheme.mass}
 m_K \phi_K \frac{s_K^n - s_K^{n-1}}{\dt^n} + \sum_{\sig \in \Ee_K} m_\sig F_{K\sig}^n = 0, \qquad K \in \Tt, \; n \geq 1,
\ee
using the approximation
\begin{subequations}
\be\label{eq:scheme.flux}
F_{K\sig}^n = \frac{1}{d_\sig}\lambda_\sig \eta_\sig^n (\phead_K^n  - \phead_{K\sig}^n ), \quad \sig \in \Ee_K, \; K \in \Tt, \; n\geq 1, 
\ee
for the flux \eqref{eq:darcymuskat}, with
\be
\phead_K^n = p_K^n + \psi_K, \qquad \phead_{K\sig}^n = p_{K\sig}^n + \psi_{K\sig} ,
\ee
where the mirror values $p_{K\sig}^n$ and $\psi_{K\sig}$ are given by \eqref{eq:mirror}.
In the numerical flux \eqref{eq:scheme.flux}, the edge permeabilities $(\lambda_\sig)_{\sig \in \Ee}$ are set to
\[
\lambda_\sig = \begin{cases}
\, \displaystyle\frac{\lambda_K \lambda_L d_{\sig}}{\lambda_K d_{L,\sigma}+ \lambda_L d_{K,\sigma}} & \text{if}\; \sig = K|L \in \Ee_{\rm int}, \\
\, \lambda_K & \text{if}\; \sig \in \Ee_{K} \cap \Ee_{\rm ext},  
 \end{cases}
\]
while the edge mobilities are upwinded according to
\be\label{eq:scheme.upwind}
\eta_\sig^n = \begin{cases}
\, \eta_K(s_K^n) & \text{if}\; \phead_K^n > \phead_{K\sig}^n , \\
\frac{1}{2}(\eta_K(s_K^n)+\eta_{K\sig}(s_{K\sig}^n)) &  \text{if}\; \phead_K^n = \phead_{K\sig}^n ,\\
\, \eta_{K\sig}(s_{K\sig}^n) & \text{if}\; \phead_K^n < \phead_{K\sig}^n.
\end{cases}
\ee
\end{subequations}
In practice, the definition of $\eta_\sig^n$ when $\phead_K^n = \phead_{K\sig}^n$ has no influence on the scheme. We choose here to give a symmetric definition that does not depend on the orientation of the edge $\sig$ in order to avoid ambiguities.

The boundary condition $p^{\rm D}$ is discretized into 
\be\label{eq:scheme.pD}
\begin{cases}
\, p_K^{\rm D} =  \frac{1}{m_K} \int_K p^{\rm D}(x) \, \d x & \text{for}\; K \in \Tt, \\
\, p_\sig^{\rm D} = \;\, \frac{1}{m_\sig} \int_\sig p^{\rm D}(x) \, \d \measure^{d-1}(x)&\text{for}\; \sig \in  \Ee_{\rm ext}^{\rm D},  
\end{cases}
\ee
whereas the initial condition is discretized into 
\be\label{eq:scheme.init}
s_K^0 = \frac1{m_K} \int_K s^0(x) \, \d x, \qquad \text{for}\; K \in \Tt. 
\ee
The Dirichlet boundary condition is encoded in the fluxes~\eqref{eq:scheme.flux} by setting 
\be\label{eq:p_Dir}
p_\sig^n = p^{\rm D}_\sig, \qquad \forall \sig \in \Ee_{\rm ext}^{\rm D}, \; n \geq 1.
\ee
Bearing in mind the definition~\eqref{eq:mirror} of the mirror values for $\sig \in \Ee_{\rm ext}^{\rm N}$, 
the no-flux boundary condition across $\sig\in\Ee_{\rm ext}^{\rm N}$ is automatically encoded, i.e., $F_{K\sig}^n = 0$ for all $\sig \in \Ee_{K}\cap \Ee_{\rm ext}^{\rm N}$, $K\in\Tt$ and $n\geq 1$.

In what follows, we denote by $\bp^n = \left(p_K^n\right)_{K\in\Tt}$ for $1 \leq n \leq N$, 
and  by $\bs^n = \left(s_K^n\right)_{K\in\Tt}$ for $0 \leq n \leq N$.
Besides, we set $\bp^{\rm D} = ((p_K^{\rm D})_{K\in\Tt}, (p_\sig^{\rm D})_{\sig\in\Ee^{\rm D}} )$.

\subsection{Main results and organization of the paper}\label{ssec:main}
The theoretical part of this paper includes two main results. The first one, which emerges from the analysis at fixed grid, states that the schemes admits a unique solution $(\bp^n)_{1 \leq n \leq N}$. 
\begin{thm}\label{thm:1}
For all $n\in \{1, \dots, N\}$, there exists a unique solution $\bp^n$ to the scheme~\eqref{eq:scheme.capi}--\eqref{eq:scheme.upwind}. 
\end{thm}
With Theorem~\ref{thm:1} at hand, we define the approximate pressure $p_{\Tt, \dt}$ by
\begin{subequations}
\be\label{eq:approx_sol}
p_{\Tt,\bdt}(t,x) = p_K^n \qquad \text{for}\; (t,x) \in (t^{n-1}, t^n] \times K. 
\ee
We also define the approximate saturation as
\be\label{eq:approx_sat}
s_{\Tt, \bdt} = \Ss(p_{\Tt, \bdt}, x) .
\ee 
\end{subequations}

The second main result guarantees the convergence towards a weak solution of the sequence of approximate solutions as the mesh size and the time steps tend to $0$. Let $(\Tt_m, \Ee_m, \left(x_K\right)_{K\in\Tt_m})_{m\geq 1}$ be 
a sequence of admissible discretizations of the domain $\O$ in the sense of Definition~\ref{def:meshAdmissible} such that 
\be\label{eq:mesh.m}
\size_{\Tt_m} \underset{m\to\infty}\longrightarrow 0, \qquad \sup_{m\geq 1} \; \regul_{\Tt_m}  =: \regul < +\infty, 
\ee
where the size $\size_{\Tt_m}$ and the regularity $\regul_{\Tt_m}$ are defined in \eqref{eq:diamreg}.
Let $(\bdt_m)_{m\geq 1}$ be time discretizations of $(0,T)$ such that 
\be\label{eq:dt.m}
\lim_{m\to \infty} \; \max_{1 \leq n \leq N_m} \dt^n_{m}  = 0.
\ee

\begin{thm}\label{thm:2}
There exists a weak solution $p: Q_T \to \R$ in the sense of Definition~\ref{def:weak} such that, up to a subsequence, 
\begin{subequations}
\begin{alignat}{2}
s_{\Tt_m, \bdt_m} & \underset{m\to\infty} \longrightarrow \Ss(p,x) & \qquad & \text{a.e. in}\; Q_T, 
\\
\kirmin(p_{\Tt_m, \bdt_m}) & \underset{m\to\infty} \longrightarrow \kirmin(p) & \qquad & \text{weakly in}\; L^2(Q_T). 
\end{alignat}
\end{subequations}
\end{thm}

The rest of this paper is outlined as follows. Section \S\ref{sec:fixed} is devoted to the numerical analysis at fixed grid. This encompasses the existence and uniqueness result stated in Theorem~\ref{thm:1} as well as a priori estimates that will help proving Theorem~\ref{thm:2}. The convergence of the scheme, which is taken up in \S\ref{sec:convergence}, 
relies on compactness arguments, which require a priori estimates that 
are uniform w.r.t. the grid. These estimates are mainly adaptations to the discrete setting of their continuous counterparts that arised in the stability analysis sketched out in \S\ref{ssec:cont.stability}. These estimates are shown in \S\ref{ssec:compact} 
to provide some compactness on the sequence of approximate solutions. In \S\ref{ssec:identify}, we show that these compactness properties together with the a priori estimates are sufficient to identify any limit of an approximate solution as a weak solution to the problem. 

In \S\ref{sec:numerics}, we provide some details about the practical numerical resolution by laying emphasis on the switch of variable for selecting the primary unknown and on the mesh refinement at an interface in order to better enforce pressure continuity. Finally, in \S\ref{sec:results}, numerical experiments on two configurations (drying and filling cases) for two capillary pressure models (Brooks-Corey and van Genuchten-Mualem) testify to the relevance of the local refinement strategy as a simple technique to preserve accuracy.

\begin{rmk}
Theorem~\ref{thm:2} only states the convergence of the scheme up to a subsequence. 
In the case where the weak solution is unique, then the whole sequence of approximate solutions 
would converge towards this solution. As far as we know, uniqueness of the weak solutions to Richards' equation 
is in general an open problem for heterogeneous media where $x \mapsto \Ss(p,x)$ is discontinuous. 
Uniqueness results are however available in the one-dimensional setting for a slightly more restrictive notion 
of solutions, cf. \cite{FVbarriere}, or under additional assumptions on the nonlinearities $\eta_i, \Ss_i$, cf. \cite{NoDEA}.
\end{rmk}

\section{Analysis at fixed grid}\label{sec:fixed}

\subsection{Some uniform a priori estimates}\label{ssec:apriori}
In this section, our aim is to derive a priori estimates on the solutions to the scheme~\eqref{eq:scheme.capi}--\eqref{eq:scheme.init}.
These estimates will be at the core of the existence proof of a solution to the scheme. 
They will also play a key role in proving the convergence of the scheme. 

The main estimate on which our analysis relies is a discrete counterpart of~\eqref{eq:NRG.1}. We recall that $a_\sig$ is the transmissivity introduced in \eqref{eq:transmissive}.

\begin{prop}\label{prop:dissip}
There exist two constants $\ctel{c:main}$, $\ctel{c:dissip.2}$ depending only on $\lambda$, $\mu$, $p^{\rm D}$, $\psi$, $\regul$, $\Omega$, $T$, $\phi$, and 
$\lVert\Ss_i\rVert_{L^1(\R_-)}$ such that 
\begin{subequations}\label{eq:dissconts}
\begin{align}
\sum_{n=1}^N \dtn \sum_{\sig \in \Ee} a_\sig \lambda_\sig \eta_\sig^n (p_K^n - p_{K\sig}^n)^2 & \leq \cter{c:main},  \label{eq:dissip}\\
\sum_{n=1}^N \dtn \sum_{\sig \in \Ee} a_\sig \lambda_\sig \eta_\sig^n (\phead_K^n - \phead_{K\sig}^n )^2 & \leq \cter{c:dissip.2}. \label{eq:Flux.L2}
\end{align}
\end{subequations}
\end{prop}
In \eqref{eq:dissconts}, the relationship between $\sig$ and $K$ is to be understood as follows. For an inner edge $\sig \in \Ee_{\rm int}$, although it can be written as $\sig = K|L$ or $L|K$, only one of these contributes to the sum. For a boundary edge $\sig \in \Ee_{\rm ext}$, there is only one cell $K$ such that $\sig\in\Ee_K$, so there is no ambiguity in the sum.
\begin{proof}
Multiplying \eqref{eq:scheme.mass} by $\Delta t^n (p_K^n -p_K^{\rm D})$, summing over $K \in \Tt$ and $n\in\{1,\dots, N\}$, and carrying out discrete integration by parts yield
\be\label{eq:dissip.A+B}
\Aa + \Bb =0,
\ee
where we have set
\begin{subequations}
\begin{align} 
\Aa & =  \sum_{n=1}^N \sum_{K\in \Tt} m_K \phi_K (s_K^n- s_K^{n-1})(p_K^n -p_K^{\rm D}), \label{eq:uni.estim.A}\\
\Bb & = \sum_{n=1}^N \Delta t^n \sum_{\sig\in \Ee}
a_\sigma \lambda_\sigma \eta_{\sigma}^n (\phead_K^n - \phead_{K\sigma}^n )
(p_K^n-p_K^{\rm D} -p_{K\sigma}^n + p_{K\sigma}^{\rm D}). \label{eq:uni.estim.B}
\end{align}
\end{subequations}
The discrete energy density function $\Energy_K: [0,1] \to \R_+$, defined by means of the notation \eqref{eq:fsubdom} from the functions $f_i=\Energy_i$ introduced in \eqref{eq:defenergy}, is convex by construction. Consequently,
\[
\Energy_K(s_K^{n-1}) - \Energy_K(s_K^n) \geq \Energy'_K(s_K^n) (s_K^{n-1} - s_K^n) = 
\phi_K (p_K^n -p_K^{\rm D}) (s_K^{n-1} - s_K^n) .
\]
Therefore, the quantity $\Aa$ of \eqref{eq:uni.estim.A} can be bounded below by 
\begin{multline}
\qquad\qquad \Aa \, \geq \, \sum_{n=1}^N \sum_{K\in \Tt} m_K  (\Energy_K(s_K^n)- \Energy_K(s_K^{n-1})) \\
  = \sum_{K\in \Tt} m_K (\Energy_K(s_K^N)- \Energy_K(s_K^{0})) \, \geq \, - C_{\Aa}, \qquad\qquad \label{eq:dissp.A}
\end{multline}
the last inequality being a consequence of the boundedness of $\Energy_K$ on $[0,1]$.

Writing $\phead = p + \psi$ and expanding each summand of \eqref{eq:uni.estim.B}, we can split $\Bb$ into
$$\Bb= \Bb_1+ \Bb_2+ \Bb_3,$$ 
with 
\begin{align*}
\Bb_1 & = \phantom{-} \sum_{n=1}^N  \Delta t^n  \sum_{\sigma\in \Ee} a_\sigma \lambda_\sigma \eta_\sigma^n (p_K^n - p_{K\sigma}^n)^2,\\
\Bb_2 & =  \phantom{-} \sum_{n=1}^N  \Delta t^n  \sum_{\sigma\in \Ee} a_\sigma \lambda_\sigma \eta_\sigma^n (p_K^n - p_{K\sigma}^n)(\psi_K - \psi_{K\sigma}- p_K^{\rm D} + p_{K\sigma}^{\rm D} ),\\
\Bb_3 & = -\sum_{n=1}^N  \Delta t^n  \sum_{\sigma\in \Ee} a_\sigma \lambda_\sigma \eta_\sigma^n (\psi_K - \psi_{K\sig})
(p_K^{\rm D} - p_{K\sigma}^{\rm D}).
\end{align*}
It follows from \cite[Lemma 9.4]{eymardEtAl2000finiteVolumeMethod}
and from the boundedness of $\eta$ that there exists a constant $C$ depending only on $\lambda$, $\mu$, $\regul_{\Tt}$ and $\O$ such that
\begin{subequations}
\begin{align}
\sum_{\sigma\in \Ee} a_\sigma \lambda_\sigma \eta_\sigma^n (p_K^{\rm D} - p_{K\sigma}^{\rm D})^2 & \leq C  \,
\lVert\nabla p^{\rm D}\rVert^2_{L^2(\Omega)^d}, \\
\sum_{\sigma\in \Ee} a_\sigma \lambda_\sigma \eta_\sigma^n (\psi_K - \psi_{K\sigma})^2 
& \leq C \, \lVert\nabla \psi\rVert^2_{L^\infty(\Omega)^d} .
\end{align}
\end{subequations}
Thanks to these estimates and to the Cauchy-Schwarz inequality, we have
\[
\Bb_3\geq -C T \,\lVert\nabla p^{\rm D}\rVert_{L^2(\Omega)^d} \, \lVert\nabla \psi\rVert_{L^\infty(\Omega)^d} .
\]
On the other hand, Young's inequality provides
$$\Bb_2\geq -\frac{1}{2} \Bb_1 - C T \, \big( \lVert\nabla p^{\rm D}\rVert^2_{L^2(\Omega)^d} + \lVert\nabla \psi\rVert^2_{L^\infty(\Omega)^d} \big).$$
Hence, 
\be\label{eq:dissip.B}
\Bb \, \geq \,  \frac{1}{2}\sum_{n=1}^N  \Delta t^n  \sum_{\sigma\in \Ee} a_\sigma  
\lambda_\sigma \eta_\sigma^n (p_K^n - p_{K\sigma}^n)^2 - C_{\Bb}, 
\ee
by setting $C_{\Bb} = C T \,(\lVert\nabla p^{\rm D}\rVert^2_{L^2(\Omega)^d} + \lVert\nabla \psi\rVert^2_{L^\infty(\Omega)^d} + \lVert\nabla p^{\rm D}\rVert_{L^2(\Omega)^d} \, \lVert\nabla \psi\rVert_{L^\infty(\Omega)^d} )$.
Inserting \eqref{eq:dissp.A} and \eqref{eq:dissip.B} into \eqref{eq:dissip.A+B}, we recover \eqref{eq:dissip} with $\cter{c:main} = 2(C_{\Aa} + C_{\Bb})$.
From \eqref{eq:dissip}, we can deduce \eqref{eq:Flux.L2} by elementary manipulations.
\end{proof}

So far, we have not used the upwind choice~\eqref{eq:scheme.upwind} for the mobilities $\eta_\sig^n$. This will be done in the next lemma, where we derive a more useful variant of estimate \eqref{eq:dissip}, in which $\eta_\sig^n$ is replaced by $\ov{\eta}_\sig^n$ defined below. In a homogeneous medium, $\ov{\eta}_\sig^n \geq \eta_\sig^n$ so that the new estimate \eqref{eq:dissip.max} seems to be stronger than \eqref{eq:dissip}. 

We begin by introducing the functions $\widecheck{\eta}_\sig: \R \to (0, 1/\mu]$ defined for $\sig \in \Ee$ by 
\begin{subequations}
\be
\widecheck{\eta}_\sig(p) = 
\min\big\{\eta_K \circ \Ss_K(p), \eta_{K\sig} \circ \Ss_{K\sig}(p) \big\},
\qquad \forall p \in \R.
\ee
By virtue of assumptions \eqref{eq:EtaSs}, each argument of the minimum function is nondecreasing and positive function of $p\in\R$. As a result, $\widecheck{\eta}_\sig$ is also a nondecreasing and positive function of $p\in\R$. Note that $\widecheck{\eta}_\sig = \eta_i \circ \Ss_i$ for all $\sig \in \Ee_i$, while for interface edges $\sigma \subset \Gamma_{i,j}$, the mere inequality $\widecheck{\eta}_\sig \leq \eta_i \circ \Ss_i$ holds. Next, we consider the intervals
\be
\Jj_{\sig}^n = [  p_K^n \bot p_{K\sig}^n , \, p_K^n \top p_{K\sig}^n], \qquad \text{for } \, \sig \in \Ee_K, \; K \in \Tt, \; 1 \leq n \leq N,
\ee
with the notations $a\bot b = \min(a,b)$ and $a\top b = \max(a,b)$. At last, we set
\be\label{eq:ov_eta_sig}
\ov{\eta}_\sig^n = \max_{p \in 
\Jj_\sig^n} \widecheck{\eta}_\sig(p), \qquad \text{for } \, \sig \in \Ee, \; 1 \leq n \leq N. 
\ee
\end{subequations}
\begin{lem}\label{lem:dissip.max}
There exists a constant $\ctel{c:dissip.max}$ depending on the same data as $\cter{c:main}$ such that 
\be\label{eq:dissip.max}
\sum_{n=1}^N \dt^n \sum_{\sig \in\Ee} a_\sig \ov\eta_\sig^n \left(p_K^n - p_{K\sig}^n\right)^2 \leq \cter{c:dissip.max} .
\ee
\end{lem}
\begin{proof}
We partition the set $\Ee$ of edges into three subsets, namely,
\[
\Ee^{n}_{+} = \big\{\sig \, |\; \phead_K^n > \phead_{K\sig}^n \big\},\quad
\Ee^{n}_{-} = \big\{\sig \, |\; \phead_K^n < \phead_{K\sig}^n \big\},\quad
\Ee^{n}_{0} = \big\{\sig \, |\; \phead_K^n = \phead_{K\sig}^n \big\}.
\]
Invoking $\widecheck{\eta}_\sig = \min(\eta_K \circ \Ss_K, \,\eta_{K\sig}\circ \Ss_{K\sig})$, we can minorize the left-hand side of~\eqref{eq:dissip} to obtain
\begin{multline*}
\sum_{n = 1}^N \dt^n \Big[ 
\sum_{\sig \in \Ee^{n}_{+}} a_\sig \lambda_\sig \widecheck{\eta}_\sig(p_K^n) (p_K^n - p_{K\sig}^n)^2  
+ \sum_{\sig \in \Ee^{n}_{-}} a_\sig \lambda_\sig \widecheck{\eta}_\sig(p_{K\sig}^n) (p_K^n - p_{K\sig}^n)^2 \\
+ \sum_{\sig \in \Ee^{n}_{0}} a_\sig \lambda_\sig \textstyle\frac{1}{2} 
(\widecheck{\eta}_\sig(p_K^n) + \widecheck{\eta}_\sig(p_{K\sig}^n)) (p_K^n - p_{K\sig}^n)^2 
\Big] \leq \cter{c:main}.
\end{multline*}
Starting from this inequality and using the boundedness of $\eta_i$ and $\psi$, we can readily show that there exists a constant $C$ depending on the same data as $\cter{c:main}$
such that 
\begin{multline*}
\Dd_1 := \sum_{n = 1}^N \dt^n \Big[ 
\sum_{\sig \in \Ee^{n}_{+}} a_\sig \lambda_\sig \widecheck{\eta}_\sig(p_K^n) (p_K^n - p_{K\sig}^n) (\phead_K^n - \phead_{K\sig}^n)  \\
+ \sum_{\sig \in \Ee^{n}_{-}} a_\sig \lambda_\sig \widecheck{\eta}_\sig(p_{K\sig}^n) (p_K^n - p_{K\sig}^n) (\phead_K^n - \phead_{K\sig}^n )\Big] 
\leq C,
\end{multline*}
in which the sum over $\Ee_0^n$ was omitted because all of its summands vanish.
Simlarly to what was pointed out in equation 2.9 in~\cite{Ahmed_M2AN}, we notice that since $\eta_\sigma$ is nondecreasing w.r.t. $p$, it is straightforward to check that the definition
\be \label{eq:laurel}
\widecheck{\eta}_\sig^n := \begin{cases}
\; \widecheck{\eta}_\sig(p_K^n) & \text{ if } \phead^n_K > \phead^n_{K\sig},\\
\; \textstyle\frac{1}{2} (\widecheck{\eta}_\sig(p_K^n) + \widecheck{\eta}_\sig(p_{K\sig}^n))
& \text{ if } \phead^n_K = \phead^n_{K\sig},\\ 
\; \widecheck{\eta}_\sig(p_{K\sig}^n) & \text{ if } \phead^n_K < \phead^n_{K\sig}
\end{cases}
\ee 
exactly amounts to
\be \label{eq:hardy}
\widecheck{\eta}_\sig^n = \begin{cases}
\; \max_{p\in \Jj_\sig^n}\widecheck{\eta}_\sig(p) & \text{ if } (p^n_K - p^n_{K\sig})(\phead^n_K - \phead^n_{K\sig}) > 0,\\
\; \textstyle\frac{1}{2} (\widecheck{\eta}_\sig(p_K^n) + \widecheck{\eta}_\sig(p_{K\sig}^n))
& \text{ if } (p^n_K - p^n_{K\sig})(\phead^n_K - \phead^n_{K\sig}) = 0,\\ 
\; \min_{p\in \Jj_\sig^n}\widecheck{\eta}_\sig(p) & \text{ if } (p^n_K - p^n_{K\sig})(\phead^n_K - \phead^n_{K\sig}) < 0.
\end{cases}
\ee 
Taking advantage of this equivalence, we can transform $\Dd_1$ into
\begin{multline}
\Dd_1 = \sum_{n = 1}^N \dt^n \Big[ 
\sum_{\sig \in \Ee^{n}_{>}} a_\sig \lambda_\sig \max_{\Jj_\sig^n}\widecheck{\eta}_\sig (p_K^n - p_{K\sig}^n) (\phead_K^n - \phead_{K\sig}^n)  \\
+ \sum_{\sig \in \Ee^{n}_{<}} a_\sig \lambda_\sig \min_{\Jj_\sig^n}\widecheck{\eta}_\sig (p_K^n - p_{K\sig}^n) (\phead_K^n - \phead_{K\sig}^n )
\Big] \leq C ,
\end{multline}
where $\Ee^n_{>} = \{ \sigma \,|\, (p^n_K - p^n_{K\sig})(\phead^n_K - \phead^n_{K\sig}) > 0 \}$ and  $\Ee^n_{<} = \{ \sigma \,|\, (p^n_K - p^n_{K\sig})(\phead^n_K - \phead^n_{K\sig}) < 0 \}$. The second sum over $\Ee^n_{<}$ contains only negative summands and can be further minorized if $\min_{\Jj^n_\sig}\widecheck{\eta}_\sig$ is replaced by $\max_{\Jj^n_\sig}\widecheck{\eta}_\sig$. In other words,
\[
\Dd_2 := \sum_{n = 1}^N \dt^n  \sum_{\sig\in\Ee} a_\sig \lambda_\sig  \ov \eta_\sig^n (p_K^n - p_{K\sig}^n)
(\phead_K^n - \phead_{K\sig}^n ) \leq \Dd_1 \leq C. 
\]
Writing $\phead = p+\psi$, expanding each summand of $\Dd_2$ and applying Young's inequality, we end up with
\[
\frac{1}{2}\sum_{n = 1}^N \dt^n  \sum_{\sig\in\Ee} a_\sig \lambda_\sig  \ov \eta_\sig^n
\, [ (p_K^n - p_{K\sig}^n)^2 - (\psi_K^n - \psi_{K\sig}^n)^2 ]
\leq \Dd_2 \leq C .
\]
Estimate \eqref{eq:dissip.max} finally follows from
the boundedness of $\eta$, $1/\lambda$ and $\psi$. 
\end{proof}

The above lemma has several important consequences for the analysis. 
Let us start with discrete counterparts to estimations~\eqref{eq:NRG.2} and \eqref{eq:NRG.Upsilon}.

\begin{cor}\label{cor:xi_Upsilon}
Let $\cter{c:dissip.max}$ be the constant in Lemma~\ref{lem:dissip.max}. Then,
\begin{subequations}\label{eq:dissip.xiUpsilon}
\begin{align}
\sum_{n=1}^N \dtn \sum_{i=1}^I \sum_{\sig \in \Ee_i} a_\sig (\kirsqr_i(p_K^n) - \kirsqr_i(p_{K\sigma}^n))^2 
& \leq \cter{c:dissip.max}, \label{eq:dissip.xi}\\
\sum_{n=1}^N \dtn \sum_{\sig \in \Ee} a_\sig (\kirmin(p_K^n) - \kirmin(p_{K\sigma}^n))^2 
& \leq \cter{c:dissip.max}. \label{eq:dissip.Upsilon}
\end{align}
\end{subequations}
Moreover, there exists two constants $\ctel{c:L2.Upsilon}$,  $\ctel{c:L2.xi}$ 
depending on the same data as $\cter{c:main}$ and additionnally on  
$\lVert\sqrt{\eta_i \circ \Ss_i}\rVert_{L^1(\R_-)}$, $1\leq i\leq I$, such that 
\begin{subequations}\label{eq:L2.Upsilonxi}
\begin{align}
\sum_{n=1}^N \dt^n \sum_{K\in\Tt} m_K |\kirmin(p_K^n)|^2  \leq \cter{c:L2.Upsilon},\label{eq:L2.Upsilon} \\
\sum_{n=1}^N \dt^n \sum_{i=1}^I \sum_{K\in\Tt_i} m_K |\kirsqr_i(p_K^n)|^2  \leq \cter{c:L2.xi}.\label{eq:L2.xi}
\end{align}
\end{subequations}
\end{cor}
\begin{proof}
Consider those edges $\sig \in \Ee_i$ ---defined in \eqref{eq:EdgesOfI}--- corresponding to some fixed $i\in\{1,\ldots,I\}$, for which $\widecheck{\eta}_\sig = \eta_i \circ \Ss_i = |\kirsqr'_i|^2$ and $\ov{\eta}^n_\sig = \max_{\Jj^n_\sig} |\kirsqr'_i|^2$ due to \eqref{eq:xi.def}. By summing the elementary inequality
\[
(\kirsqr_i(p_K^n) - \kirsqr_i(p_{K\sig}^n))^2 \leq  \ov{\eta}^n_{\sig} \; 
(p_K^n - p_{K\sig}^n)^2,  
\]
over $\sig\in\Ee_i$, $i\in\{1,\ldots,I\}$ and $n\in\{1,\ldots,N\}$ using appropriate weights, we get
\[
\sum_{n=1}^N \dtn \sum_{i=1}^I \sum_{\sig \in \Ee_i} a_\sig (\kirsqr_i(p_K^n) - \kirsqr_i(p_{K\sigma}^n))^2
\leq 
\sum_{n=1}^N \dtn \sum_{i=1}^I \sum_{\sig \in \Ee_i} a_\sig \ov{\eta}^n_{\sig} \; 
(p_K^n - p_{K\sig}^n)^2 ,
\]
whose right-hand side is obviously less than $\cter{c:dissip.max}$, thanks to~\eqref{eq:dissip.max}. This proves \eqref{eq:dissip.xi}.

Similarly, the respective definitions of $\ov{\eta}_\sig^n$ and $\kirmin$ have been tailored so that $\max_{\Jj^n_\sig} |\kirmin'|^2 \leq \ov{\eta}^n_\sig$ for all $\sig\in\Ee$. As a consequence,
\[
(\kirmin(p_K^n) - \kirmin(p_{K\sig}^n))^2
\leq 
\ov{\eta}_\sig^n (p_K^n  - p_{K\sig}^n )^2.
\]
Summing these inequalities over $\sig\in\Ee$ and $n\in \{1,\ldots,N\}$ with appropriate weights and invoking \eqref{eq:dissip.max}, we prove~\eqref{eq:dissip.Upsilon}.
 
The argument for \eqref{eq:L2.Upsilon} is subtler. Starting from the basic inequality
\begin{multline*}
\qquad (\kirmin(p_K^n)  - \kirmin(p_K^{\rm D})- \kirmin(p_{K\sig}^n) 
+ \kirmin(p_{K\sig}^{\rm D}))^2\\
\leq 
2 (\kirmin(p_K^n) - \kirmin(p_{K\sig}^n))^2 + 2 (\kirmin(p_K^{\rm D}) - \kirmin(p_{K\sig}^{\rm D}))^2 , \qquad
\end{multline*}
we apply the discrete Poincar\'e inequality of~\cite[Lemma 9.1]{eymardEtAl2000finiteVolumeMethod} ---which is legitimate since $\Gamma^{\rm D}$ has positive measure--- followed by \cite[Lemma 9.4]{eymardEtAl2000finiteVolumeMethod} to obtain
\[
\sum_{n=1}^N \dt^n \sum_{K\in\Tt} m_K ( \kirmin(p_K^n)  - \kirmin(p_K^{\rm D}) )^2 \\
\leq 2 C_{{\rm P},\Tt} \big( \cter{c:dissip.max} + C_\regul T \lVert \kirmin'\rVert_\infty \lVert \grad p^{\rm D} \rVert^2 \big) ,
\]
where $C_{{\rm P},\Tt}$ denotes the discrete Poincar\'e constant, and $C_\regul$ is the quantity appearing in \cite[Lemma 9.4]{eymardEtAl2000finiteVolumeMethod} and only depends on $\regul_{\Tt}$. $\cter{c:L2.Upsilon} = 4 C_{{\rm P},\Tt} \big( \cter{c:dissip.max} + C_\regul T \lVert \kirmin'\rVert_\infty \lVert \grad p^{\rm D} \rVert^2 \big) 
+ 2 m_\O T \lVert\kirmin(p^{\rm D})\rVert_\infty^2.$

The last estimate \eqref{eq:L2.xi} results from the comparison~\eqref{eq:xi_Upsilon} of the nonlinearities $\kirsqr_i$ and $\kirmin$.
 \end{proof}

The purpose of the next lemma is to work out a weak estimate on the discrete counterpart of $\p_t s$, which will lead to compactness properties in \S\ref{ssec:compact}. For $\fntest \in C^\infty_c( Q_T)$, let 
\[
\fntest_K^n = \frac1{m_K} \int_K \fntest(t^n,x) \, \d x, \qquad \forall K \in \Tt, \; 1 \leq n \leq N. 
\]
\begin{lem}\label{lem:boundS}
There exists a constant $\ctel{c:s2}$ depending on the same data as $\cter{c:main}$ such that
\be\label{boundS}
\sum_{n=1}^N \sum_{K\in\Tt} m_K \phi_K (s_K^n -s_K^{n-1}) \fntest_K^n  \leq \cter{c:s2} \lVert \nabla \fntest \rVert_{L^\infty(Q_T)^d}, \quad \forall\fntest\in C^\infty_c( Q_T).
\ee
\end{lem}
\begin{proof}
Multiplying \eqref{eq:scheme.mass} by $\dtn ~\fntest_K^n$, summing over $K\in\Tt$ and $n\in\{1, \cdots, N\}$ and carrying out discrete integration by parts, we end up with  
\[
\Aa:= \! \sum_{n=1}^N \! \sum_{K\in\Tt} m_K \phi_K (s_K^n -s_K^{n-1})\fntest_K^n 
= - \! \sum_{n=1}^N \dtn \! \sum_{\sig\in\Ee} a_\sigma \lambda_\sigma \eta_\sigma^n 
(\phead_K^n - \phead_{K\sig}^n)(\fntest_K^n-\fntest_{K\sig}^n).
\]
Applying the Cauchy-Schwarz inequality and using~\eqref{eq:Flux.L2}, we get
\be
\Aa^2 \leq \cter{c:dissip.2} \frac{\max_i \lambda_i}\mu 
 \sum_n \dtn \sum_{\sig\in\Ee} a_\sigma (\fntest_K^n-\fntest_{K\sig}^n)^2.
\ee
The conclusion \eqref{boundS} is then reached by means of the property (see \cite[Section 4.4]{ACM17})
\[
\sum_{n=1}^N \dtn \sum_{\sig\in\Ee} a_\sigma (\fntest_K^n-\fntest_{K\sig}^n)^2 \leq C \lVert \grad \fntest \rVert_{L^\infty(Q_T)^d}^2
\]
for some $C$ depending only on $\O$, $T$ and the mesh regularity $\regul_{\Tt}$.
\end{proof}

\subsection{Existence of a solution to the scheme}
The statements of the previous section are all uniform w.r.t. the mesh and are meant to help us passing to the limit in the next section. In contrast, the next lemma provides a bound on the pressure that depends on the mesh size and on the time-step. This property is needed in the process of ensuring the existence of a solution to the numerical scheme. 

\begin{lem}\label{lem:Boundedness}
There exist two constants $\ctel{c:pb}$, $\ctel{c:pb2}$ depending on $\Tt$, $\dt^n$ as well as on the data of the continuous model 
$\lambda$, $\mu$, $p^{\rm D}$, $\psi$, $\regul$, $\Omega$, $T$, $\phi$,
$\lVert \Ss_i \rVert_{L^1(\R_-)}$ and $\lVert \sqrt{\eta_i \circ \Ss_i} \rVert_{L^1(\R_-)}$, $1\leq i\leq I$, such that 
\be\label{eq:Boundedness}
-\cter{c:pb} \leq p_K^n \leq \cter{c:pb2},
\quad \forall K \in \Tt,
\; n \in \{1, \dots, N\}.
\ee
\end{lem}
\begin{proof}
From~\eqref{eq:L2.Upsilon} and from $\kirmin(p) = p \sqrt{\min_i{\lambda_i}/\mu}$ for $p\geq 0$, we deduce that
\[
p_K^n \leq \sqrt{\frac{\mu \cter{c:L2.Upsilon}}{\dt^n m_K \min_i \lambda_i}}, 
\qquad \forall K\in\Tt, \; 1 \leq n \leq N .
\]
Hence, the upper-bound $\cter{c:pb2}$ is found by maximizing the right-hand side over $K \in \Tt$ and $n \in \{1,\dots, N\}$.

To show that $p_K^n$ is bounded from below, we employ a strategy that was developed in~\cite{cancesGuichard2016EntropyScheme} and extended to the case of Richards' equation in~\cite[Lemma 3.10]{Ahmed_M2AN}. From~\eqref{eq:scheme.pD}, \eqref{eq:p_Dir} and the boundedness of $p^{\rm D}$, it is easy to see that
\[
p_\sig^n \geq \inf_{x \in \p\O} p^{\rm D}(x), \quad \forall \sig \in \Ee_{\rm ext}^{\rm D}.
\]
Estimate~\eqref{eq:dissip.max} then shows that for all $K\in\Tt$ such that 
$\Ee_K \cap  \Ee_{\rm ext}^{\rm D} \neq \emptyset$, we have
\[
p_K^n \geq p_\sig^n - \sqrt{\frac{\cter{c:dissip.max}}{\dt^n a_\sig \widecheck{\eta}_\sig (p_\sig^n)}}=: \pi_K^n, \quad 
\forall \sig \in \Ee_K \cap  \Ee_{\rm ext}^{\rm D}.
\]
The quantity $\pi_K^n$ is well-defined, since $\widecheck{\eta}_\sig(p^n_\sig) > 0$ for $p^n_\sig > -\infty$, and does not depend on time, as $p^{\rm D}$ does not either. Furthermore, if $p_K^n$ is bounded from below by some $\pi_K$, then the pressure in 
all its neighboring cells $L\in\Tt$ such that $\sig = K|L \in \Ee_K$ is bounded from below by 
\[
p_L^n \geq \pi_K^n -  \sqrt{\frac{\cter{c:dissip.max}}{\dt^n a_\sig \widecheck{\eta}_\sig (\pi_K^n)}} =: \pi_L^n. 
\]
Again, $\pi_L^n$ is well-defined owing to $\widecheck{\eta}_\sig(\pi_K^n) > 0$. Since the mesh is finite and since the domain is connected, only a finite number of edge-crossings is required to create a path from a Dirichlet boundary edge $\sig \in \Ee_{\rm ext}^{\rm D}$ to any prescribed cell $K \in \Tt$. Hence, the lower bound $\cter{c:pb}$ is found by minimizing $\pi^n_K$ over $K\in\Tt$ and $n\in\{1, \ldots,N\}$. 
\end{proof}

Lemma~\ref{lem:Boundedness} is a crucial step in the proof of the existence of a solution 
$\bp^n = (p_K^n)_{K\in\Tt}$ to the scheme~\eqref{eq:scheme.capi}--\eqref{eq:p_Dir}.
\begin{prop}\label{prop:existence}
Given $\bs^{n-1} = (s_K^{n-1})_{K\in\Tt} \in [0,1]^\Tt$, there exists a solution $\bp^n \in \R^\Tt$ to the scheme~\eqref{eq:scheme.capi}--\eqref{eq:p_Dir}.
\end{prop}
The proof relies on a standard topological degree argument and is omitted here. However, we make the homotopy explicit for readers' convenience. Let $\gamma \in [0,1]$ be the homotopy parameter. We define the nondecreasing functions $\eta_i^{(\gamma)}:[0,1] \to \R_+$ by setting $\eta_i^{(\gamma)}(s) = (1-\gamma)/\mu + \gamma \eta_i(s)$ for $s \in [0,1]$, and we seek a solution $\bp^{(\gamma)} = {(p_K^{(\gamma)})}_{K\in\Tt}$ to the problem
\begin{subequations}\label{eq:homotopy}
\be
\gamma  m_K \phi_K \frac{\Ss_K(p_K^{(\gamma)}) - s_K^{n-1}}{\dt_n} + 
\sum_{\sig \in \Ee_K} m_\sig F_{K\sig}^{(\gamma)} = 0, \qquad K \in \Tt, \; \gamma \in [0,1], 
\ee
where the fluxes $F_{K\sig}^{(\gamma)}$ are defined by
\be
F_{K\sig}^{(\gamma)}= \frac{1}{d_\sig}\lambda_\sig \eta_\sig^{(\gamma)} 
\big( \phead_K^{(\gamma)} - \phead_{K\sig}^{(\gamma)}  \big), \qquad \sig \in \Ee_K, \; K \in \Tt, \;\gamma \in [0,1]
\ee
with $\phead^{(\gamma)} = p^{(\gamma)} + \psi$ and using the upwind mobilities 
\be
\eta_\sig^{(\gamma)} = \begin{cases}
\, \eta_K^{(\gamma)}(\Ss_K(p_K^{(\gamma)})) & \text{if }\; \phead_K^{(\gamma)} > \phead_{K\sig}^{(\gamma)} , \\
\, \frac{1}{2}(\eta_K^{(\gamma)}(\Ss_K(p_K^{(\gamma)})) + \eta_{K\sig}^{(\gamma)}(\Ss_{K\sig}(p_K^{(\gamma)})))  &  
\text{if }\; \phead_K^{(\gamma)} = \phead_{K\sig}^{(\gamma)} ,\\
\, \eta_{K\sig}^{(\gamma)}(\Ss_{K\sig}(p_K^{(\gamma)})) & \text{if }\; \phead_K^{(\gamma)} < \phead_{K\sig}^{(\gamma)} .
\end{cases}
\ee
\end{subequations}
At the Dirichlet boundary edges, we still set $p_\sig^{(\gamma)} = p_\sig^{\rm D}$. 
For $\gamma = 0$, the system is linear and invertible, 
while for $\gamma = 1$, system \eqref{eq:homotopy} coincides with the original system \eqref{eq:scheme.capi}--\eqref{eq:p_Dir}. A priori estimates on $\bp^{(\gamma)}$ that are uniform w.r.t. $\gamma \in [0,1]$ (but not uniform w.r.t. $\Tt$ nor $\dt^n$) can be derived on the basis of what was exposed previously, so that one can unfold Leray-Schauder's machinery~\cite{LS34, Dei85} to prove the existence of (at least) one solution to the scheme.

\subsection{Uniqueness of the discrete solution}\label{ssec:uniqueness}
To complete the proof of Theorem~\ref{thm:1}, it remains to show that the solution to the scheme is unique. This is the purpose of the following proposition.
\begin{prop}\label{prop:uniqueness}
Given $\bs^{n-1} = (s_K^{n-1})_{K\in\Tt} \in [0,1]^\Tt$, the solution $\bp^n \in \R^\Tt$ to the scheme~\eqref{eq:scheme.capi}--\eqref{eq:p_Dir} is unique.
\end{prop}
\begin{proof}
The proof heavily rests upon the monotonicity properties inherited from the upwind choice \eqref{eq:scheme.upwind} for the mobilities.
Indeed, due to the upwind choice of the mobility, the flux $F_{K\sig}^n$ is a function of $p_K^n$ and $p_{K\sig}^n$ that is nondecreasing w.r.t. $p_K^n$ and nonincreasing w.r.t. $p_{K\sig}^n$. Moreover, by virtue of the monotonicity of $\Ss_K$, the discrete volume balance~\eqref{eq:scheme.mass} can be cast under the abstract form
\be\label{eq:Hh}
\scheme_K^n(p_K^n, (p_{K\sig}^n)_{\sig \in \Ee_K}) = 0, \qquad \forall K\in \Tt, 
\ee
where $\scheme_K^n$ is nondecreasing w.r.t its first argument $p_K^n$ and nonincreasing w.r.t each of the remaining variables $(p_{K\sig}^n)_{\sig \in \Ee_K}$.

Let $\wt{\bp}^n = \left(\wt p_K^n\right)_{K\in\Tt}$ be another solution to the system  \eqref{eq:scheme.capi}--\eqref{eq:p_Dir}, 
i.e., 
\be\label{eq:Hh.wt}
\scheme_K^n (\wt p_K^n, (\wt p_{K\sig}^n)_{\sig \in \Ee_K} ) = 0, \qquad \forall K\in \Tt . 
\ee
The nonincreasing behavior of $\scheme_K^n$ w.r.t. all its variables except the first one implies that 
\[
\scheme_K^n (p_K^n, (p_{K\sig}^n \top \wt{p}_{K\sig}^n)_{\sig \in \Ee_K} ) \leq  0, \qquad 
\scheme_K^n (\wt{p}_K^n, (p_{K\sig}^n \top \wt{p}_{K\sig}^n)_{\sig \in \Ee_K} ) \leq  0,
\]
for all $K\in \Tt$, where $a\top b = \max(a,b)$.
Since $p_K^n\top \wt p_K^n$ is either equal to $p_K^n$ or to $\wt p_K^n$, we infer from the above inequalities that 
\be\label{eq:Hh.top}
\scheme_K^n (p_K^n\top \wt p_K^n, (p_{K\sig}^n \top \wt p_{K\sig}^n)_{\sig \in \Ee_K} ) \leq  0, 
\qquad \forall K \in \Tt. 
\ee
By a similar argument, we can show that 
\be\label{eq:Hh.bot}
\scheme_K^n (p_K^n\bot \wt p_K^n, (p_{K\sig}^n \bot \wt p_{K\sig}^n)_{\sig \in \Ee_K}) \geq  0, 
\qquad \forall K \in \Tt, 
\ee
where $a\bot b = \min(a, b)$. Subtracting~\eqref{eq:Hh.bot} from \eqref{eq:Hh.top} and summing over $K\in\Tt$, we find
\be\label{eq:uniqueness.1}
\sum_{K\in\Tt} m_K \phi_K  \frac{|s_K^n - \wt s_K^n|}{\Delta t^n} 
+ \sum_{\sig \in \Ee_{\rm ext}^{\rm D}}  a_\sig \lambda_\sig \Rr_\sig^n 
\leq 0, 
\ee
where $s_K^n = \Ss_K(p_K^n)$, $\wt{s}_K^n = \Ss_K(\wt{p}_K^n)$ and
\begin{align}
\Rr_\sig^n & = \eta_K(s_K^n \top \wt{s}_K^n) (\phead_K^n \top \wt{\phead}_K^n - \phead_\sig^n )^+ - \eta_K(s_\sig^n)(\phead_\sig^n - \phead_K^n\top \wt \phead_K^n )^+ \nonumber\\
& -  \eta_K(s_K^n \bot \wt{s}_K^n) (\phead_K^n \bot \wt{\phead}_K^n - \phead_\sig^n )^+
+ \eta_K(s_\sig^n) (\phead_\sig^n - \phead_K^n\bot \wt \phead_K^n)^+ ,
\label{eq:residue1}
\end{align}
with $s_\sig^n = \Ss_K(p^n_\sig)$. The top line of \eqref{eq:residue1} expresses the upwinded flux of \eqref{eq:Hh.top}, while the bottom line of \eqref{eq:residue1} is the opposite of the upwinded flux of \eqref{eq:Hh.bot}. Note that, since $p^n_\sig = p^{\rm D}_\sig$ is prescribed at $\sig\in\Ee^{\rm D}_{\rm ext}$, we have $\phead^n_\sigma = \phead^n_\sigma \top \wt{\phead}^n_\sigma = \phead^n_\sigma \bot \wt{\phead}^n_\sigma$. Upon inspection of the rearrangement
\begin{align}
\Rr_\sig^n & = [\eta_K(s_K^n \top \wt{s}_K^n) - \eta_K(s_K^n \bot \wt{s}_K^n)] (\phead_K^n \top \wt{\phead}_K^n - \phead_\sig^n )^+ \nonumber\\
& + \eta_K(s_K^n \bot \wt{s}_K^n) [ (\phead_K^n \top \wt{\phead}_K^n - \phead_\sig^n )^+ - (\phead_K^n \bot \wt{\phead}_K^n - \phead_\sig^n )^+ ] \nonumber\\
& +\eta_K(s_\sig^n) \big[ (\phead_\sig^n - \phead_K^n\bot \wt{\phead}_K^n)^+ - (\phead_\sig^n - \phead_K^n\top \wt{\phead}_K^n )^+\big] , \label{eq:residue}
\end{align}
it is trivial that $\Rr_\sig^n \geq 0$. As a consequence, \eqref{eq:uniqueness.1} implies that $\Rr_\sig^n = 0$ for all $\sig \in \Ee_{\rm ext}^{\rm D}$ and that $s_K^n = \wt{s}_K^n$ for all $K\in\Tt$. At this stage, however, we cannot yet claim that $p_K^n = \wt{p}_K^n$, as the function $\Ss_K$ is not invertible.

Taking into account $s_K^n = \wt{s}_K^n$, the residue \eqref{eq:residue} becomes
\begin{align}
\Rr_\sig^n & =  \eta_K(s_K^n) [ (\phead_K^n \top \wt{\phead}_K^n - \phead_\sig^n )^+ - (\phead_K^n \bot \wt{\phead}_K^n - \phead_\sig^n )^+ ] \nonumber\\
& +\eta_K(s_\sig^n) \big[ (\phead_\sig^n - \phead_K^n\bot \wt{\phead}_K^n)^+ - (\phead_\sig^n - \phead_K^n\top \wt{\phead}_K^n )^+\big] ,
\end{align}
which can be lower-bounded by
\be\label{eq:residuelower}
\Rr_\sig^n  \geq \min(\eta_K(s_K^n),\, \eta_K(s_\sig^n)) |\phead_K^n - \wt \phead_K^n|
\ee
thanks to the algebraic identities $a^+ - (-a)^+=a$ and $a\top b - a\bot b = |a-b|$. In view of the lower-bound on the discrete pressures of Lemma~\ref{lem:Boundedness}, we deduce from \eqref{eq:Ss.1} that $s_K^n>0$ and $\wt s_K^n >0$. The increasing behavior of $\eta_K$ implies, in turn, that $\eta_K(s_K^n)>0$ and $\eta_K(\wt s_K^n)>0$. Therefore, the conjunction of $\Rr_\sig^n=0$ and \eqref{eq:residuelower} yields $\phead_K^n = \wt{\phead}_K^n$ and hence
$p_K^n = \wt{p}_K^n$ for all cells $K$ having a Dirichlet boundary edge, i.e.,  $\Ee_K \cap \Ee_{\rm ext}^{\rm D} \neq \emptyset$.

It remains to check that $p_K^n = \wt{p}_K^n$, or equivalently $\phead_K^n = \wt{\phead}_K^n$ for those cells $K\in\Tt$ that are far away from the Dirichlet part of the boundary. 
Subtracting~\eqref{eq:Hh.wt} from~\eqref{eq:Hh} and recalling that $s_K^n = \wt{s}_K^n$, we arrive at
\begin{align}
 \sum_{\sig \in \Ee_{K}} a_\sig\lambda_\sig 
\Big\{ \eta_K(s_K^n) & \big[ \left( \phead_K^n - \phead_{K\sig}^n \right)^+ -  ( \wt \phead_K^n - \wt \phead_{K\sig}^n )^+ \big] \nonumber\\
 +  \eta_{K\sig}(s_{K\sig}^n) & \big[  ( \wt \phead_{K\sig}^n - \wt \phead_K^n )^+ -  ( \phead_{K\sig}^n - \phead_K^n )^+\big] \Big\}  = 0. \label{eq:HJ}
\end{align}
Consider a cell $K\in\Tt$ where $\phead_K^n - \wt \phead_K^n$ achieves its maximal value, i.e., 
\be\label{eq:HJ.1}
\phead_K^n - \wt \phead_K^n \geq \phead_L^n - \wt \phead_L^n, \qquad \forall L\in\Tt. 
\ee
This entails that 
\[
\phead_K^n - \phead_{K\sig}^n \geq \wt \phead_K^n - \wt \phead_{K\sig}^n, \qquad \forall \sig \in \Ee_K, 
\]
so that the two brackets in the right-hand side of \eqref{eq:HJ} are nonnegative. In fact, they both vanish by the positivity of $\eta_K(s_K^n)$ and $\eta_{K\sig}(s_{K\sig}^n)$. As a result, $\phead_K^n - \phead_{K\sig}^n = \wt{\phead}_K^n - \wt{\phead}_{K\sig}^n$ for all $\sig \in \Ee_K$. This implies that $\phead_K^n - \wt \phead_K^n = \phead_L^n - \wt \phead_L^n$ for all the cells $L\in\Tt$ sharing an edge $\sig=K|L$ with $K$, and thus that the cell $L$ also achieves the maximality condition~\eqref{eq:HJ.1}. The process can then be repeated over and over again. Since $\Omega$ is connected, we deduce that $\phead_K^n - \wt \phead_K^n$ is constant over $K\in\Tt$. The constant is finally equal to zero since $\phead_K^n = \wt \phead_K^n$ on the cells having a Dirichlet edge. 
\end{proof}

\section{Convergence analysis}\label{sec:convergence}
Once existence and uniqueness of the discrete solution have been settled, the next question to be addressed is the convergence of the discrete solution towards a weak solution of the continuous problem, as the mesh-size and the time-step are progressively refined. In accordance with the general philosophy expounded in~\cite{eymardEtAl2000finiteVolumeMethod}, the proof is built on compactness arguments. We start by highlighting compactness properties in \S\ref{ssec:compact}, before identifying the limit values as weak solutions in \S\ref{ssec:identify}.

\subsection{Compactness properties}\label{ssec:compact}
Let us define $G_{\Ee_m,\bdt_m}: Q_T \to \R^d$ and $J_{\Ee_m,\bdt_m} : Q_T \to \R^d$ by
\be
G_{\Ee_m,\bdt_m}(t,x) = 
\begin{cases}
\; d \displaystyle\frac{\kirsqr_i(p_{K\sig}^n) - \kirsqr_i(p_K^n)}{d_\sig} n_{K\sig} ,
& \text{if }\; (t,x) \in (t_m^{n-1},t_m^n] \times \Delta_{\sig} , 
\; \\
\; 0 & \text{otherwise},
\end{cases}
\ee
for $\sig \in \Ee_{i,m}$, $1 \leq n \leq N_m$ and, respectively, 
\be
J_{\Ee_m,\bdt_m}(t,x) = d \displaystyle\frac{\kirmin(p_{K\sig}^n) - \kirmin(p_K^n)}{d_\sig} n_{K\sig}, \quad 
\text{if }\; (t,x) \in (t_m^{n-1},t_m^n] \times \Delta_{\sig},
\ee
for $\sig \in \Ee_{m}$, $1 \leq n \leq N_m$. We remind that $s_{\Tt_m,\bdt_m} = \Ss(p_{\Tt_m,\bdt_m},x)$ is the sequence of approximate saturation fields computed from that of approximate pressure fields $p_{\Tt_m,\bdt_m}$ by \eqref{eq:approx_sat}.

\begin{prop}\label{prop:compact}
There exists a measurable function $p:Q_T \to \R$ such that $\kirmin(p)-\kirmin(p^{\rm D}) \in L^2((0,T);V)$ and $\kirsqr_i(p) \in L^2((0,T);H^1(\O_i))$, $1\leq i \leq I$, such that, up to a subsequence,
\begin{subequations}\label{eq:compact}
\begin{alignat}{2}
s_{\Tt_m,\bdt_m} & \underset{m\to+\infty} \longrightarrow \Ss(p,x)  & \qquad & \text{a.e. in }\; Q_T, 
\label{eq:compact.s}\\
G_{\Ee_m, \bdt_m}  & \underset{m\to+\infty} \longrightarrow \grad \kirsqr_i(p)  & \qquad &
\text{weakly in }\; L^2(Q_{i,T})^d, \label{eq:compact.xi}\\
J_{\Ee_m, \bdt_m}  & \underset{m\to+\infty} \longrightarrow \grad \kirmin(p)  & \qquad &
\text{weakly in }\; L^2(Q_{T})^d . \label{eq:compact.Upsilon}
\end{alignat}
\end{subequations}
\end{prop}
\begin{proof}
We know from Corollary~\ref{cor:xi_Upsilon} that $\kirsqr_i(p_{\Tt_m, \bdt_m})$ and $\kirmin(p_{\Tt_m, \bdt_m})$ are bounded w.r.t. $m$ in $L^2(Q_{i,T})$ and $L^2(Q_T)$ respectively, while $G_{\Ee_m, \bdt_m}$ and $J_{\Ee_m, \bdt_m}$ are respectively bounded in $L^2(Q_{i,T})^d$ and $L^2(Q_T)^d$.
In particular, there exist $\wh \kirsqr_i \in L^2(Q_{i,T})$,  $\wh \kirmin \in L^2(Q_T)$, $J \in L^2(Q_{i,T})^d$, and 
$J \in L^2(Q_{T})^d$ such that 
\begin{subequations}\label{eq:conv.L2.weak}
\begin{alignat}{4}
\kirsqr_i(p_{\Tt_m, \bdt_m}) & \underset{m\to+\infty} \longrightarrow \wh \kirsqr_i & \qquad & \text{weakly in}\; L^2(Q_{i,T}), \\
\kirmin(p_{\Tt_m, \bdt_m}) & \underset{m\to+\infty} \longrightarrow \wh \kirmin & \qquad & \text{weakly in}\; L^2(Q_{T}), \\
G_{\Ee_m, \bdt_m}  & \underset{m\to+\infty} \longrightarrow G   & \qquad &
\text{weakly in }\; L^2(Q_{i,T})^d, \\
J_{\Ee_m, \bdt_m}  & \underset{m\to+\infty} \longrightarrow J  & \qquad &
\text{weakly in }\; L^2(Q_{T})^d. 
\end{alignat}
\end{subequations}
Establishing that $\wh \kirsqr_i \in L^2((0,T); H^1(\O_i))$ and $\wh \kirmin\in L^2((0,T);H^1(\O))$ with $G = \grad \wh \kirsqr_i$ and $J = \grad \wh \kirmin$ is now classical, see for instance \cite[Lemma 2]{EG03} or \cite[Lemma 4.4]{CHLP03}.

The key points of this proof are the identification $\wh \kirsqr_i = \kirsqr_i(p)$ and $\wh \kirmin=\kirmin(p)$ for some measurable $p$, as well as the proofs of the almost everywhere convergence property~\eqref{eq:compact.s}. The identification of the limit and the almost everywhere convergence can be handled  simultaneously 
by using twice~\cite[Theorem 3.9]{ACM17}, once for $\kirsqr_i(p)$ and once for $\kirmin(p)$. 
More precisely, Lemma~\ref{lem:boundS} provides a control 
on the time variations of the approximate saturation $s_{\Tt_m, \bdt_m}$, whereas Corollary~\ref{cor:xi_Upsilon} provides some 
compactness w.r.t. space on $\kirsqr_i(p_{\Tt_m,\bdt_m})$ and 
$\kirmin(p_{\Tt_m,\bdt_m})$. Using further that 
$s_{\Tt_m, \bdt_m} = \Ss_i \circ \kirsqr_i^{-1}\left(\kirsqr_i(p_{\Tt_m,\bdt_m})\right)$ with 
$\Ss_i \circ \kirsqr_i^{-1}$ nondecreasing and continuous, 
then one infers from \cite[Theorem 3.9]{ACM17} that 
\[
s_{\Tt_m, \bdt_m} \underset{m\to+\infty} \longrightarrow 
\Ss_i\circ \kirsqr_i^{-1}(\wh \kirsqr_i) \quad \text{a.e. in}\;
Q_{i,T}.
\]
Let $p = \kirsqr_i^{-1}(\wh \kirsqr_i)$. Then, \eqref{eq:compact.s} and \eqref{eq:compact.xi} hold. Proving~\eqref{eq:compact.s} and \eqref{eq:compact.Upsilon} is similar, and 
the properties \eqref{eq:compact} can be assumed to hold for the same function $p$ up to the extraction of yet another subsequence.

Finally, by applying the arguments developed in \cite[\S4.2]{BCH13}, we show that $\kirmin(p)$ and $\kirmin(p^{\rm D})$ share the same trace on $(0,T) \times \Gamma^{\rm D}$, hence $\kirmin(p) - \kirmin(p^{\rm D}) \in L^2((0,T);V)$.
\end{proof}

Let us now define 
\be
\eta_{\Ee_m, \bdt_m}(t,x) = 
\eta_\sig^n \quad \text{if }\; (t,x) \in (t_m^{n-1}, t_m^n]\times \Delta_\sig 
\ee
for $\sig \in \Ee_m$, $1 \leq n \leq N_m$.
\begin{lem}\label{lem:eta_E}
Up to a subsequence, the function $p$ whose existence is guaranteed by Proposition~\ref{prop:compact} satisfies
\be
\eta_{\Ee_m, \bdt_m} \underset{m\to\infty}\longrightarrow \eta(\Ss(p,x)) \qquad \text{in }\; L^q(Q_T), \; 1 \leq q < +\infty.
\ee
\end{lem}
\begin{proof}
Because of~\eqref{eq:compact.s}, $\eta_{\Tt_m, \bdt_m} = \eta(s_{\Tt_m, \bdt_m},x)$ converges almost everywhere to $\eta(\Ss(p,x),x)$. Since $\eta$ is bounded, Lebesgue's dominated convergence theorem ensures that the convergence holds in $L^q(Q_T)$ for all $q\in [1,+\infty)$. The reconstruction $\eta_{\Ee_m, \bdt_m}$ of the mobility is also uniformly bounded, so we have just to show that $\lVert \eta_{\Tt_m, \bdt_m} -\eta_{\Ee_m, \bdt_m} \rVert_{L^1(Q_T)} \rightarrow 0$ as $m \rightarrow +\infty$. Letting $\Delta_{K\sigma}=K \cap \Delta_\sig$ denote the half-diamond cell, we have
\begin{multline*}
\lVert \eta_{\Tt_m, \bdt_m} -\eta_{\Ee_m, \bdt_m} \rVert_{L^1(Q_T)} \leq \;
\sum_{n=1}^{N_m} \dt^n_m \sum_{K\in\Tt_m} \sum_{\sigma\in\Ee_K} m_{\Delta_{K\sigma}} |\eta_K(s_K^n) - \eta_\sigma^n|\\
 \leq \;\sum_{n=1}^{N_m} \dt^n_m \sum_{\sigma\in\Ee_m} m_{\Delta_{\sigma}} |\eta_K(s_K^n) - \eta_{K\sigma}(s_{K\sig}^n)|
 \, \leq \;\sum_{i=1}^I\Rr_{i,m} + \Rr_{\Gamma,m},  
 \end{multline*}
where 
\begin{align*}
\Rr_{i,m}  = \;& \sum_{n=1}^{N_m} \dt^n_m \sum_{\sig \in \Ee_{i,m}} m_{\Delta_{\sigma}} |\eta_K(s_K^n) - \eta_{K\sigma}(s_{K\sig}^n)|, 
\\
\Rr_{\Gamma,m}  = \;&   \sum_{n=1}^{N_m} \dt^n_m \sum_{\sig \in \Ee_{\Gamma,m}} m_{\Delta_{\sigma}} |\eta_K(s_K^n) - \eta_{K\sigma}(s_{K\sig}^n)|.
\end{align*}
Let us define
\[
r_{\Ee_m, \bdt_m}(t,x)= |\eta_K^n -\eta_{K\sigma}^n|= r_\sigma^n \quad \text{if}\; (t,x) \in (t_m^{n-1}, t_m^n]\times \Delta_{\sig},
\]
then $r_{\Ee_m, \bdt_m}$ is uniformly bounded by $\lVert \eta \rVert_\infty = 1/\mu$. Therefore, 
\[
\Rr_{\Gamma,m} \leq \frac T\mu \sum_{\sig \in \Ee_{\Gamma,m}} m_{\Delta_{\sigma}} \leq \frac {2T \,\measure^{d-1}(\Gamma)}{\mu d} \,\size_{\Tt_m}
\]
where $\size_{\Tt_m}$ is the size of $\Tt_m$ as defined in \eqref{eq:diamreg}. 
Besides, for $i \in \{1,\dots, I\}$, $\eta_i \circ \Ss_i \circ \kirsqr_i^{-1}$ is continuous, monotone and bounded, hence uniformly continuous. This provides the existence of a modulus of continuity 
$\varpi_i:\R_+ \to \R_+$ with $\varpi_i(0) = 0$ such that 
\be\label{eq:rnsig}
r_\sigma^n := |\eta \circ \Ss \circ \kirsqr_i^{-1}(\kirsqr_K^n) - \eta \circ \Ss \circ \kirsqr_i^{-1}(\kirsqr_{K\sigma}^n)| 
\leq \varpi_i(|\kirsqr_{K}^n - \kirsqr_{K\sig}^n|)
\ee
for $\sig \in \Ee_{i,m}$. Therefore, if the function
\begin{subequations}
\be\label{eq:qeimdtm}
q_{\Ee_{i,m}, \bdt_m}(t,x) = \begin{cases}
\, |\kirsqr_i(p_K^n) - \kirsqr_i(p_{K\sig}^n)| & \text{if}\; (t,x) \in (t_m^{n-1},t_m^n] \times \Delta_\sig, 
\\
\, 0 & \text{otherwise},
\end{cases}
\ee
for $\sig \in \Ee_{i,m}$, $1 \, \leq \, n \leq N_m$,
\end{subequations}
could be proven to converge to $0$ almost everywhere in $Q_{i,T}$, then it would also be the case for $r_{\Ee_m,\bdt_m}$ and $\Rr_{i,m}$ as $m\rightarrow+\infty$, thanks to Lebesgue's dominated convergence theorem. Now, it follows from \eqref{eq:dissip.xi} and from the elementary geometric relation 
\[
m_{\Delta_\sig} = \frac{a_\sig}{d} d_\sig^2 \, \leq \,  4 \frac{a_\sig}{d} \size_{\Tt_m}^2 ,
\]
that 
\[
\lVert q_{\Ee_{i,m}, \bdt_m} \rVert^2_{L^2(Q_{i,T})} 
= \sum_{n=1}^{N_m} \dt^n_m \sum_{\sig \in \Ee_{i,m}} m_{\Delta_\sig} |\kirsqr_i(p_K^n) - \kirsqr_i(p_{K\sig}^n)|^2 
\leq \frac{4 \cter{c:dissip.max}}d \size_{\Tt_m}^2. 
\]
Therefore, $q_{\Ee_{i,m}, \bdt_m}\rightarrow 0$ in $L^2(Q_{i,T})$, thus also almost everywhere up to extraction of a subsequence. This provides the desired result.
\end{proof}

\subsection{Identification of the limit}\label{ssec:identify}

So far, we have exhibited some ``limit'' value $p$ for the approximate solution $p_{\Tt_m, \bdt_m}$ in Proposition~\ref{prop:compact}. Next, we show that the scheme is consistent with the continuous problem by showing that any limit value is a weak solution.

\begin{prop}\label{prop:identify}
The function $p$ whose existence is guaranteed by Proposition~\ref{prop:compact} is a weak solution of the problem~\eqref{eq:cont.mass}--\eqref{eq:Ss} in the sense of Definition~\ref{def:weak}.
\end{prop}
\begin{proof}
Let $\fntest \in C_{c}^{\infty}(\{\Omega\cup \Gamma^{\rm N}\} \times [0, T))$ and denote by $\fntest_{K}^{n}=\fntest(t^n_m,x_K)$, for all $K \in \Tt_m$ and all $n\in\{0,\dots,N_m\}$. We multiply \eqref{eq:scheme.mass}
 by $\dt^n_m \fntest_K^{n-1}$ and sum over $n\in\{1,\dots,N_m\}$ and $K\in\Tt_m$ to obtain 
\be\label{eq:A+B_m}
\Aa_m + \Bb_m=0, \qquad m \geq 1, 
\ee
where we have set
\begin{subequations} 
\begin{align}
\Aa_m =& \sum_{n=1}^{N_m} \sum_{K\in\Tt_m} m_K \phi_K (s_K^n - s_K^{n-1}) \fntest_K^{n-1}, \label{eq:ident.Am}\\
\Bb_m = & \sum_{n=1}^{N_m} \dt^n_m \sum_{\sig \in \Ee_m} a_\sig \lambda_\sig 
\eta_\sig^n (\phead_K^n - \phead_{K\sig}^n ) (\fntest_K^{n-1} - \fntest_{K\sig}^{n-1} ). \label{eq:ident.Bm}
\end{align}
\end{subequations}
The quantity $\Aa_m$ in \eqref{eq:ident.Am} can be rewritten as
\begin{align*}
\Aa_m&=-\sum_{n=1}^{N_m} \dt^n_m \sum_{K\in\Tt_m}  m_K  \phi_K s_K^n \frac{\fntest_K^n - \fntest_K^{n-1}}{\dt_m^n}- \sum_{K\in\Tt_m} m_K \phi_K s_K^0 \fntest_K^0  \\
&=- \iint_{Q_T}\phi\, s_{\Tt_m,\bdt_m} \delta \fntest_{\Tt_m,\bdt_m} \, \d x \, \d t-\int_\Omega \phi\, s_{\Tt_m}^0  \fntest_{\Tt_m}^0 \, \d x
\end{align*}
where
\begin{alignat*}{2}
\delta \fntest_{\Tt_m,\bdt_m}(t,x) & = \frac{\fntest_K^n - \fntest_K^{n-1}}{\dt_m^n},&\qquad\text{if }\; & (t,x) \in (t_m^{n-1},t_m^{n}) \times K,\\
\fntest_{\Tt_m}^0 &= \fntest(0,x_K) & \qquad \text{if } \; & x \in K.
\end{alignat*}
Thanks to the regularity of $\fntest$, the function $\delta \fntest_{\Tt_m,\dt_m}$ converges uniformly to $\partial_t \fntest$ on $\Omega \times [0,T]$. Moreover, by virtue of \eqref{eq:compact.s} and the boundedness of $s_{\Tt_m,\bdt_n}$ we can state that
\[
\iint_{Q_T} \phi \, s_{\Tt_m,\dt_m} \delta \fntest_{\Tt_m,\dt_m} \, \d x \, \d t \underset{m\to+\infty}\longrightarrow \iint_{Q_T} \phi \,\Ss(p,x) \partial_t \fntest \, \d x \, \d t,
\]
and, in view of the definition~\eqref{eq:scheme.init} of $s^0_{\Tt_m}$ and of the uniform convergence of $\fntest_{\Tt_m}^0$ towards 
$\fntest(0,\cdot)$, 
\[\int_\Omega \phi_{\Tt_m}  s_{\Tt_m}^0  \fntest^0_{\Tt_m} \, \d x \underset{m\to+\infty}\longrightarrow \int_\Omega \phi \, s^0 \fntest(0,\cdot) \, \d x.\]
From the above, we draw that
\be\label{eq:A_m}
\lim_{m\to+\infty} \Aa_m  = -\iint_{Q_T} \phi \, \Ss(p,x) \partial_t \fntest \, \d x \, \d t - \int_\Omega \phi \, s^0 \fntest(0,\cdot) \, \d x.
\ee

Let us now turn our attention to the quantity $\Bb_m$ of \eqref{eq:ident.Bm}, which can be split into $\Bb_m= \Bb_m^1 + \Bb_m^2$ using 
\begin{align*}
\Bb_m^1=& \sum_{n=1}^{N_m} \dt_m^n \sum_{\sigma \in \Ee_m} 
a_\sigma \lambda_\sigma \eta_\sigma^n (p_K^n-p_{K\sig}^n)(\fntest_K^{n-1}-\fntest_{K\sig}^{n-1}), \\
\Bb_m^2=&\sum_{n=1}^{N_m} \dt_m^n \sum_{\sigma \in \Ee_m}  a_\sigma \lambda_\sigma \eta_\sigma^n 
(\psi_K - \psi_{K\sig})(\fntest_K^{n-1}-\fntest_{K\sig}^{n-1}).
\end{align*}
Consider first the convective term $\Bb_m^2$. It follows from the definition of the discrete gravitational potential 
\[
\psi_K = -\varrho g\cdot x_K, \quad \psi_\sig =  -\varrho g\cdot x_\sig, \qquad K\in\Tt_m, \; \sig \in \Ee_{{\rm ext},m}^{\rm D}
\]
and from the orthogonality of the mesh that 
\[
\psi_K - \psi_{K\sig} =  d_\sig \varrho g\cdot n_{K\sig}, \qquad \forall \sig \in \Ee_K \setminus \Ee_{\rm ext}^{\rm N}, \; K \in \Tt_m.
\]
Therefore, $\Bb^2_m$ can be transformed into 
\begin{align}
\Bb_m^2 = & \sum_{n=1}^{N_m} \dt_m^n \sum_{\sigma \in \Ee_m} m_{\Delta_\sig} \lambda_\sig \eta_\sig^n 
d  \frac{\fntest_K^{n-1} - \fntest_{K\sig}^{n-1}}{d_\sig} n_{K\sig} \cdot \varrho g \nonumber\\
=&  - \iint_{Q_T}\lambda_{\Ee_m}\eta_{\Ee_m, \bdt_m} H_{\Ee_m, \bdt_m} \cdot \varrho g\, \, \d x \, \d t, \label{eq:Bm2.1}
\end{align}
where 
\begin{alignat*}{2}
\lambda_{\Ee_m}(x) &= \lambda_\sig & \qquad \text{if }\; & x \in \Delta_\sig, \; \sig \in \Ee_m,\\
H_{\Ee_m, \bdt_m}(t,x) &= (d/d_\sig)  (\fntest_{K\sig}^{n-1} - \fntest_{K}^{n-1}) n_{K\sig}  & \qquad \text{if } \; & (t,x) \in [t^{n-1}_m, t^n_m) \times \Delta_\sig. 
\end{alignat*}
After~\cite[Lemma 4.4]{CHLP03}, $H_{\Ee_m, \bdt_m}$ converges weakly in $L^2(Q_T)^d$ towards $\grad \fntest$, while $\lambda_{\Ee_m}$ and $\eta_{\Ee_m, \bdt_m}$ converge strongly in $L^4(\O)$ and $L^4(Q_T)$ towards $\lambda$ and $\eta(\Ss(p,x))$ respectively (cf. Lemma~\ref{lem:eta_E}). Thus, we can pass to the limit in~\eqref{eq:Bm2.1} and 
\be\label{eq:Bm2.2}
\lim_{m\to+\infty} \Bb_m^2 = - \iint_{Q_T}\lambda \eta(\Ss(p,x)) \varrho g \cdot \grad \fntest\, \d x \, \d t.
\ee

The capillary diffusion term $\Bb_m^1$ appears to be the most difficult one to deal with. Taking inspiration from~\cite{cancesGuichard2016EntropyScheme}, we introduce the auxiliary quantity
\begin{align*}
\wt{\Bb}{}_m^1 &=  \sum_{i=1}^I \wt \Bb_{i,m}^1 \\
& = \sum_{i=1}^I  \sum_{n=1}^{N_m} \dt_m^n \!\! \sum_{\sig \in \Ee_{i,m}} 
\!\! a_\sig \sqrt{\lambda_i \eta_\sig^n} (\kirsqr_i(p_K^n)-\kirsqr_i(p_{K\sig}^n))(\fntest_K^{n-1}-\fntest_{K\sig}^{n-1}).
\end{align*}
Analogously to~\cite{DE_FVCA8}, {we can define a piecewise-constant vector field $\ov{H}_{\Ee_m, \bdt_m}$ such that 
\[
\ov{H}_{\Ee_m, \bdt_m}(t,x) \cdot n_{K\sig} = \fntest_{K\sig}^{n-1} - \fntest_K^{n-1}, 
\qquad \text{if } \; (t,x) \in [t^{n-1}_m, t^n_m) \times \Delta_\sig, \; \sig \in \Ee_m,
\]
and such} that $\ov{H}_{\Ee_m, \bdt_m}$ converges uniformly towards $\grad \fntest$ on $\ov Q_T$. Under these circumstances, $\wt{\Bb}_{i,m}^1$ reads
\[
\wt \Bb_{i,m}^1 =\int_0^T \!\! \int_{\O_{i,m}}  \sqrt{\lambda_i \eta_{\Ee_m, \bdt_m}} \, 
G_{\Ee_m, \bdt_m}\cdot \ov H_{\Ee_m, \bdt_m}  \, \d x \, \d t
\]
where $\O_{i,m} = \bigcup_{\sig \in \Ee_{i,m}} \Delta_\sig \subset \O_i$. The strong convergence of $\sqrt{\eta_{\Ee_m, \bdt_m}}$ in $L^2(Q_{i,T})$ towards $\sqrt{\eta_i(\Ss_i(p))}$ directly follows from the boundedness of $\eta_i$ combined with~\eqref{eq:compact.s}. Combining this with~\eqref{eq:compact.xi} results in 
\be\label{eq:wtBim1}
\wt{\Bb}_{i,m}^1 \underset{m\to + \infty} \longrightarrow \iint_{Q_{i,T}} \!\!\!\!\!\!  \sqrt{\lambda_i \eta_i(\Ss_i(p))} \grad \kirsqr_i(p) \cdot \grad \fntest \, \d x \, \d t
= \iint_{Q_{i,T}} \!\!\!\!\!\!  \grad \kircho_i(p) \cdot \grad \fntest \, \d x \,\d t.
\ee
Therefore, to finish the proof of Proposition~\ref{prop:identify}, it only remains to check that $B^1_m$ and $\wt B^1_m$ share the same limit. To this end, we observe that, by the triangle inequality, we have 
\be\label{eq:diff.B}
| \Bb_m^1 - \wt \Bb_{m}^1 |\leq  \Rr_{\Gamma,m} +  \sum_{i = 1}^I \Rr_{i,m}, 
\ee
where
\begin{align*}
\Rr_{\Gamma,m} & =  \sum_{n=1}^{N_m}\dt_m^n  \!\!\! \sum_{\sigma \in \Ee_{\Gamma,m}} \!\!\!
a_\sigma \lambda_\sigma \eta_\sigma^n |p_K^n-p_{K\sig}^n| |\fntest_K^{n-1}-\fntest_{K\sig}^{n-1}|,
\\
\Rr_{i,m} & = \sum_{n=1}^{N_m} \dt_m^n \!\!\! \sum_{\sigma \in \Ee_{i,m}} \!\!\!
a_\sig \sqrt{\lambda_i \eta_\sig^n} | \kirsqr_i(p_K^n) - \kirsqr_i(p_{K\sig}^n) - \sqrt{\lambda_i \eta_\sig^n}(p_K^n - p_{K\sig}^n) |
|\fntest_K^{n-1}-\fntest_{K\sig}^{n-1}|.
\end{align*}
Applying the Cauchy-Schwarz inequality and using Proposition~\ref{prop:dissip}, we find
\[
|\Rr_{\Gamma,m}|^2 \leq \cter{c:main}   \sum_{n=1}^{N_m} \dt_m^n \!\!\sum_{\sig \in \Ee_{\Gamma, m}} \!\! a_\sig \lambda_\sig \eta_\sig^n 
|\fntest_{K}^{n-1} - \fntest_{K\sig}^{n-1}|^2 \leq 2 \cter{c:main} T \lVert \grad \fntest \rVert_\infty^2 \frac{\max_i \lambda_i}{\mu}\measure^{d-1}(\Gamma)\size_{\Tt_m}, 
\]
so $\Rr_{\Gamma,m}\rightarrow 0$ as $m\rightarrow +\infty$.
Besides, we also apply the Cauchy-Schwarz inequality to $\Rr_{i,m}$ in order to obtain
\begin{align*}
|\Rr_{i,m}|^2 & \leq \cter{c:main}  \sum_{n=1}^{N_m} \dt_m^n  \!\! \sum_{\sig \in \Ee_{i, m}} \!\! a_\sig \, \lambda_i
\big| \sqrt{\eta_\sig^n} - \sqrt{\wt \eta_\sig^n} \, \big|^2 \, \big|\fntest_K^{n-1} - \fntest_{K\sig}^{n-1}\big|^2 \\
& \leq   d \lambda_i   \cter{c:main} \lVert \grad\fntest \rVert_\infty^2   
\sum_{n=1}^{N_m} \dt_m^n \!\! \sum_{\sig \in \Ee_{i, m}} \!\! m_{\Delta_\sig} \left|{\eta_\sig^n} -{\wt \eta_\sig^n}\right| ,
\end{align*}
where we have set
\[
\wt \eta_\sig^n = \begin{cases}
\, \eta_i(s_K^n) & \text{if}\; p_K^n = p_{K\sig}^n,\\
\, \displaystyle\bigg[\frac{\kirsqr_i(p_{K}^n) - \kirsqr_i(p_{K\sig}^n)}{\sqrt{\lambda_i} (p_K^n - p_{K\sig}^n)}\bigg]^2 & \text{otherwise}.
\end{cases}
\]
Define 
\[
\wt \eta_{\Ee_m, \bdt_m}(t,x) = \begin{cases} 
\, \wt \eta_\sig^n & \text{if}\; (t,x) \in (t_m^{n-1}, t_m^n] \times \Delta_\sig, \; \sig \in \bigcup_{i=1}^I \Ee_{i,m}, \\
\, 0 & \text{otherwise}.
\end{cases}
\]
Reproducing the proof of Lemma~\ref{lem:eta_E}, we can show that 
\[
\wt \eta_{\Ee_m, \bdt_m} \underset{m\to\infty}\longrightarrow \eta(\Ss(p,x)) \qquad \text{in }\; L^q(Q_T), \; 1 \leq q < +\infty.
\]
Therefore, $\Rr_{i,m}\rightarrow 0$ as $m\rightarrow +\infty$. Putting things together in~\eqref{eq:diff.B}, we conclude that $\Bb_m^1$ and $\wt \Bb_{m}^1$ share the same limit, which completes the proof of Proposition~\ref{prop:identify}.
 \end{proof}


\section{Practical aspects of numerical resolution}
\label{sec:numerics}
We provide some details on the resolution strategy for the discrete problem \eqref{eq:scheme.capi}--\eqref{eq:scheme.upwind}. It is based on a parametrization technique to automatically choose the most convenient variable during the Newton iterations 
(\S\ref{ssec:variableSwitch}) and on the addition of cells on the interfaces between different rock types (\S\ref{ssec:BCET}) to improve the pressure continuity.

\subsection{Switch of variable and parametrization technique}
\label{ssec:variableSwitch}

A natural choice to solve the nonlinear system \eqref{eq:scheme.capi}--\eqref{eq:scheme.upwind} is to select the pressure $(p_K)_{K\in\Tt}$ as primary unknown and to solve it via an iterative method such as Newton's one. Nevertheless, the pressure variable is known to be an inefficient choice for $s \ll 1$ because of the degeneracy of Richards' equation. For dry soils, this strategy is outperformed by schemes in which saturation is the primary variable. On the other hand, the knowledge of the saturation is not sufficient to describe the pressure curve in saturated regions where the pressure-saturation relation cannot be inverted. This motivated the design of schemes involving a switch of variable \cite{DP99}--\cite{FWP95}. In this work, we adopt the technique proposed by Brenner and Cancès \cite{BC17}, in which a third generic variable $\tau$ is introduced to become the primary unknown of the system. Then the idea is to choose a parametrization of the graph $\{p, \mathcal{S}(p)\}$, i.e., to construct two functions $\fs: I \rightarrow [s_{\mathrm{rw}}, 1-s_{\mathrm{rn}}]$ and $\fp: I \rightarrow \mathbb{R}$ such that $\fs(\tau)=\mathcal{S}(\fp(\tau))$ and $\fs'(\tau)+\fp'(\tau)>0$ for all $\tau \in I \subset \mathbb{R}$. Such a parametrization is not unique, for instance one can take $I=\mathbb{R}$, $\fp=Id$ which amounts to solving the system always in pressure, but this is not recommended as explained before. Here, we set $I=(s_{\mathrm{rw}}, + \infty)$ and

\begin{align*}
\fs(\tau) = \begin{cases}
\tau & \text{if}\; \tau \leq s_{\rm s}, \\
\Ss\left(p_{\rm s} + \displaystyle\frac{\tau-s_{\rm s}}{\Ss'(p_{\rm s}^-)}\right) & \text{if}\; \tau \geq s_{\rm s}, \\
\end{cases}
\qquad 
\fp(\tau) = \begin{cases}
\Ss^{-1}(\tau) & \text{if}\; \tau \leq s_{\rm s}, \\
p_{\rm s} + \displaystyle\frac{\tau-s_{\rm s}}{\Ss'(p_{\rm s}^-)}& \text{if}\; \tau \geq s_{\rm s},
\end{cases}
\end{align*}
where $\Ss'(p_{\rm s}^-)$ denotes the limit as $p$ tends to $p_{\rm s}=\mathcal{S}(s_{\rm s})$ from below of $\Ss'(p)$.
Since the switch point $s_{\rm s}$ is taken as the inflexion point of $\Ss$, both $\fs$ and $\fp$ are $C^1$ and concave, and even $C^2$ if $\Ss$ is given by the Van Genuchten model. Moreover, for all $p\in \R$, there exists a unique $\tau \in  (s_{\rm rw},+\infty)$ such that 
$(p,\Ss(p)) = (\fp(\tau), \fs(\tau))$. 
The resulting system $\mathcal F_n(\boldsymbol{ \tau}^n) = \mathbf{0}$ made up of
$N_\Tt = {\rm Card}(\Tt)$ nonlinear equations admits a unique 
solution $\boldsymbol{\tau}^n$, since it is fully equivalent to \eqref{eq:scheme.capi}--\eqref{eq:scheme.upwind}. 
More details about the practical resolution of this nonlinear system via the Newton method can be found in \cite{BCET_FVCA9}.  

\subsection{Pressure continuity at rock type interfaces}
\label{ssec:BCET}
Physically, the pressure should remain continuous on both sides of an interface between two different rock types. But this continuity is here not imposed at the discrete level. The two-point flux approximation based on the cell unknowns is strongly dependent on the mesh resolution and can induce a large error close to the rock type interface. We here propose a very simple method to improve this continuity condition in pressure. It consists in adding two thin cells of resolution $\delta$ around the rock-type interface with $\delta \ll \Delta x$ as shown in Figure \ref{fig:vt2}. 
\begin{figure}[htb]
    \centering
    \includegraphics[width=6cm]{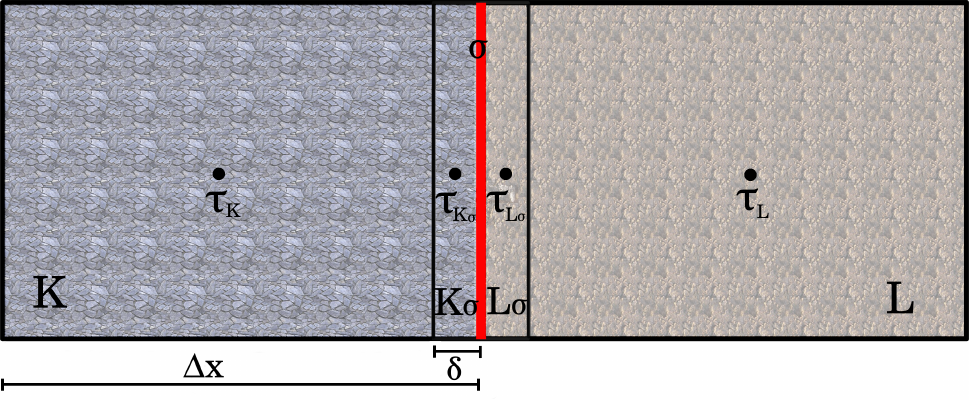}
    \caption{Mesh refinement on both sides of an interface face for a 2D case. }
    \label{fig:vt2}
\end{figure}
The idea is here to add two unknowns in the neighborhood of the interface to have a more precise approximation of the pressure gradient on each side of the faces where changes of rock types occur. In this way, we avoid the introduction of face unknowns in our solver which remains unchanged. Alternatives that strictly impose the pressure continuity will be studied and compared with this approach in a forthcoming work.

\section{Numerical results}\label{sec:results}

In this section, we present the results obtained for different test cases. For all these cases, we consider a two-dimensional layered domain $\Omega=[0\mathrm{m}, 5\mathrm{m}]\times [-3 \mathrm{m}, 0\mathrm{m}]$ made up of two rock types denoted by RT0 and RT1 respectively, RT0 being less permeable than RT1. Using these two lithologies, the domain $\Omega$ is partitioned into three connected subdomains: $\Omega_1=[1\mathrm{m}, 4\mathrm{m}] \times [-1\mathrm{m}, 0\mathrm{m}]$, $\Omega_2= [0 \mathrm{m},5 \mathrm{m}] \times [-3 \mathrm{m}, -2\mathrm{m}]$ and $\Omega_3= \Omega ~\setminus~ (\Omega_1 \cup \Omega_2)$, {as depicted in Figure} \ref{fig:testDomain}.
\begin{figure}[htb]
    \centering
    \includegraphics[width=6.4cm]{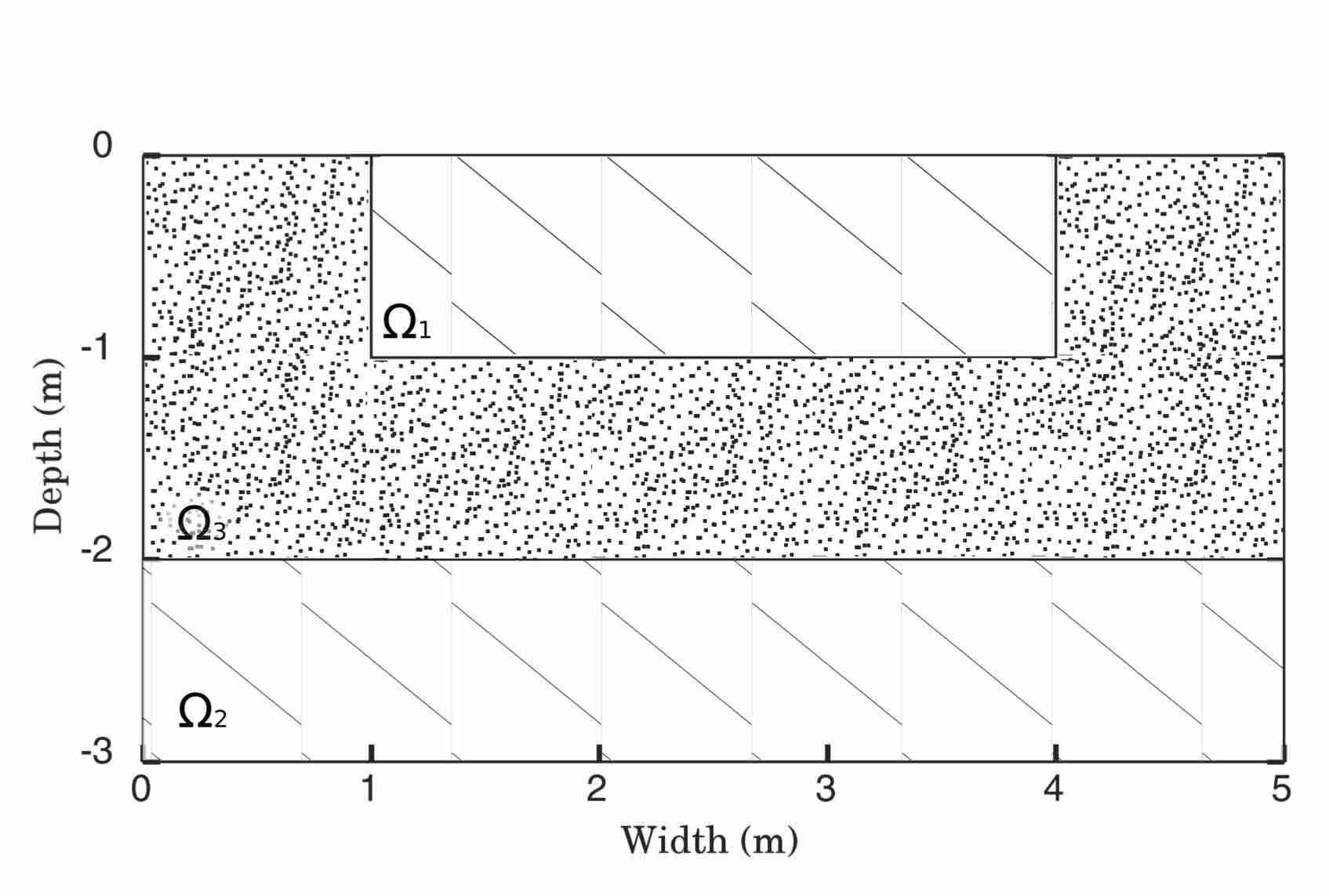}
    \caption{Simulation domain $\Omega=[0\mathrm{m},5\mathrm{m}]\times[-3\mathrm{m},0\mathrm{m}]$.}
     \label{fig:testDomain}
\end{figure}

The Brooks-Corey \cite{BROOKS} and van Genuchten-Mualem \cite{VANGENUCHTEN} petro-physical models are used to model the flow characteristics of both rock types. In these models, the water saturation and the capillary pressure are linked pointwise by the relation $s=\mathcal{S}(p)$ where $\mathcal{S}: \mathbb{R} \rightarrow [0,1]$ is nondecreasing and satisfies $\mathcal{S}(p)=1 - s_{\mathrm{rn}}$ if $p \geq p_b$ and $\mathcal{S}(p)\rightarrow s_{\mathrm{rw}}$ as $p \rightarrow -\infty$, $s_{\mathrm{rw}}$ being the residual wetting saturation, $s_{\mathrm{rn}}$ the residual non-wetting saturation and $p_b$ the entry pressure. More precisely, we have,
\begin{itemize}
\item for the Brooks-Corey model,
\begin{align*}
&s= \mathcal S(p)=\begin{cases}
       s_\mathrm{rw} + (1-s_\mathrm{rn}-s_\mathrm{rw})\left( \frac{p}{p_b} \right)^{- n} & \text{if}~p \leq p_b,\\
       1-s_\mathrm{rn} & \text{if}~p > p_b,
      \end{cases} \nonumber\\
 &\nonumber\\     
&k_r(s)=s_{\mathrm{eff}}^{3+ \frac{2}{n}}, \qquad\qquad  s_{\mathrm{eff}}= \frac{s-s_\mathrm{rw}}{1-s_\mathrm{rn}-s_\mathrm{rw}};\nonumber
\end{align*}

\item for the van Genuchten-Mualem model,
\begin{align*}
&s= \mathcal S(p)=\begin{cases}
        s_\mathrm{rw} + (1-s_\mathrm{rn}-s_\mathrm{rw}) \left[1 + \left|\frac{\alpha}{\rho g} p \right|^n\right ]^{-m} & \text{if}~p \leq 0,\\
	1-s_\mathrm{rn} & \text{if}~p > 0,
      \end{cases} \nonumber \\
  &\nonumber\\    
&k_r(s)=s_{\mathrm{eff}}^{\frac{1}{2}} \{1 - [1- s_{\mathrm{eff}}^{\frac{1}{m}}]^m \}^2 , \qquad s_{\mathrm{eff}}= \frac{s-s_\mathrm{rw}}{1-s_\mathrm{rn}-s_\mathrm{rw}},\qquad\qquad m= 1 -\frac{1}{n};\nonumber
\end{align*}
\end{itemize}
where $\eta(\cdot)= k_r(\cdot)/\mu$, $\mu= 10^{-3}~\textrm{Pa}\cdot\textrm{s}$ being water viscosity, is the relative permeability. 
The parameters used for both rock types are given in Table \ref{table:SoilsBC} for the Brooks-Corey model and in Table \ref{table:SoilsVG} for the van Genuchten-Mualem model.
{With these choices of parameters, water is more likely to be in RT1 than in RT0, in the sense tha, at a given pressure, the water saturation is higher in RT1 than in RT0, as it can be seen on the
plots of the capillary-pressure functions depicted in Figures \ref{img:lawsBC}--\ref{img:lawsVG} for these two petro-physical models.  
Figures \ref{img:lawsBC}--\ref{img:lawsVG} also show the relative permeability functions. Note the non-Lipschitz character of the relative permeability in the van Genuchten-Mualem framework.} {For the numerical tests, in order to avoid infinite values for the derivative of $k_r(s)$ when $s\rightarrow 1-s_\mathrm{rn}$, we approximate it for $s\in[s_{lim},1-s_\mathrm{rn}]$ using a second degree polynomial $\widetilde{k_r}(s)$. Such a polynomial satisfies the following constraints: $k_r(s_{lim})=\widetilde{k_r}(s_{lim})$ and $\widetilde{k_r}(1-s_\mathrm{rn})=1$. The value $s_{lim}$ {corresponds to $s_{\text{eff}}=0.998$.} }

\begin{table}[htp!]
\centering
\begin{tabular}{p{0.8cm}p{1.2cm}p{1cm}p{2.2cm}p{1cm}p{1cm}p{1cm}}
\hline
& $1-s_{rn}$ & $s_{rw}$ & $p_b [\textrm{Pa}]$ & $n$ & $\lambda [\mathrm{m}^2]$ & $\phi$ \\
\hline
RT0 &  $1.0$ & $0.1$ & $-1.4708 \cdot 10^3$ & $3.0$ &  $10^{-11}$ & $0.35$ \\
RT1 &$1.0$ &$0.2$& $-3.4301 \cdot 10^3$  & $1.5$ & $10^{-13}$ & $0.35$\\
\hline\\
\end{tabular}
\caption{Parameters used for the Brooks-Corey model}
\label{table:SoilsBC}
\end{table}

\begin{figure}[htb]
    \centering
    \includegraphics[height=4cm]{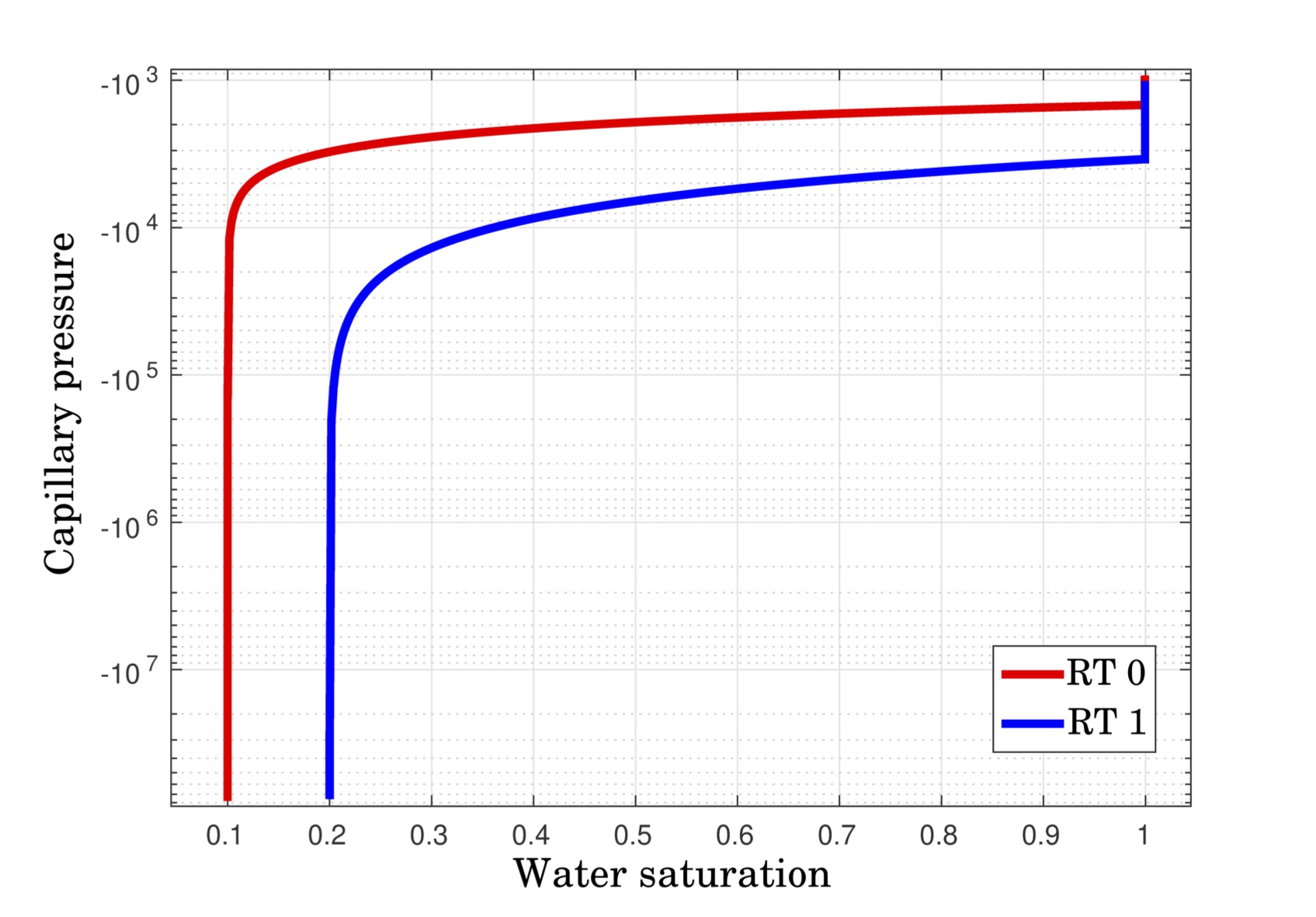}
    \includegraphics[height=4cm]{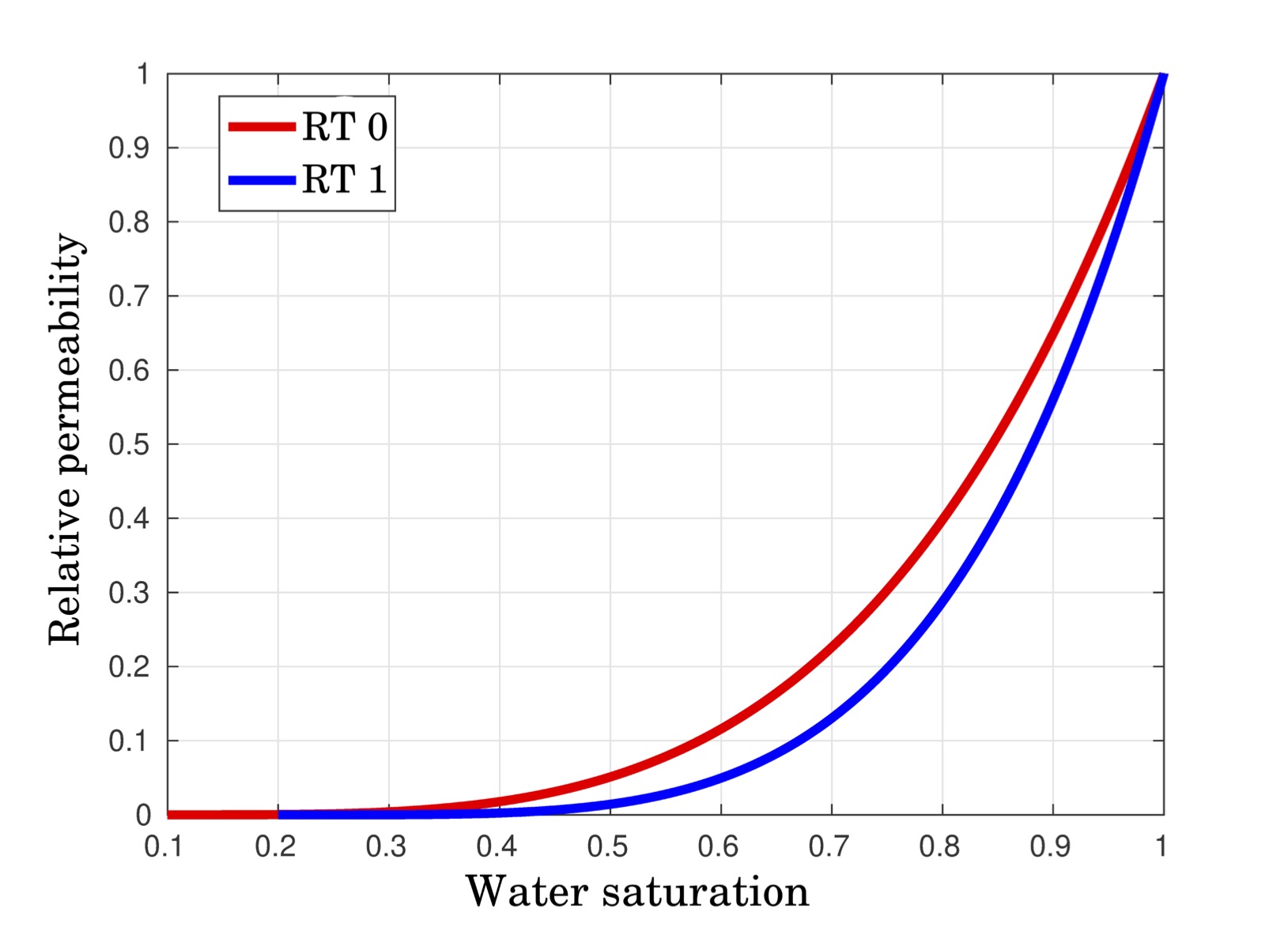}
    \caption{Capillary pressure and relative permeability curves for the Brooks-Corey model}
    \label{img:lawsBC}
\end{figure}

\begin{table}[htp!]
\centering
\begin{tabular}{p{1.8cm}p{1.2cm}p{1cm}p{1cm}p{2.2cm}p{1.1cm}p{1cm}p{1cm}}
\hline
& $1-s_{\rm rn}$  & $s_{\rm rw}$ & $n$  & $\lambda~[\mathrm{m}^2]$ & $\alpha~[\mathrm{m}^{-1}]$ & $\phi$ \\
\hline
RT0 (Sand) & $1.0$  & $0.0782$   & $2.239$ & $6.3812 \cdot 10^{-12}$ & $2.8 $  & $0.3658$\\
RT1 (Clay) & $1.0$  & $ 0.2262 $ & $1.3954$ & $ 1.5461 \cdot 10^{-13}$ & $1.04$ & $0.4686$\\
\hline\\
\end{tabular}
\caption{Parameters used for the van Genuchten-Mualem model}
\label{table:SoilsVG}
\end{table}

\begin{figure}[htp]
    \centering
    \includegraphics[height=4cm]{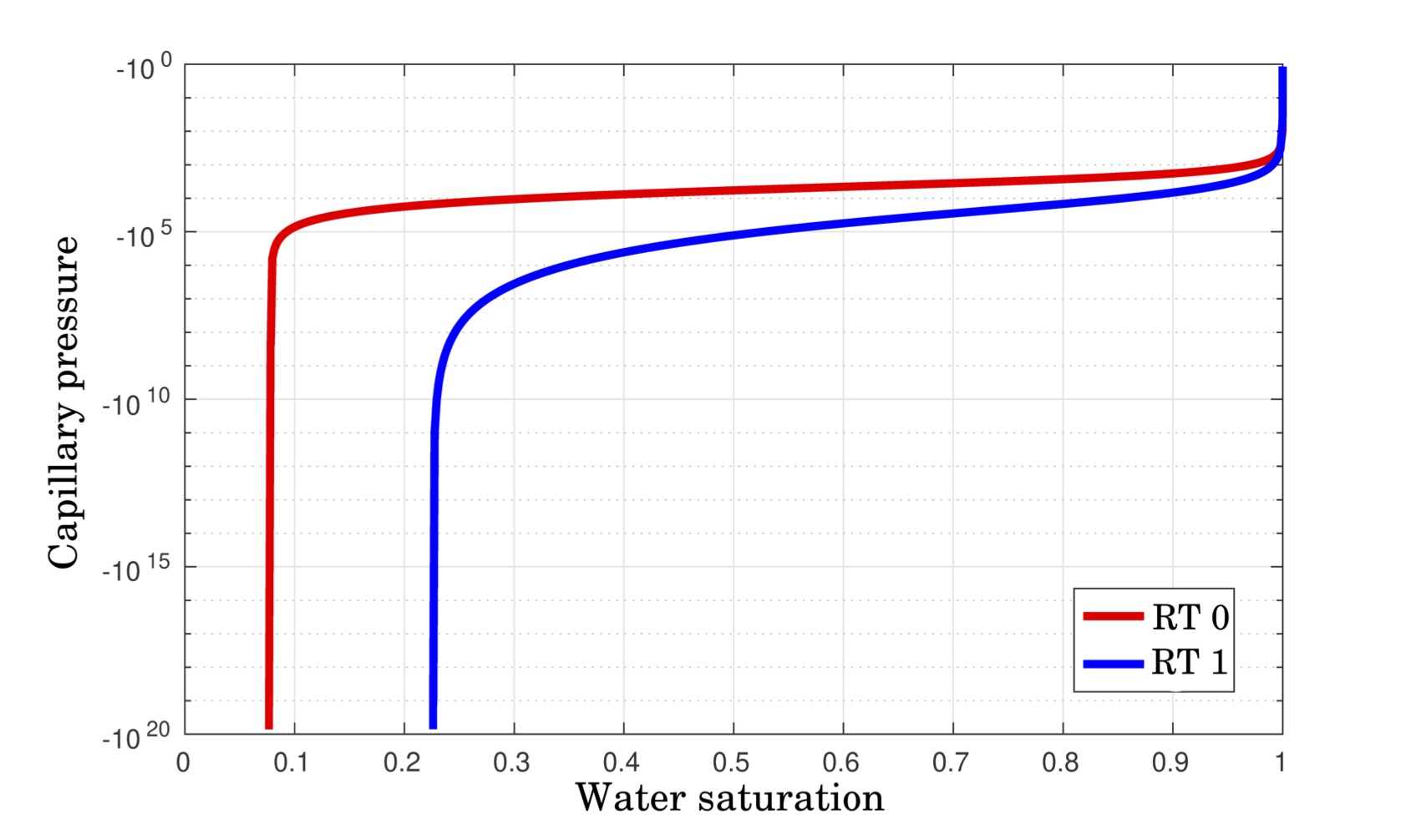}
    \includegraphics[height=4cm]{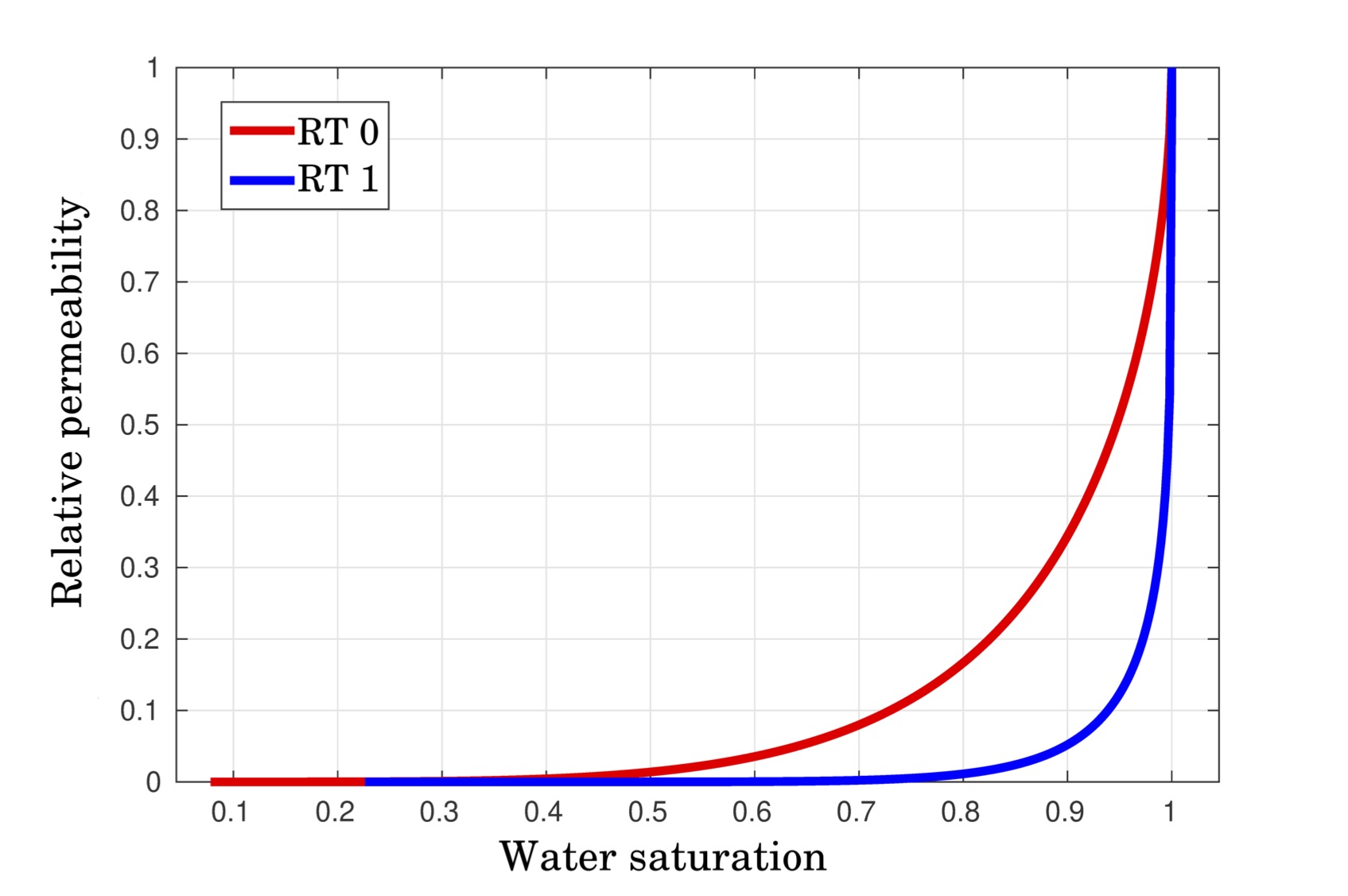}
    \caption{Capillary pressure and relative permeability curves for the van Genuchten-Mualem model}
    \label{img:lawsVG}
\end{figure}

\subsection{Configurations {of the test cases}}
For both petro-physical models, we consider two configurations further referred as filling and drainage cases, which are described in the following. 

\subsubsection{Filling case}
{The filling test case has already been considered in~\cite{KHW92,FWP95, McBride06,CZ10}}. Starting from an initially dry domain $\Omega$, whose layers' composition is reported in Figure \ref{fig:fillingTest}, water flows from a part of the top boundary during one day. A no-flow boundary condition is applied elsewhere. More precisely, the initial capillary pressure is set to $-47.088 \cdot 10^5\textrm{Pa}$ and the water flux rate to $0.5 \textrm{m/day}$ through $\Gamma_N=\{(x,y) \,| \, x \in [1\mathrm{m}, 4\mathrm{m}], y=0\mathrm{m}\}$. For this simulation an homogeneous time-step $\Delta t= 1000\mathrm{s}$ is prescribed for the test using the Brooks-Corey model and $\Delta t= 500\mathrm{s}$ for the one using the van Genuchten-Mualem model.

\begin{figure}[htb]
    \centering
    \includegraphics[height=5cm]{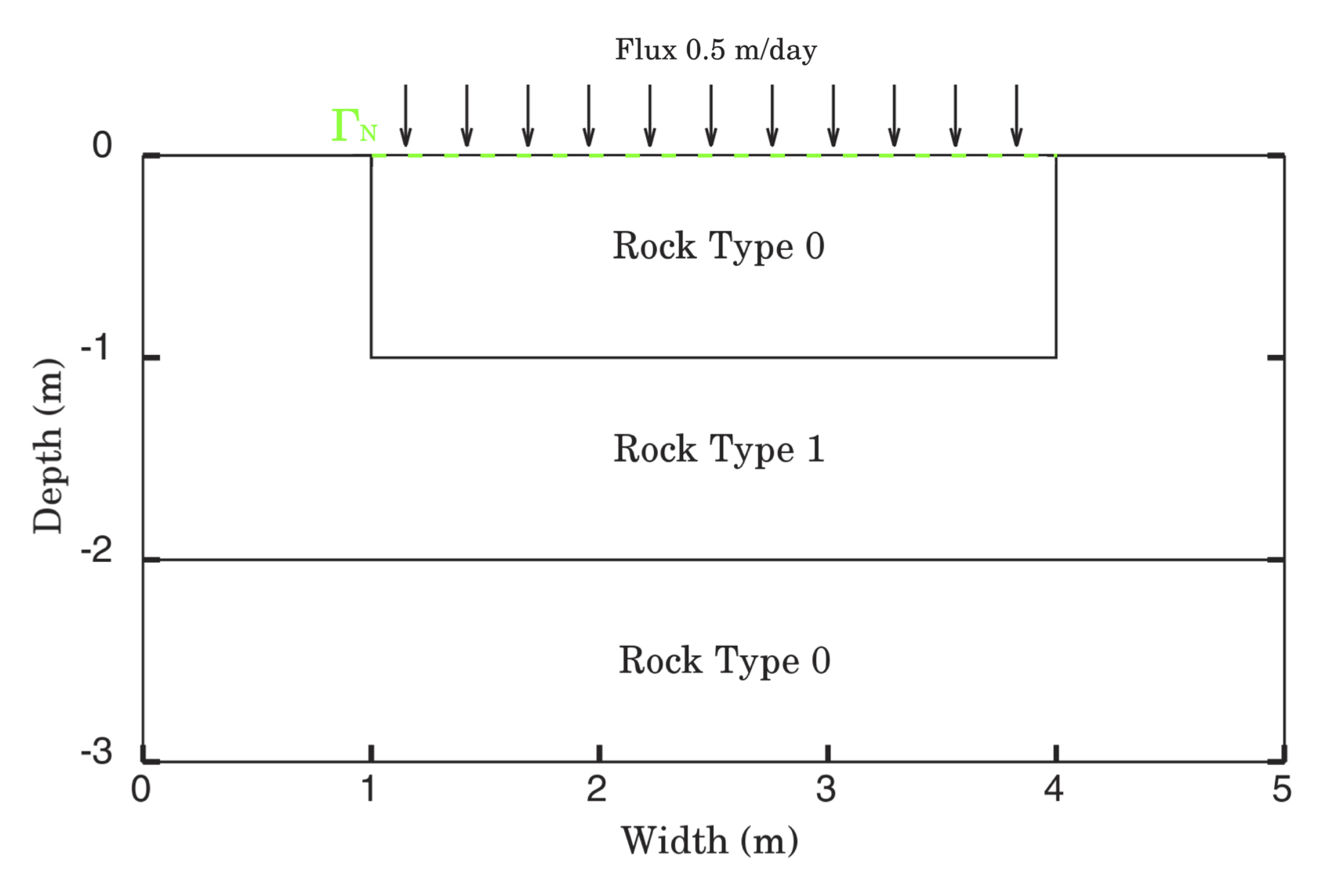}
    \caption{Boundary condition for the filling case}
    \label{fig:fillingTest}
\end{figure}

{The test case follows the following dynamics. Water starts invading the void porous space in $\Omega_1$. When it reaches the interface with $\Omega_3$, capillarity involves a suction force on water from $\Omega_1$ to $\Omega_3$. Since clay (RT1) has low permeability, water encounters difficulties to progress within $\Omega_3$. This yields a front moving downward in $\Omega_1$ which is stiffer for the Brooks-Corey model than for the van Genuchten-Mualem one. In both cases, the simulation is stopped before water reaches the bottom part corresponding to $\Omega_2$.
In Figure \ref{img:solutionFillBC} we can observe the evolution of the saturation profile during the simulation performed on a $50\times30$ cells mesh with Brooks and Corey model, whereas the evolution corresponding to van Genuchten-Mualem nonlinearities is depicted in Figure~\ref{img:solutionFillVG}.} 
    \newline
    \begin{minipage}{0.05\textwidth}
        \includegraphics[height=6cm]{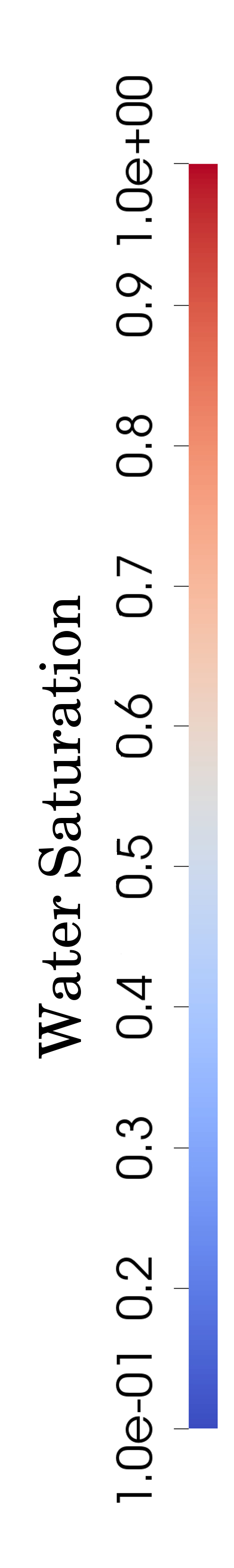}
    \end{minipage}
    \begin{minipage}{0.9\textwidth}
    \begin{center}
        \includegraphics[width=4.4cm]{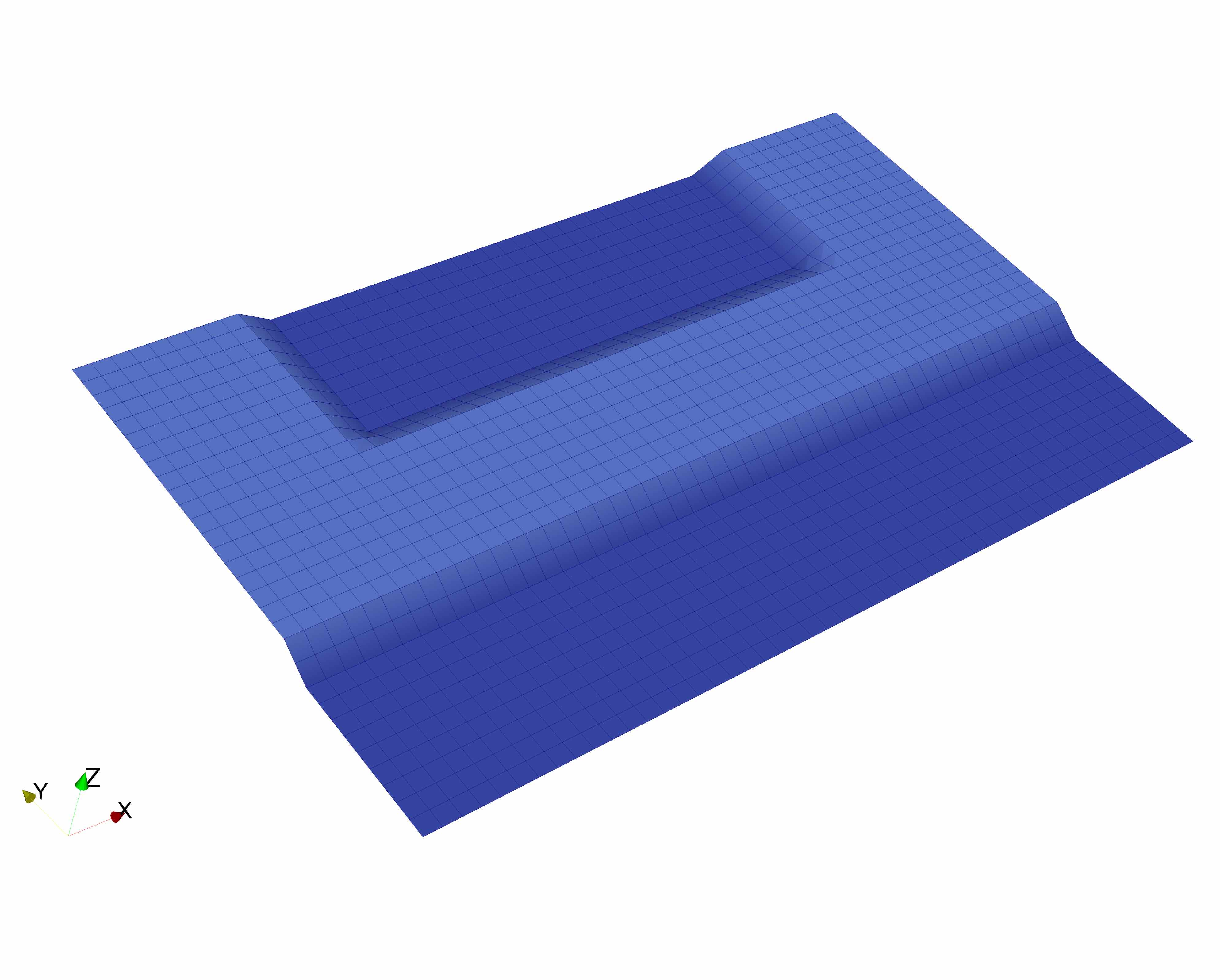}
        \includegraphics[width=4.4cm]{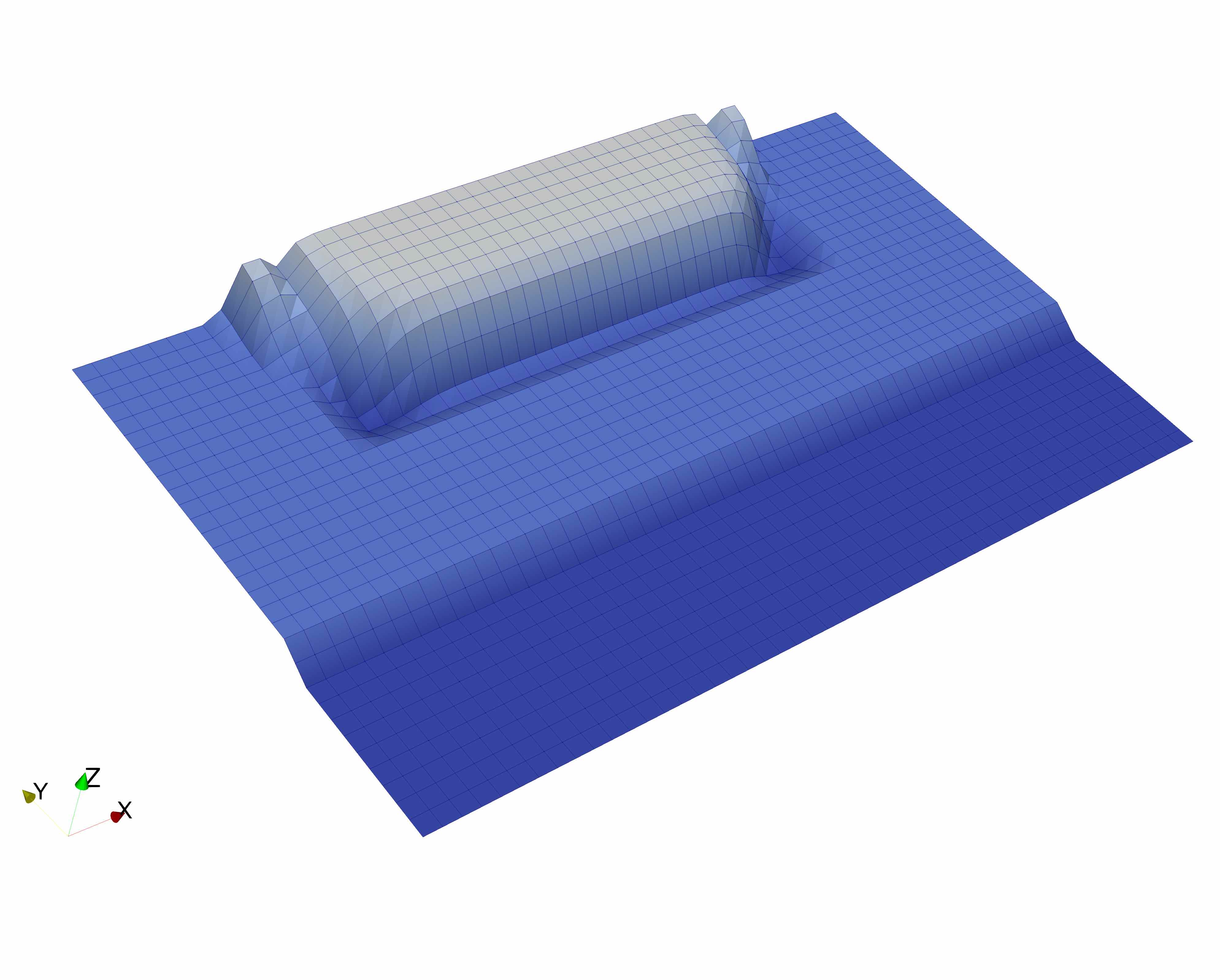}
        \includegraphics[width=4.4cm]{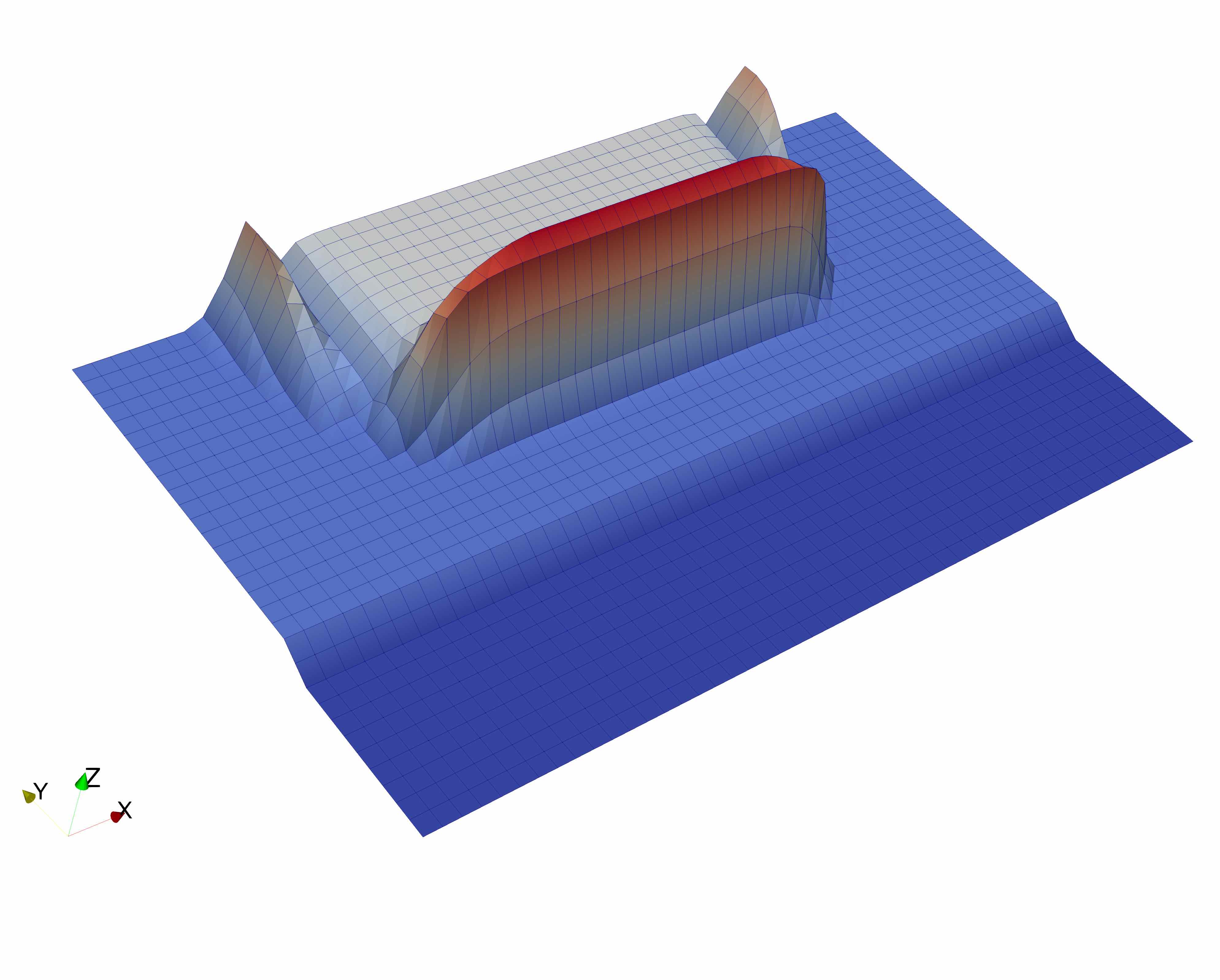}
        \includegraphics[width=4.4cm]{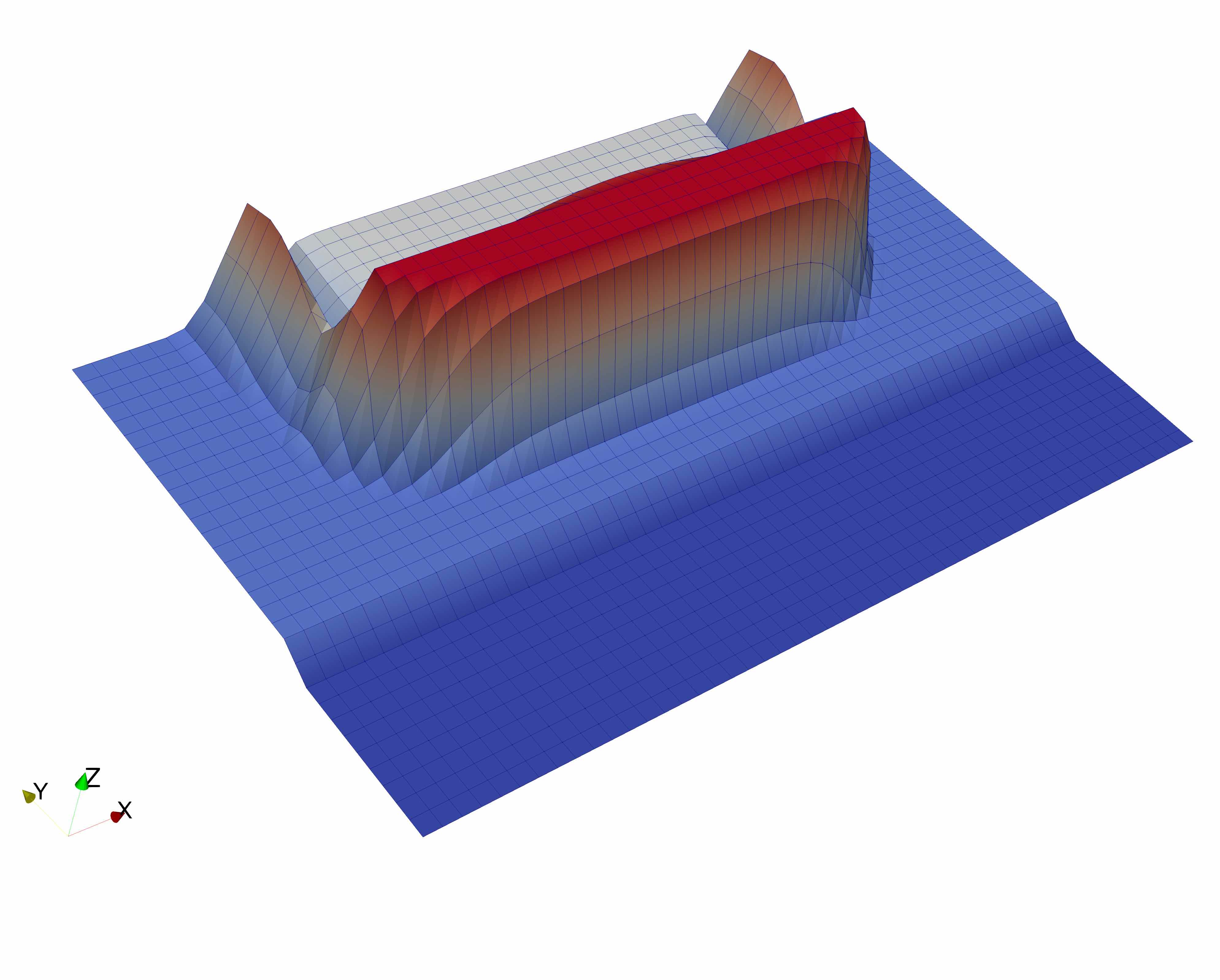}
        \includegraphics[width=4.4cm]{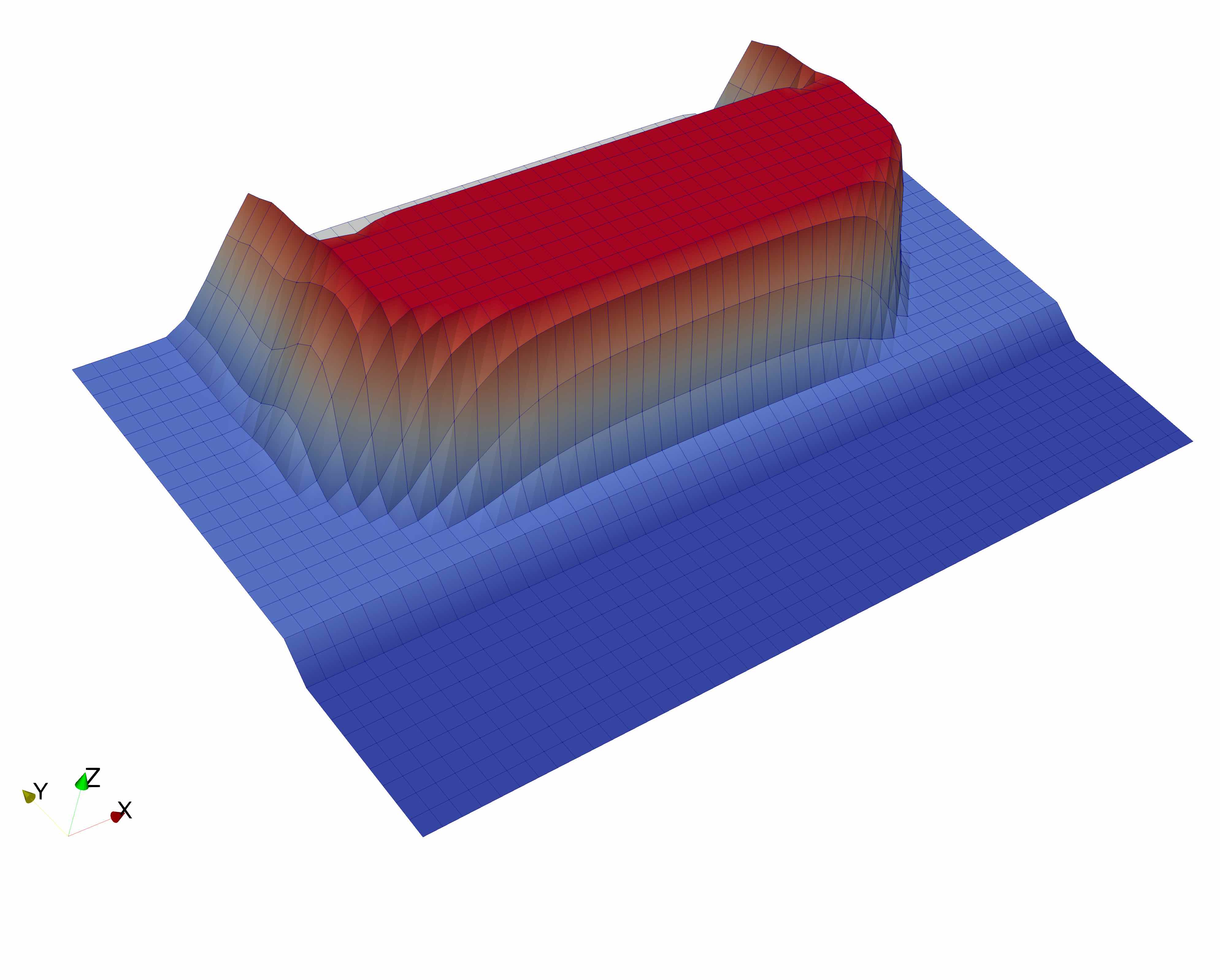}
        \captionof{figure}{Evolution of the saturation profile for $t\in\{0s,20\cdot 10^3s, 40\cdot 10^3s, 60\cdot10^3s, 86\cdot10^3s\}$ for filling case, using Brooks and Corey model, Method A and the $50\times30$ cells mesh.}
        \label{img:solutionFillBC}
        \end{center}
    \end{minipage}

      \begin{minipage}{0.05\textwidth}
        \includegraphics[height=6cm]{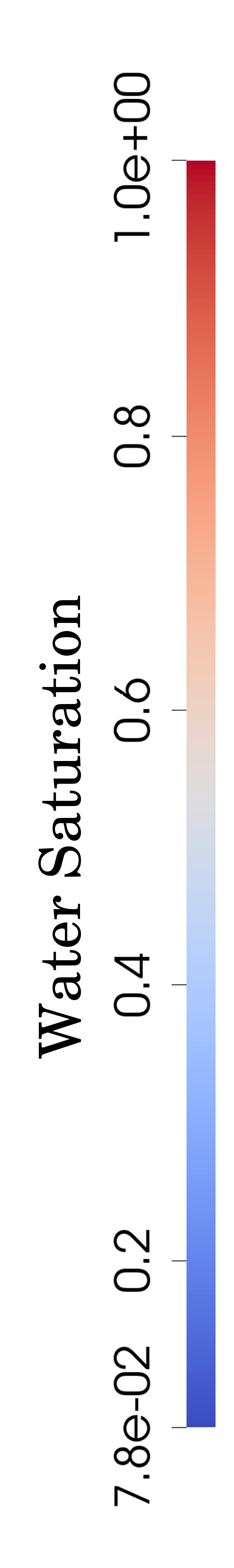}
    \end{minipage}
    \begin{minipage}{0.9\textwidth}
    \begin{center}
        \includegraphics[width=4.4cm]{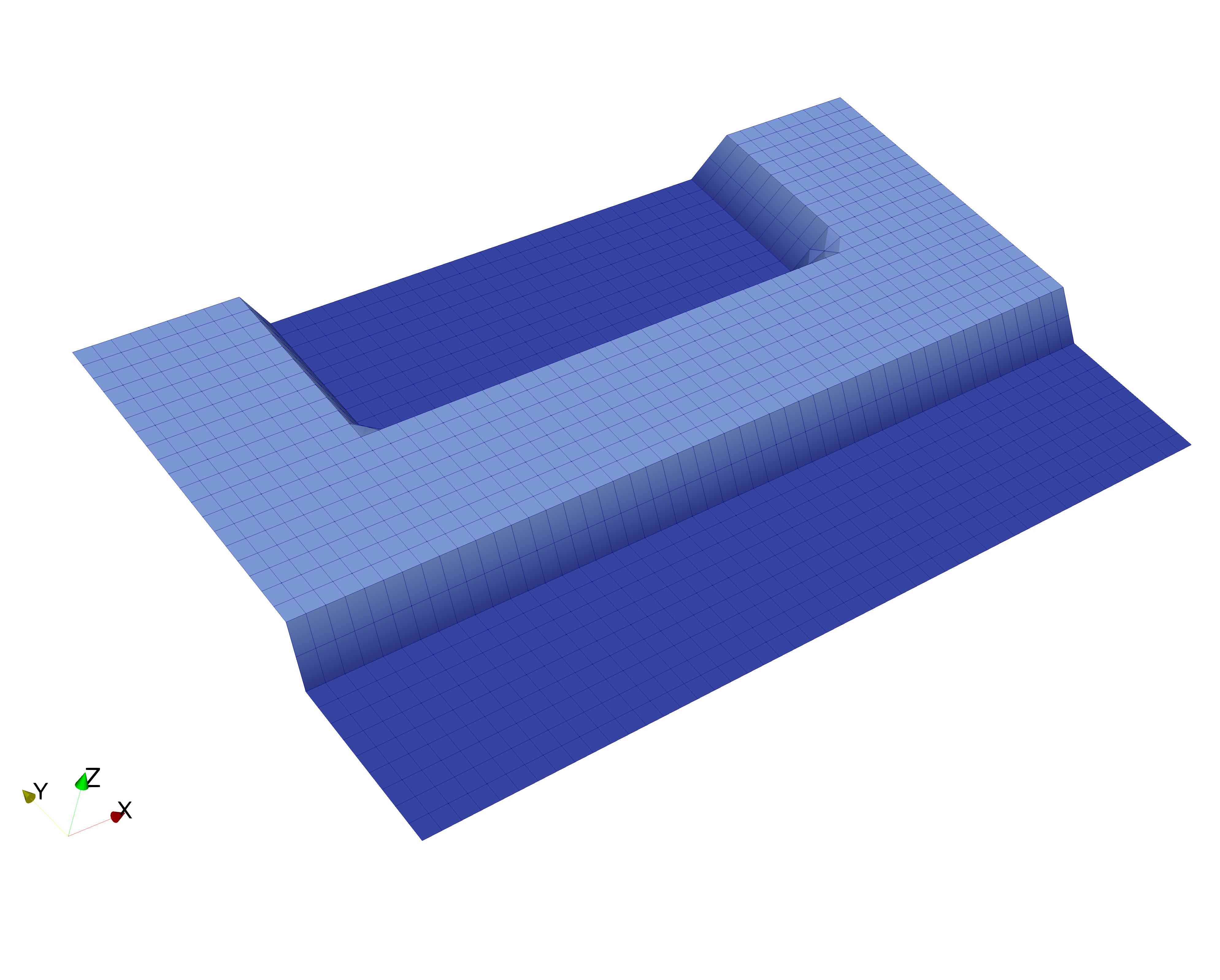}
        \includegraphics[width=4.4cm]{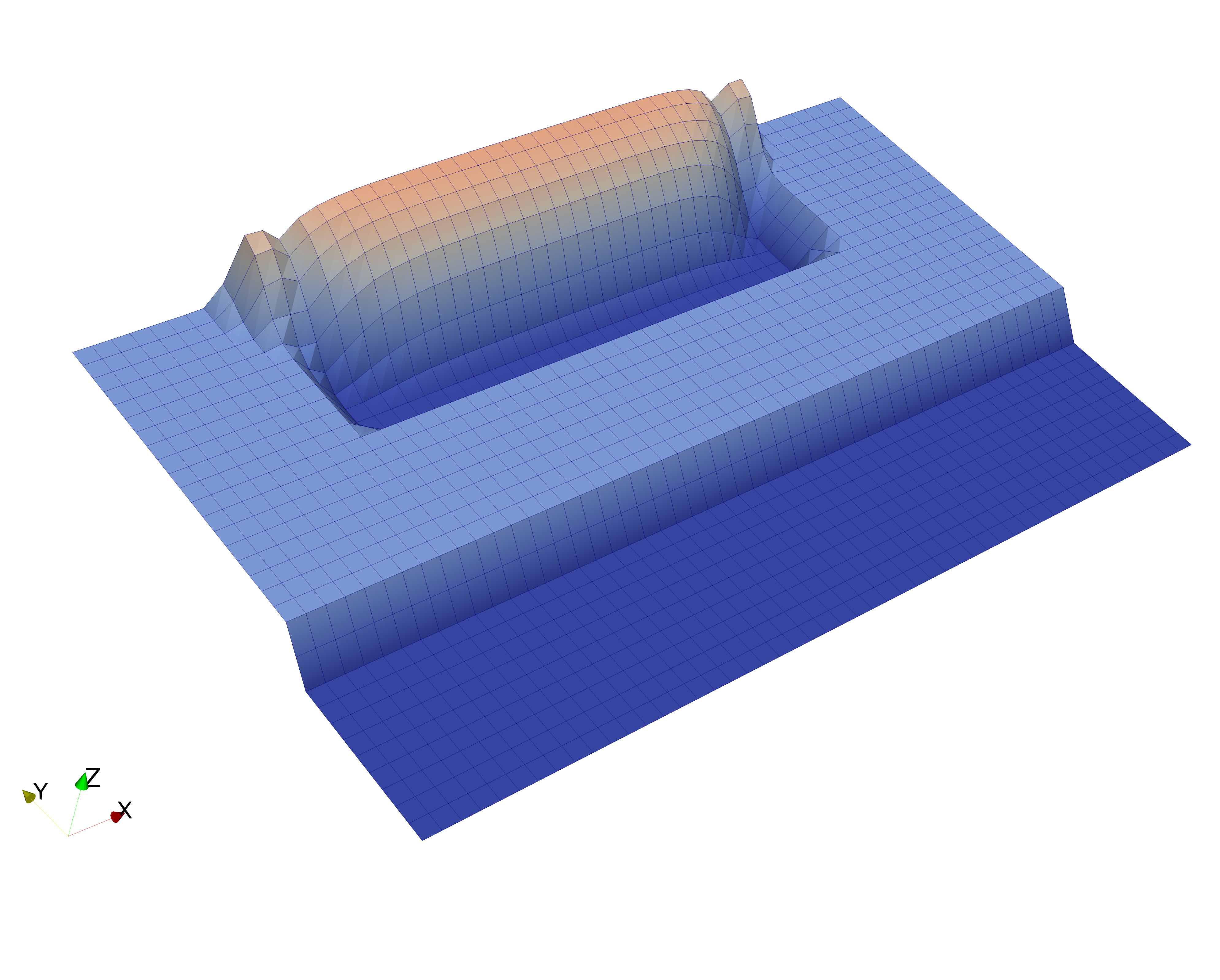}
        \includegraphics[width=4.4cm]{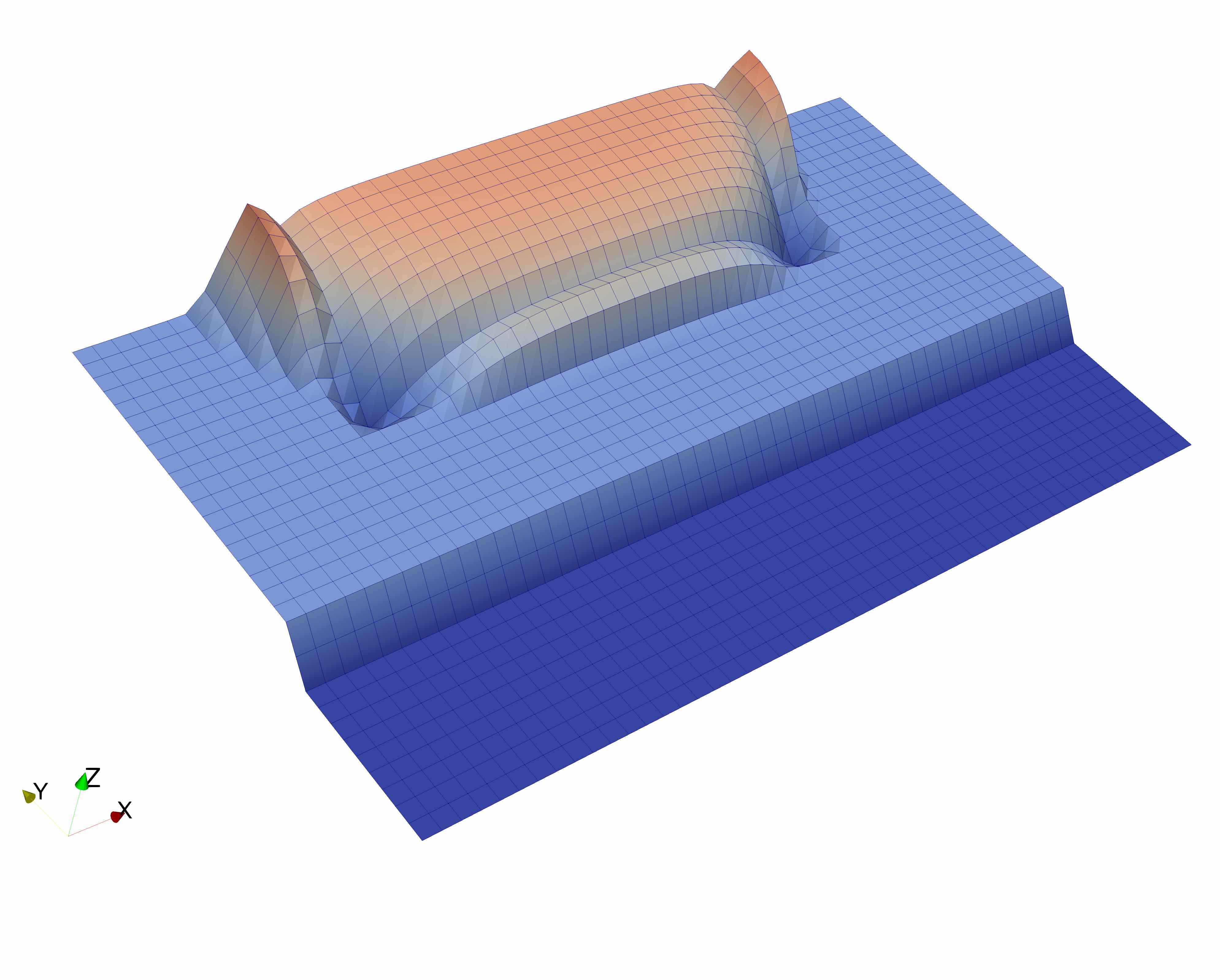}
        \includegraphics[width=4.4cm]{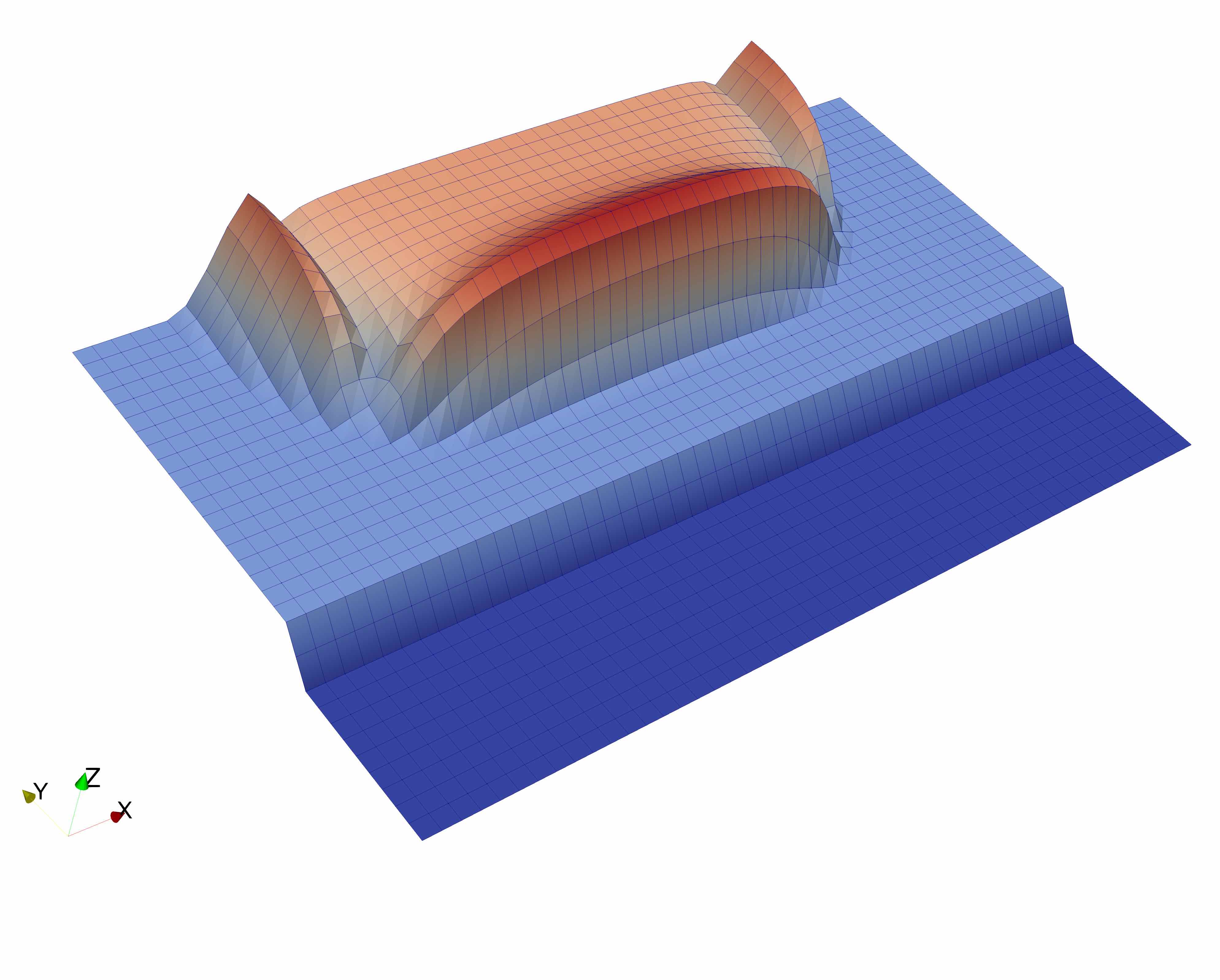}
        \includegraphics[width=4.4cm]{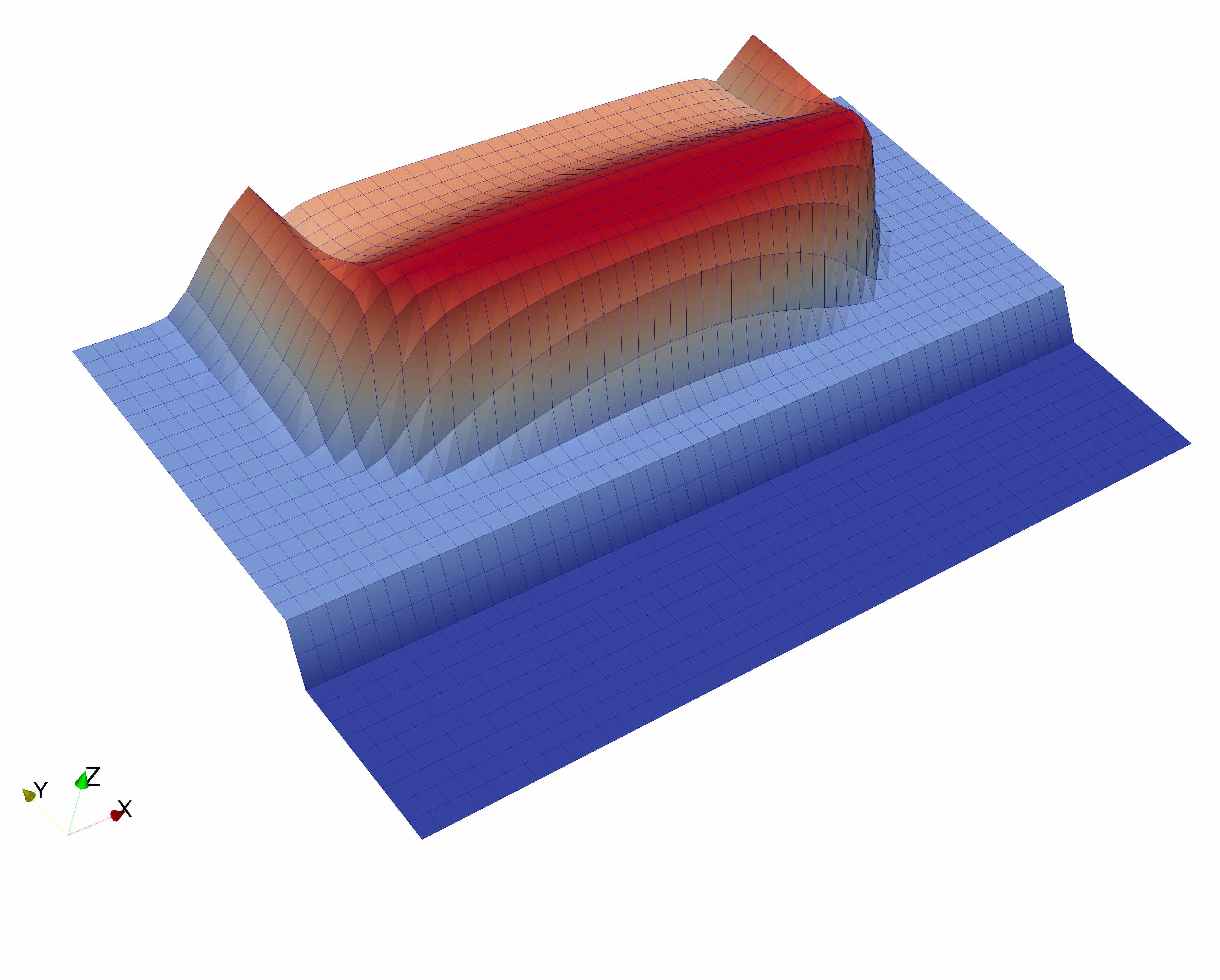}
        \captionof{figure}{Evolution of the saturation profile for $t\in\{0s,20\cdot 10^3s, 40\cdot 10^3s, 60\cdot10^3s, 86\cdot10^3s\}$ for filling case using Van Genuchten model, Method A and the $50\times30$ cells mesh.}
        \label{img:solutionFillVG}
        \end{center}
    \end{minipage}

\subsubsection{{Drainage case}}
{This test case is designed as a two-dimensional extension of a one-dimensional test case proposed by \cite{MD99} and addressed in \cite{McBride06,CZ10}.}
We simulate a vertical drainage starting from initially and boundary saturated conditions during $105 \cdot 10^4~\mathrm{s}$. At the initial time, the pressure varies with depth with $p^0(z)=-\rho g z$. A Dirichlet boundary condition $p_D=0~\textrm{Pa}$ is imposed on the bottom of the domain, more precisely on $\Gamma_D=\{(x,y) \,| \, x\in[0\mathrm{m},5\mathrm{m}], y=-3\mathrm{m}\}$. The layers' composition of $\Omega$ is reported in Figure \ref{fig:drainageTest}. For this simulation an homogeneous time-step $\Delta t= 2000\mathrm{s}$ is used for the test with the Brooks-Corey model and $\Delta t= 800\mathrm{s}$ for the one with the van Genuchten-Mualem model.

\begin{figure}[htp]
    \centering
    \includegraphics[height=5cm]{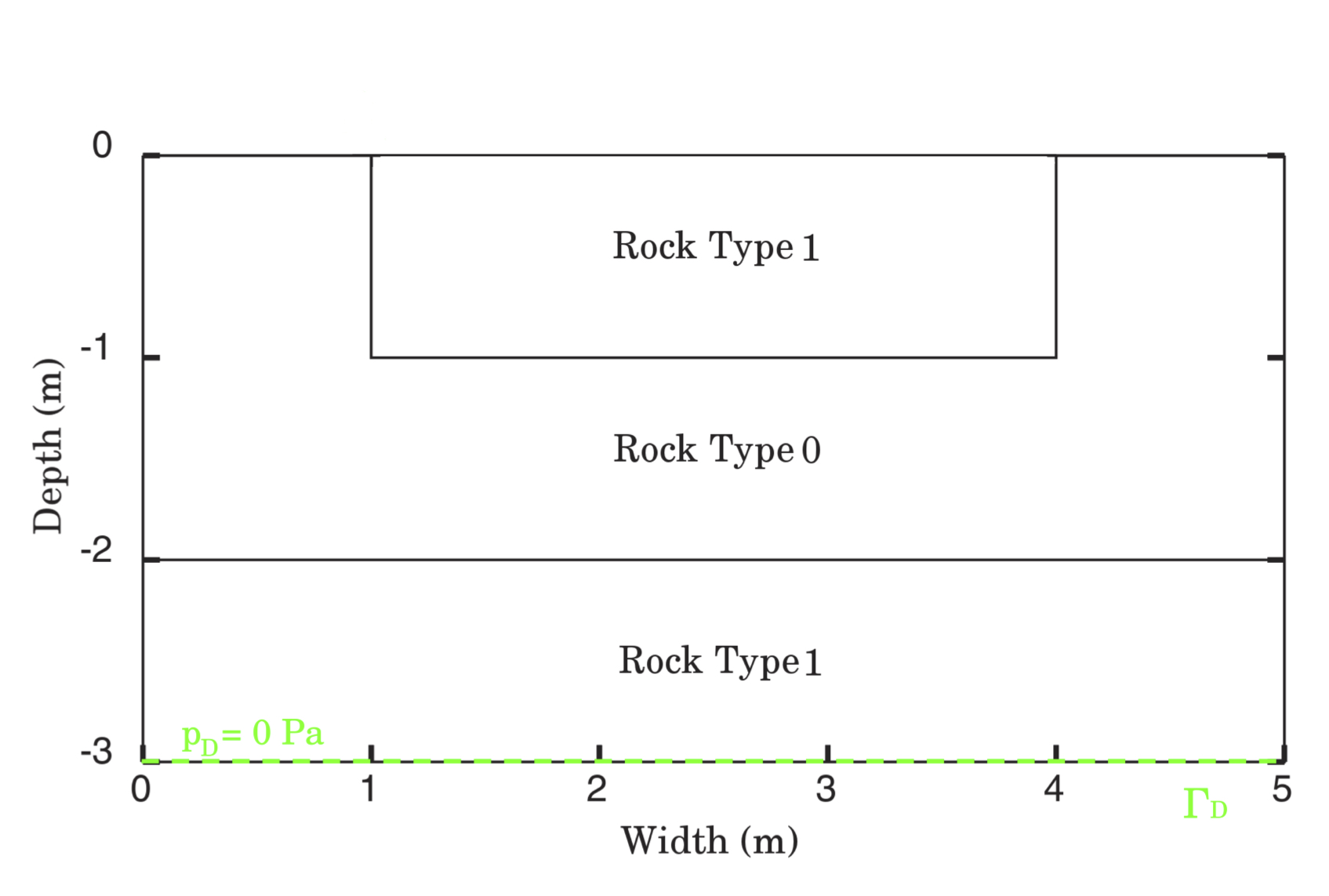}
    \caption{Boundary condition for the drainage test}
    \label{fig:drainageTest}
\end{figure}

{At the top interface between $\Omega_1$ and $\Omega_3$, capillarity acts in opposition to gravity and to the evolution of the system into a dryer configuration. The interface between $\Omega_2$ and $\Omega_3$ acts in the reverse way: suction accelerates the gravity driven drainage of RT0.}

In Figure \ref{img:solutionDryBC} we can observe the evolution of the saturation profile during the simulation performed on a $50\times30$ cells mesh with Brooks and Corey model, whereas the evolution corresponding to van Genuchten-Mualem nonlinearities is depicted in Figure~\ref{img:solutionDryVG}. \newline
     \begin{minipage}{0.025\textwidth}
    \vspace{-1.8cm}
        \includegraphics[height=5cm]{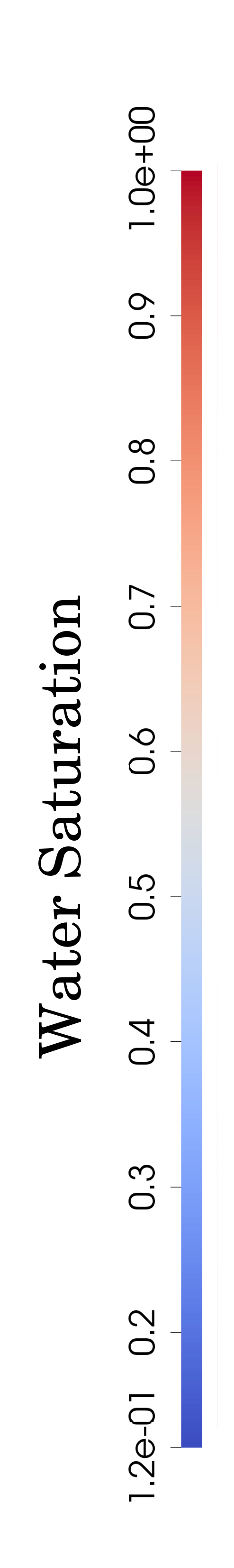}
    \end{minipage}
    \begin{minipage}{0.95\textwidth}
    \begin{center}
        \includegraphics[width=4.4cm]{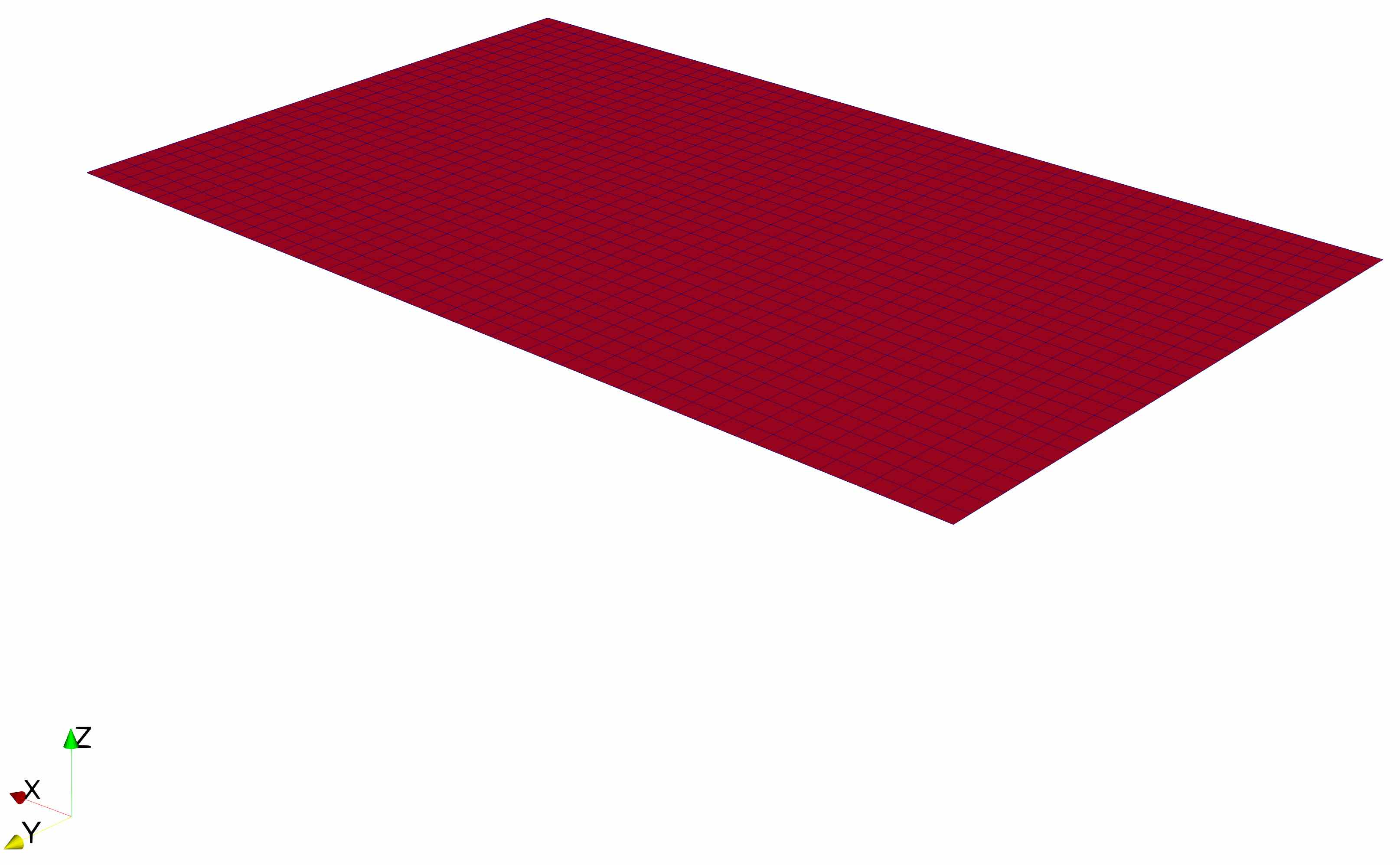}
        \includegraphics[width=4.4cm]{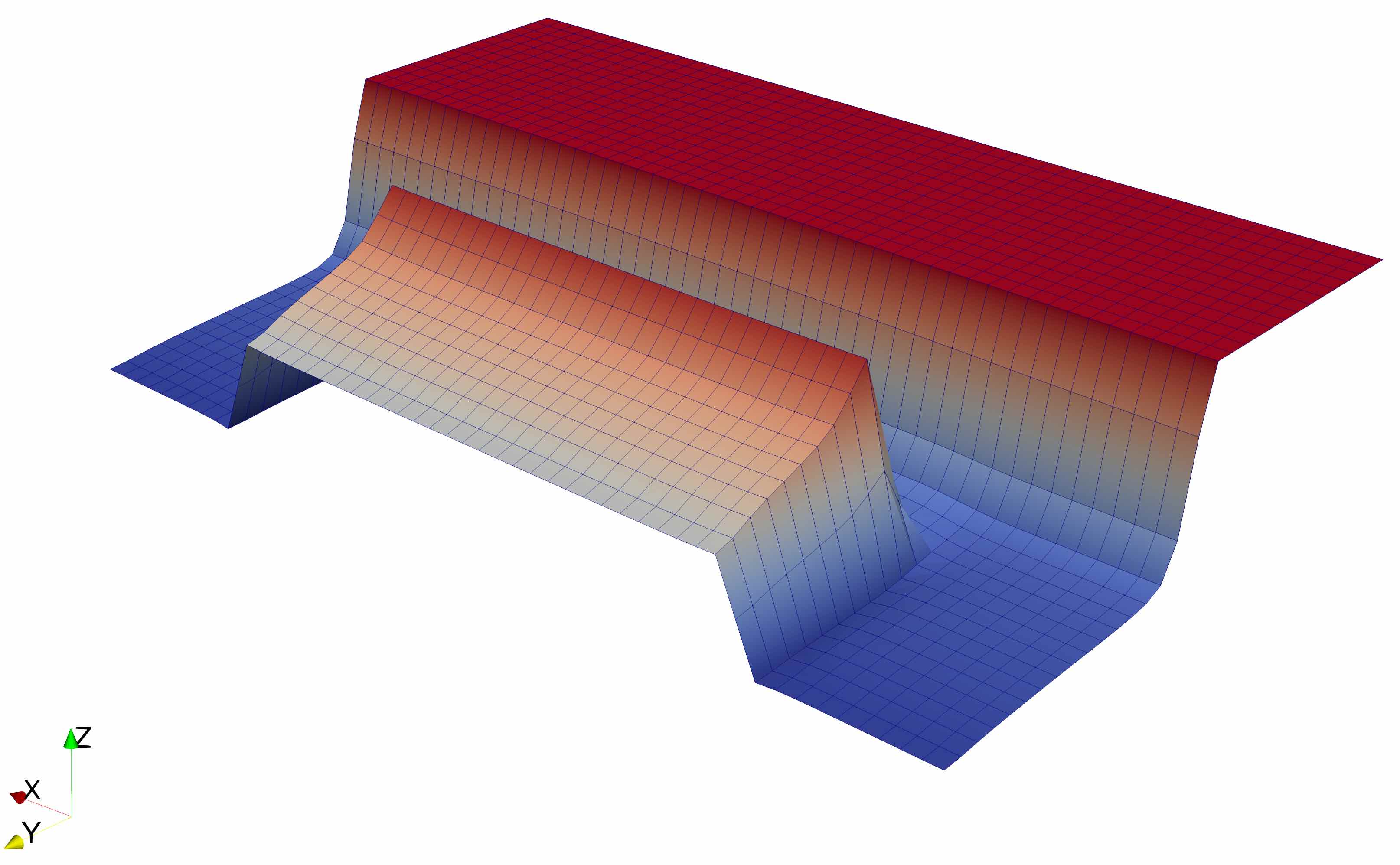}
        \includegraphics[width=4.4cm]{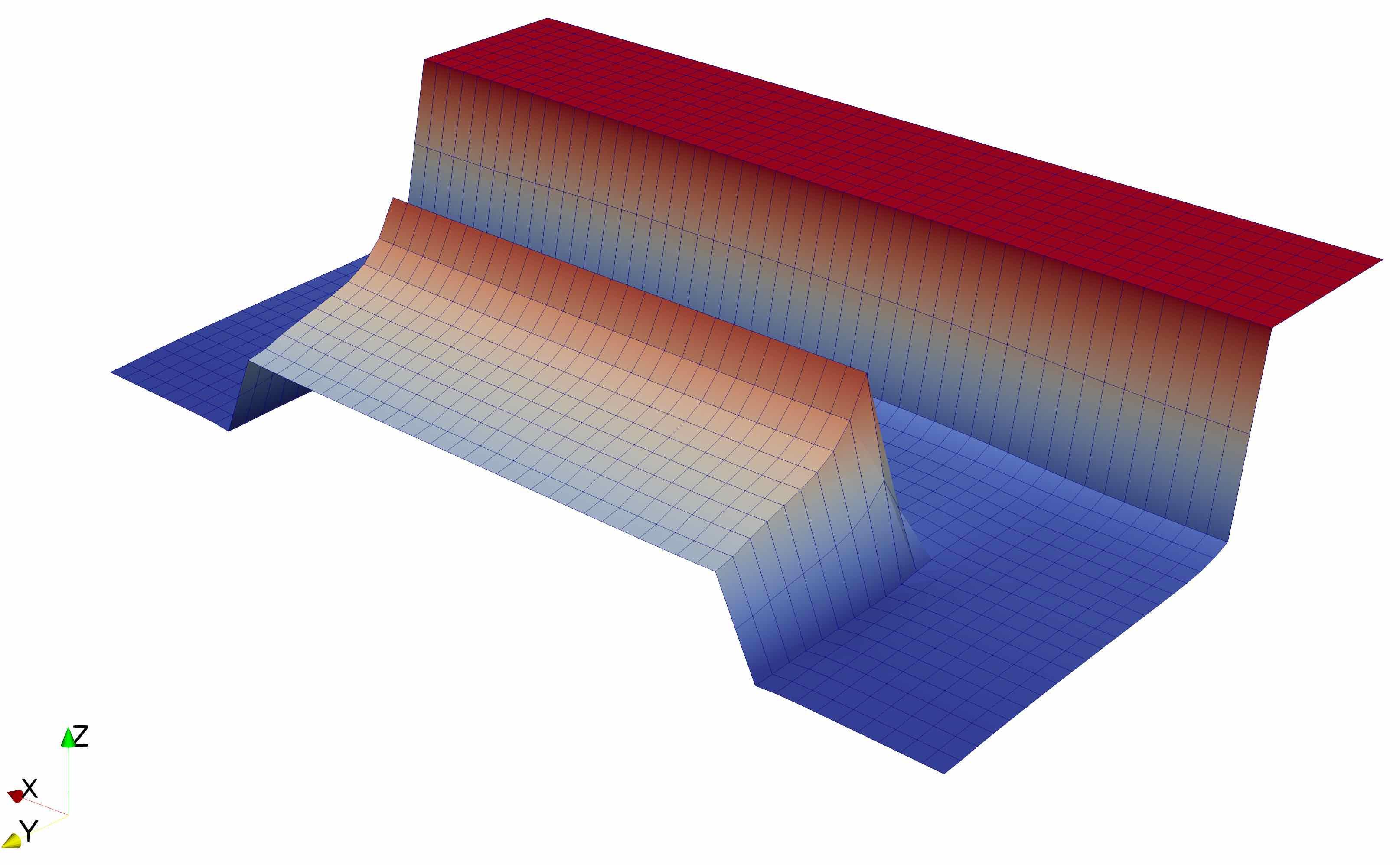}
        \includegraphics[width=4.4cm]{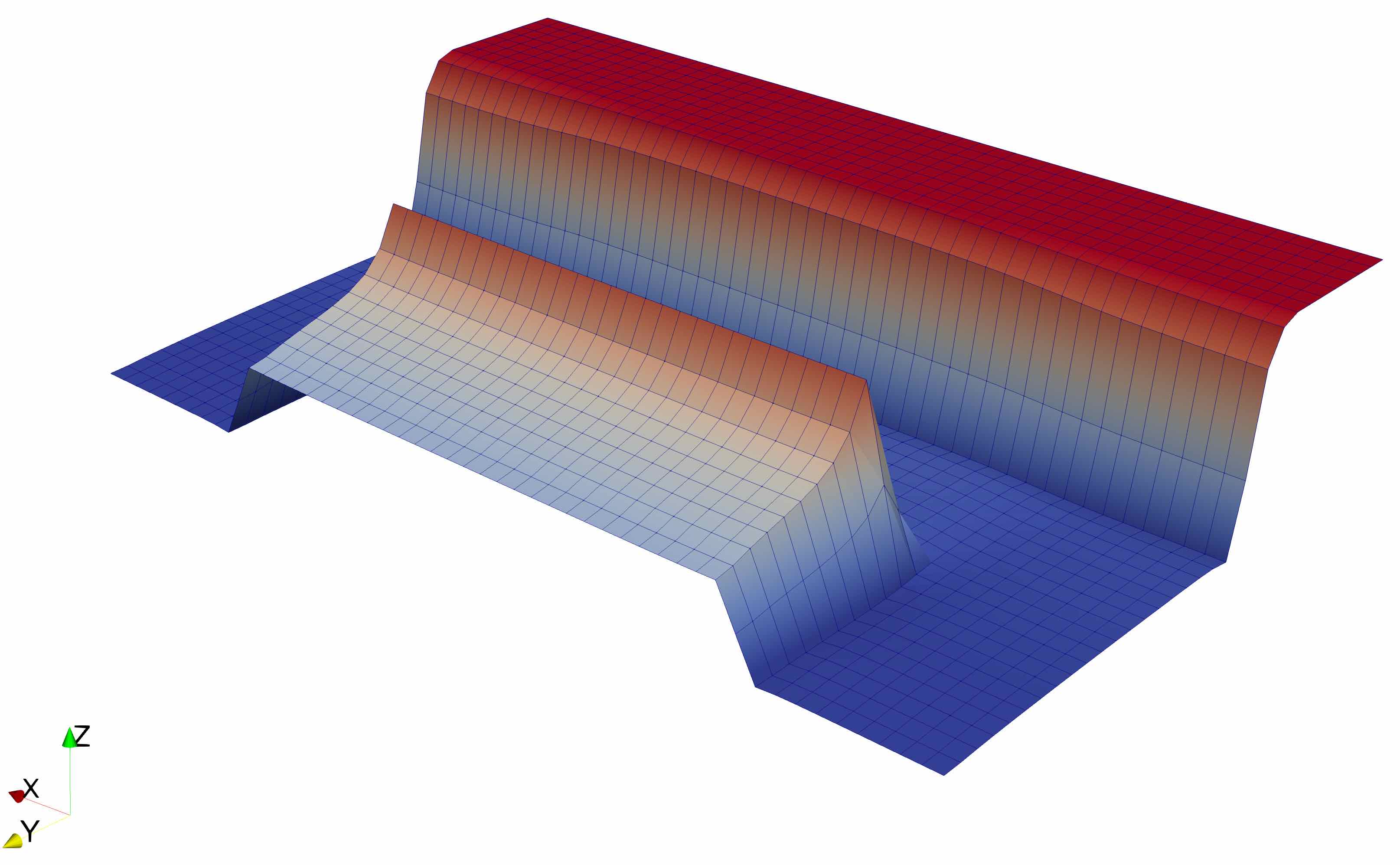}
        \includegraphics[width=4.4cm]{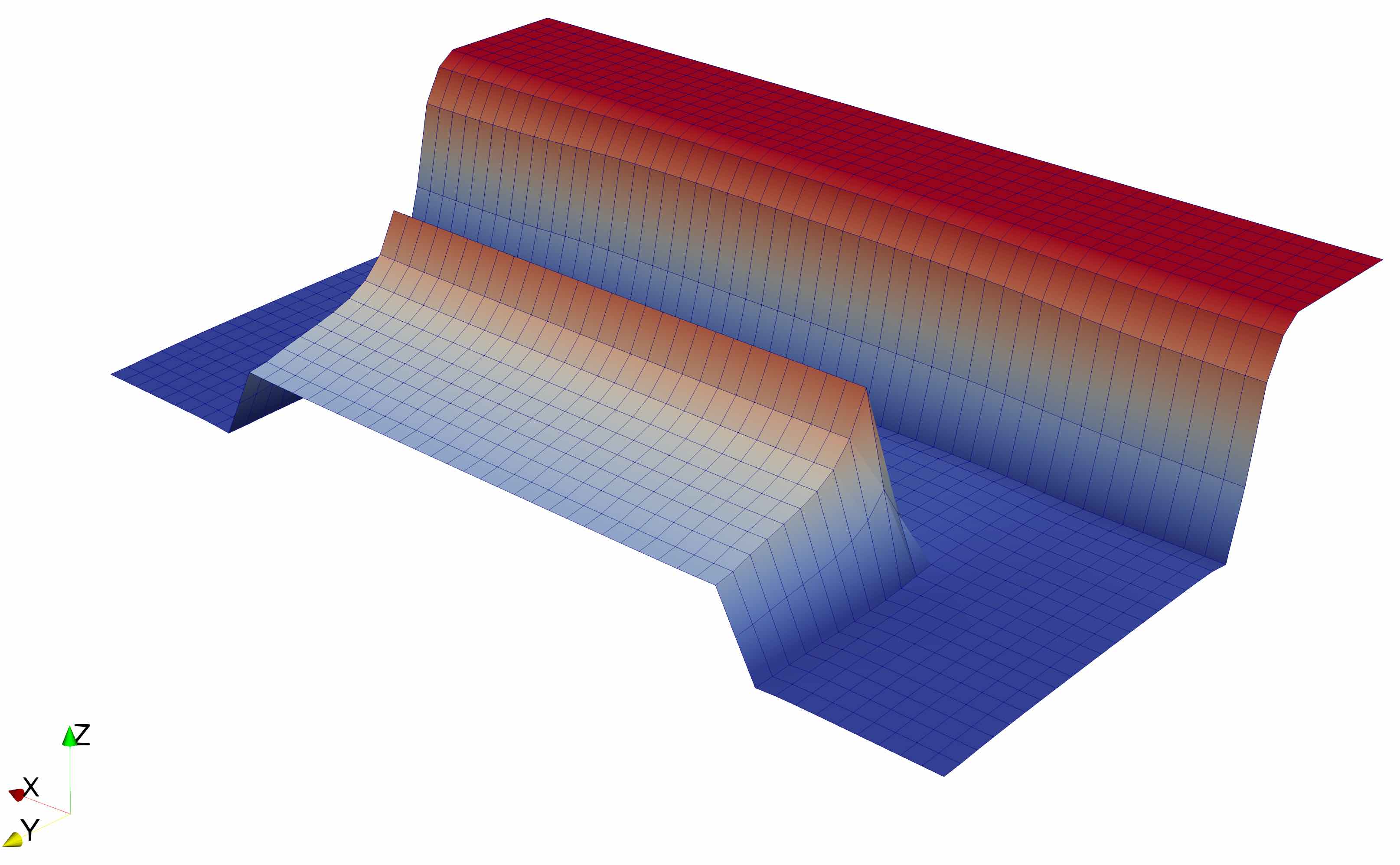}
        \captionof{figure}{Evolution of the saturation profile for $t\in\{0s,26.2\cdot 10^4s, 52.4\cdot 10^4s, 78.6\cdot10^4s, 105\cdot10^4s\}$ for drying case, using Brooks and Corey model, Method A and the $50\times30$ cells mesh.}
        \label{img:solutionDryBC}
        \end{center}
    \end{minipage}
    
\vspace{1cm}
    
    \begin{minipage}{0.05\textwidth}
    \vspace{-1.8cm}
        \includegraphics[height=5cm]{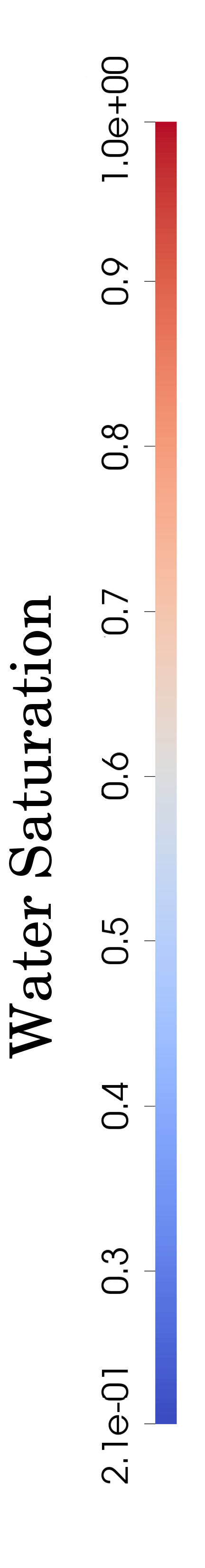}
    \end{minipage}
    \begin{minipage}{0.95\textwidth}
    \begin{center}
        \vspace{-1cm}
        \includegraphics[width=4.4cm]{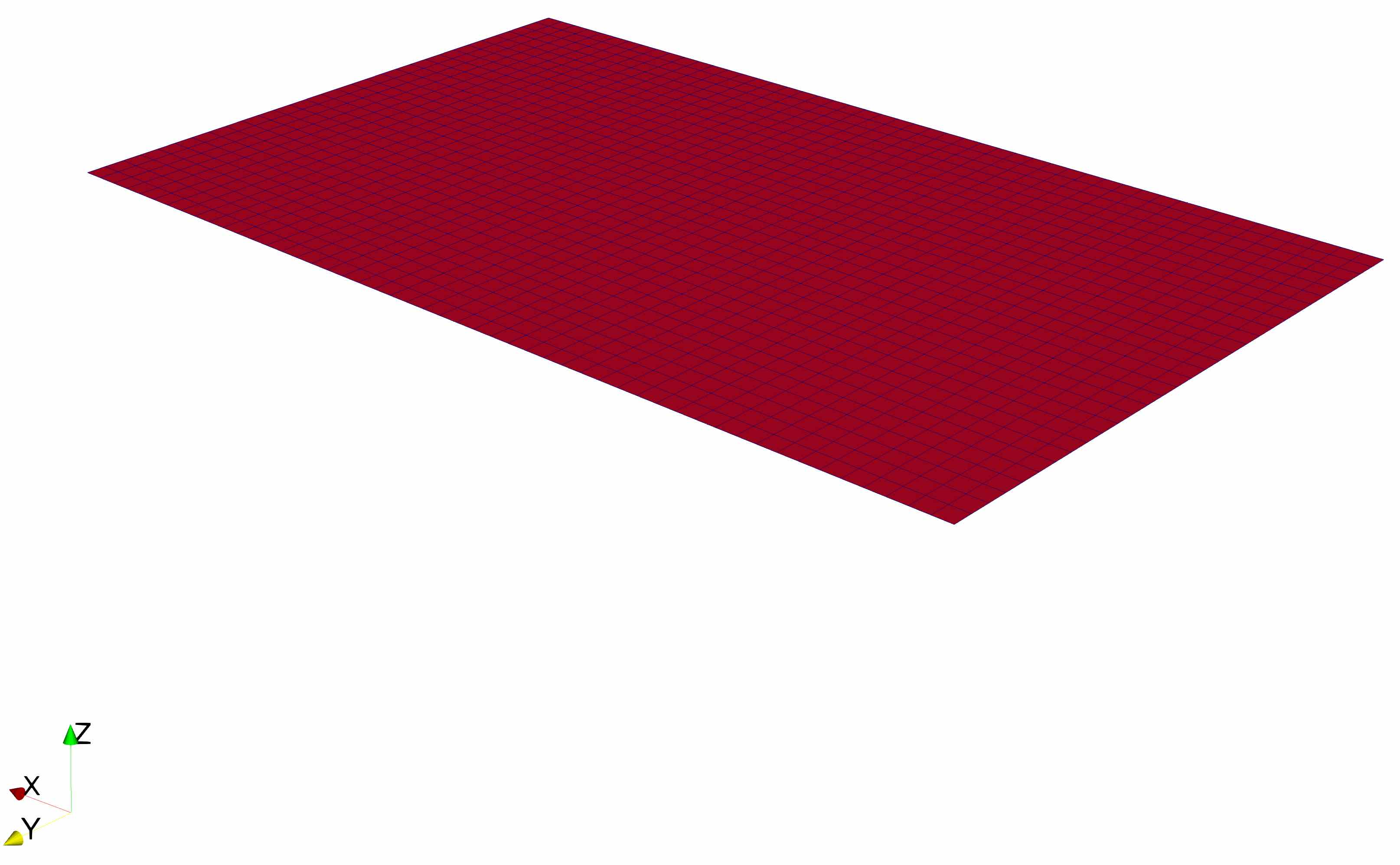}
        \includegraphics[width=4.4cm]{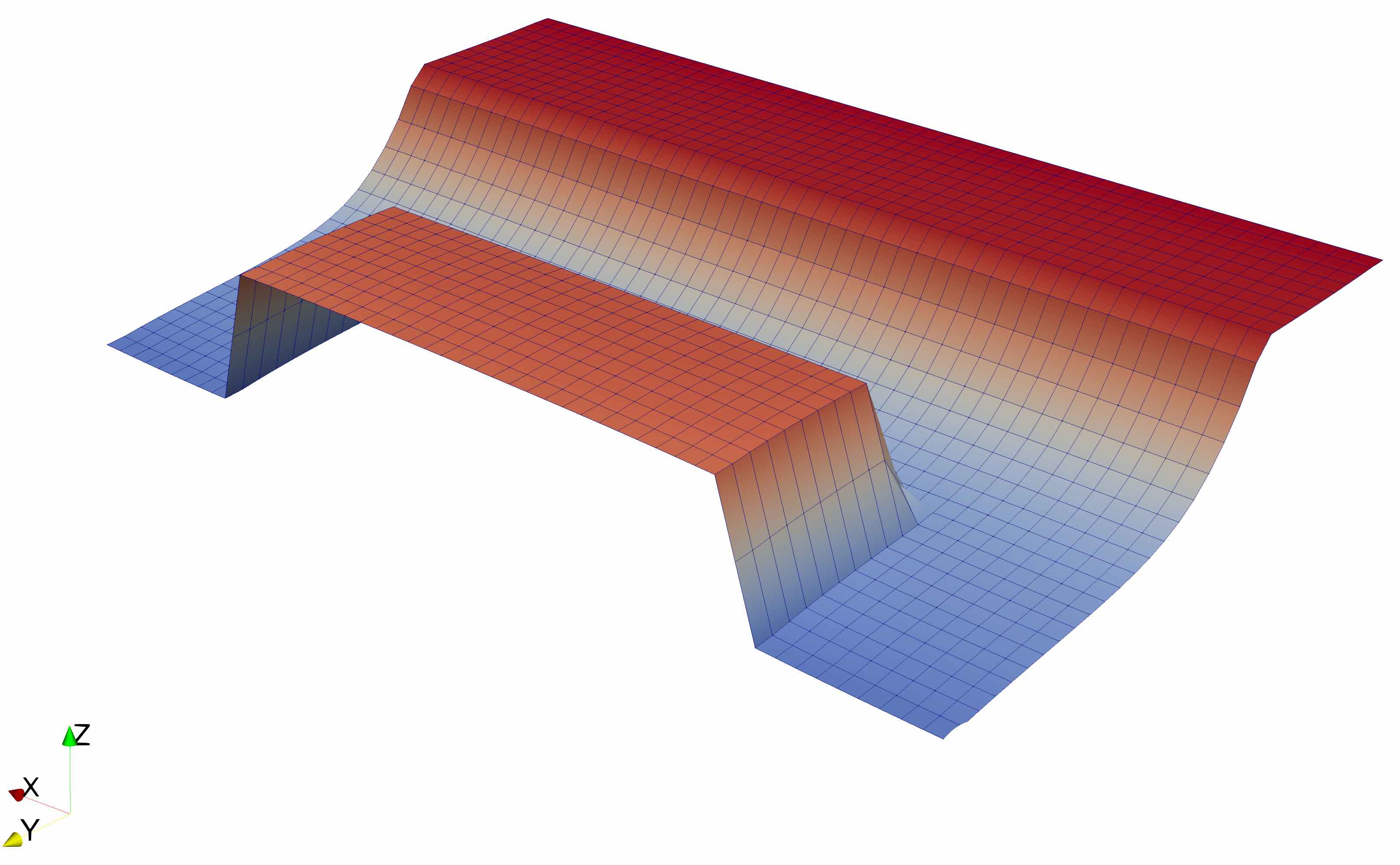}
        \includegraphics[width=4.4cm]{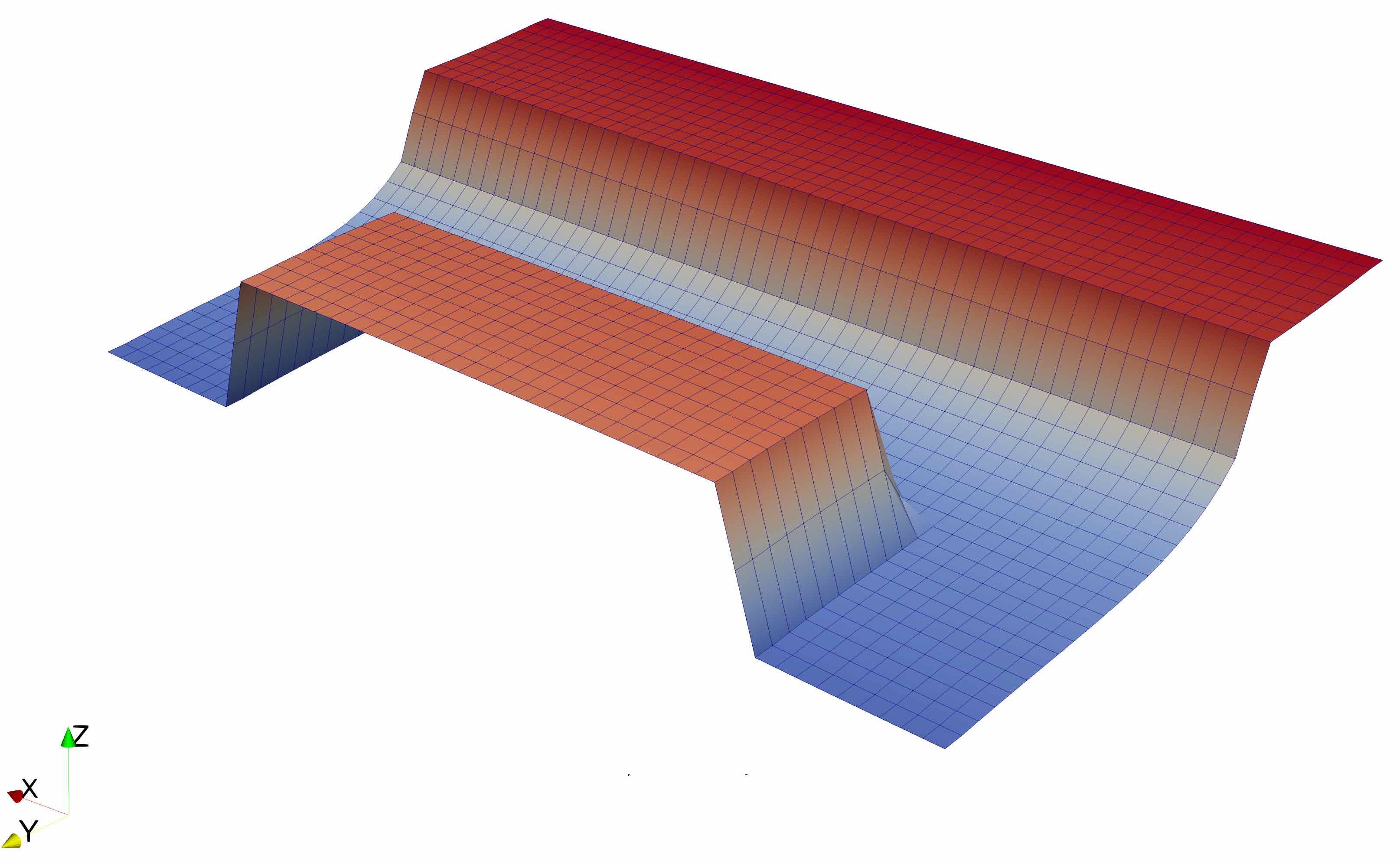}
        \includegraphics[width=4.4cm]{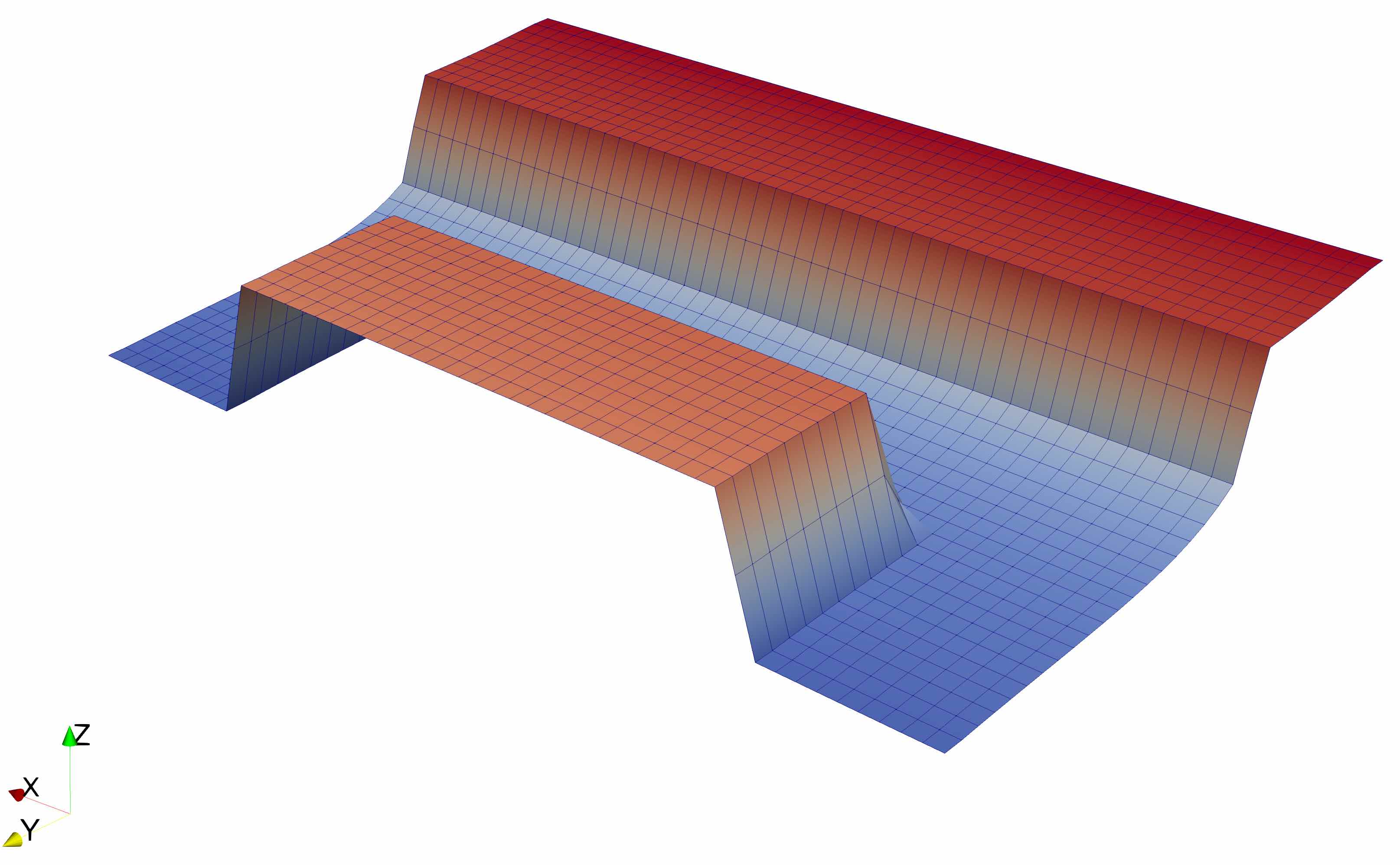}
        \includegraphics[width=4.4cm]{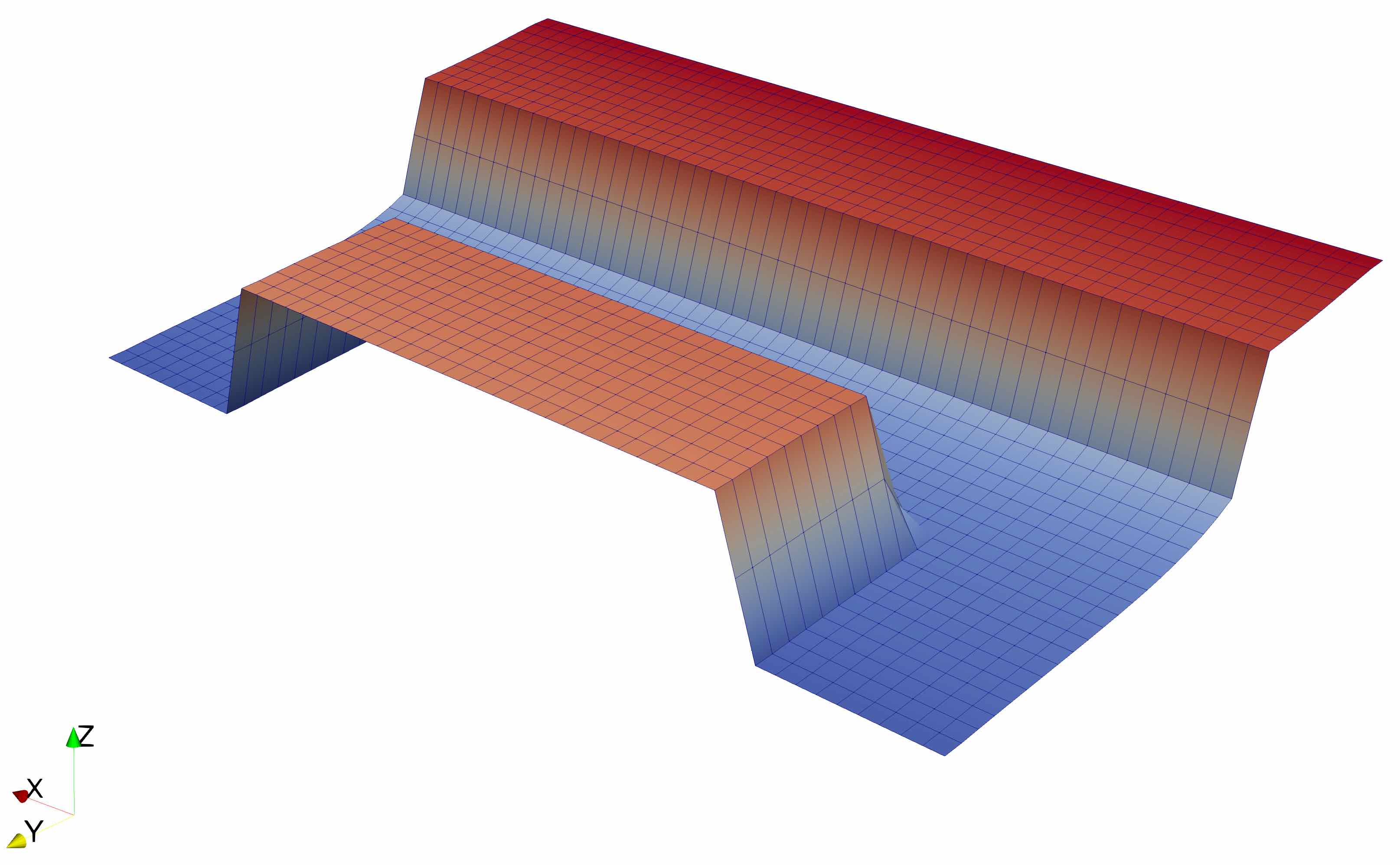}
        \captionof{figure}{Evolution of the saturation profile for $t\in\{0s,26.16\cdot 10^4s, 52.4\cdot 10^4s, 78.56\cdot10^4s, 105\cdot10^4s\}$  for drying case using Van Genuchten model, Method A and the $50\times30$ cells mesh.}
        \label{img:solutionDryVG}
        \end{center}
    \end{minipage}

\subsection{Comparisons {of the numerical treatments of the interfaces}}\label{ssec:num.results}
For each petro-physical model and configuration, a numerical convergence analysis is carried out for the schemes with (method A) or without (method B) thin cells, whose thickness is fixed to $\delta=10^{-6}\mathrm{m}$, at rock type interfaces. Five structured meshes with the following resolutions are considered for this analysis: $50 \times 30 $, $100 \times 60$, $200 \times 120$, 
$400 \times 240$, $800 \times 480$. The evolution of the error is measured using the $L^2([0,T], \Omega)$-norm of the relative difference between the saturations obtained on a given mesh and a reference solution obtained with Method A and the mesh $800\times 480$. The number of Newton iterations obtained with both methods is also compared.

\subsubsection{Brooks-Corey model: drainage case}
For the drainage case with the Brooks-Corey model, the convergence error is given in Figure \ref{img:convBCdry}. First we notice that, for all meshes, the error is smaller with method A than with method B and that we have a linear rate of convergence with the first one whereas this rate is smaller with the latter one.
The total, average and maximal number of Newton iterations are also given in Table \ref{table:newtonBCdry}. Method A appears to be slightly more expensive.

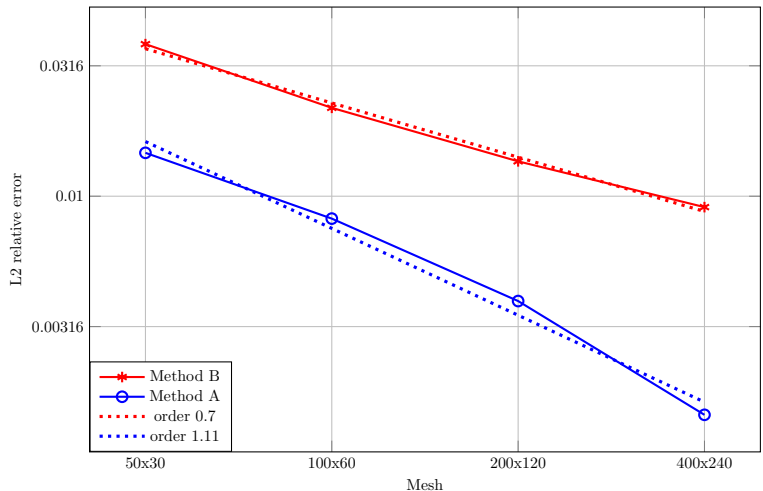
\begin{figure}[htb]
\centering
\begin{tikzpicture}
\begin{axis}[grid=major,
width=0.7\textwidth,
height=0.5\textwidth,
xlabel={Mesh},
ylabel={L2 relative error},
ymode=log,
xmode=log,
xticklabels={50x30, 100x60, 200x120, 400x240},
xtick={1500,6000,24000, 96000},
cycle list name=MyCyclelist1,
legend style={at={(0,0)},anchor=south west},nodes={scale=0.6, transform shape}]
\addplot table[x=mesh, y=L2]{errorBCdry.txt};
\addlegendentry{Method B };
\addplot table[x=mesh, y=L2]{errorBCdryV2.txt};
\addlegendentry{Method A };
\addplot[red,dotted,domain=15e+2:96e+3,line width= 1.2pt] {0.460214823/x^0.34553081};
\addlegendentry{order $0.7$};
\addplot[blue,dotted,domain=15e+2:96e+3,line width= 1.2pt] {0.92503662/x^0.553083133};
\addlegendentry{order $1.11$};
\end{axis}
\end{tikzpicture}
\caption{$L^2(Q_T)$ relative error in saturation for the drainage case using Brooks and Corey model.}\label{img:convBCdry}
\end{figure}

\begin{table}[htp!]
\centering
\begin{tabular}{p{1.6cm}p{1.4cm}p{1.4cm}p{1.2cm}}
\hline
 & $\sharp$ total  & $\sharp$ avg  & $\sharp$ max  \\
\hline
Method A & $2038$ & $3$ & $29$ \\
Method B & $1927$ & $3$ & $29$ \\
\hline\\
\end{tabular}
\caption{Newton's iterations for the mesh $200 \times 120$ for the drainage case using Brooks and Corey model.}
\label{table:newtonBCdry}
\end{table}

    Let us now evaluate the saturation absolute error between results obtained with Method A and Method B. In Figure \ref{img:errDryBC} we plot the absolute-error distribution over the domain at three different times: when the cells line in $\Omega_1$ above the interface between $\Omega_1$ and $\Omega_3$ starts drying,  when the cells line in $\Omega_2$ below the interface between $\Omega_3$ and $\Omega_2$ starts drying and at final time. \newline 
    \begin{minipage}{0.03\textwidth}
    \vspace{-1.2cm}
        \includegraphics[height=3cm]{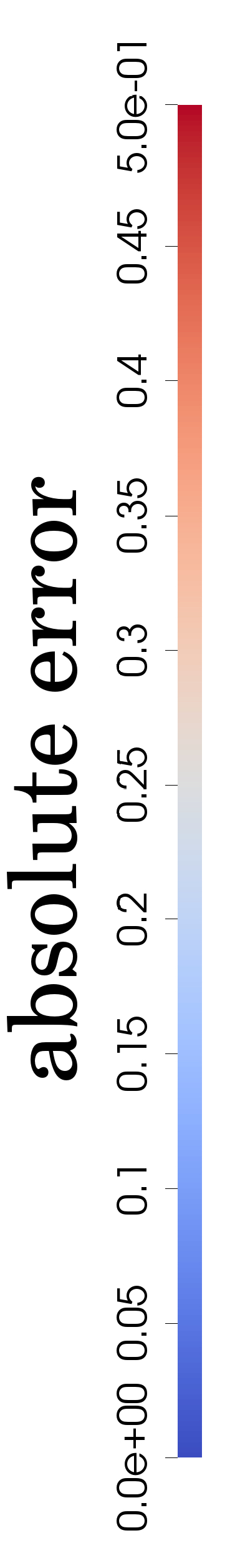}
    \end{minipage}
    \begin{minipage}{0.9\textwidth}
    \begin{center}
        \includegraphics[width=4.4cm]{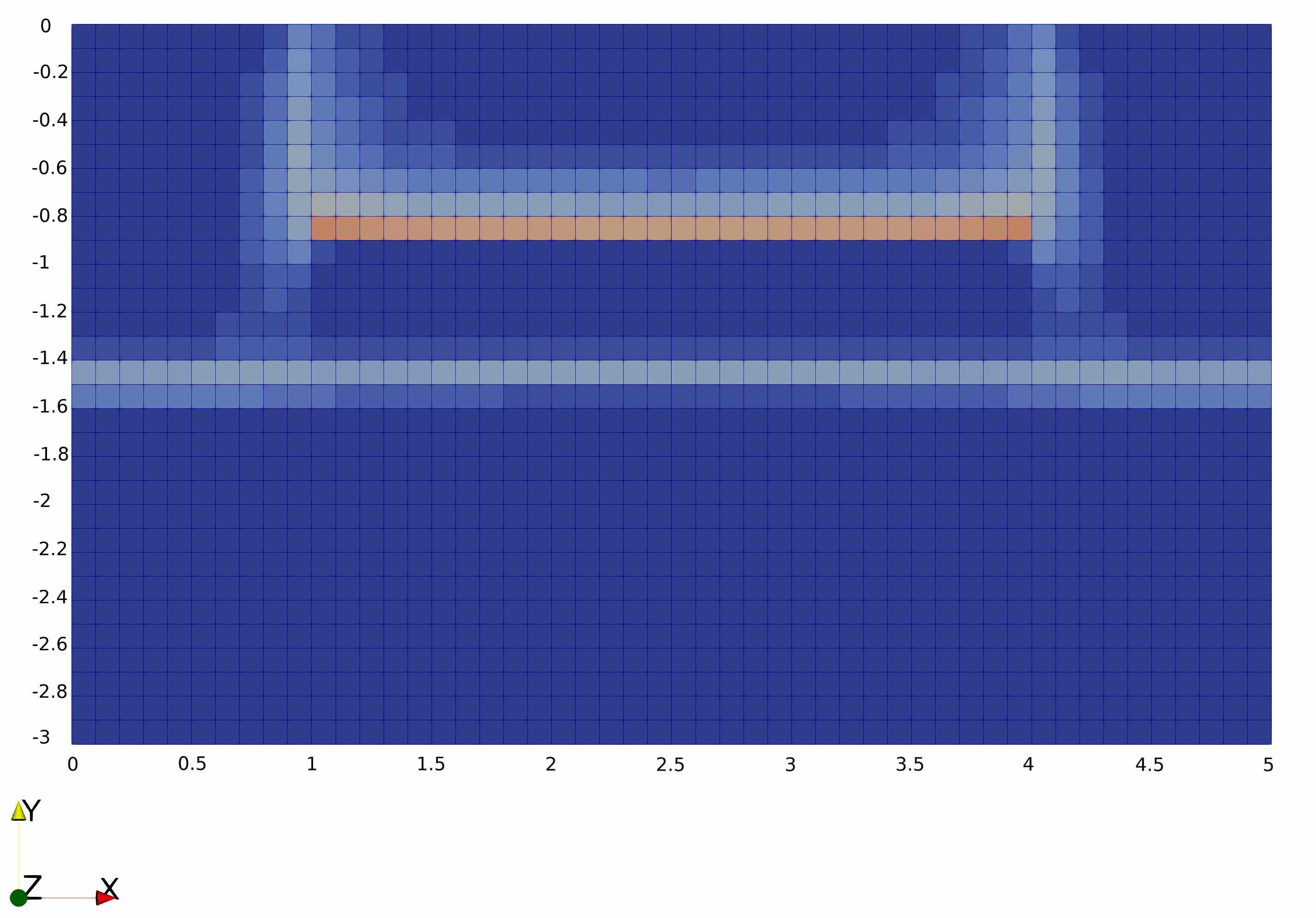}
        \includegraphics[width=4.4cm]{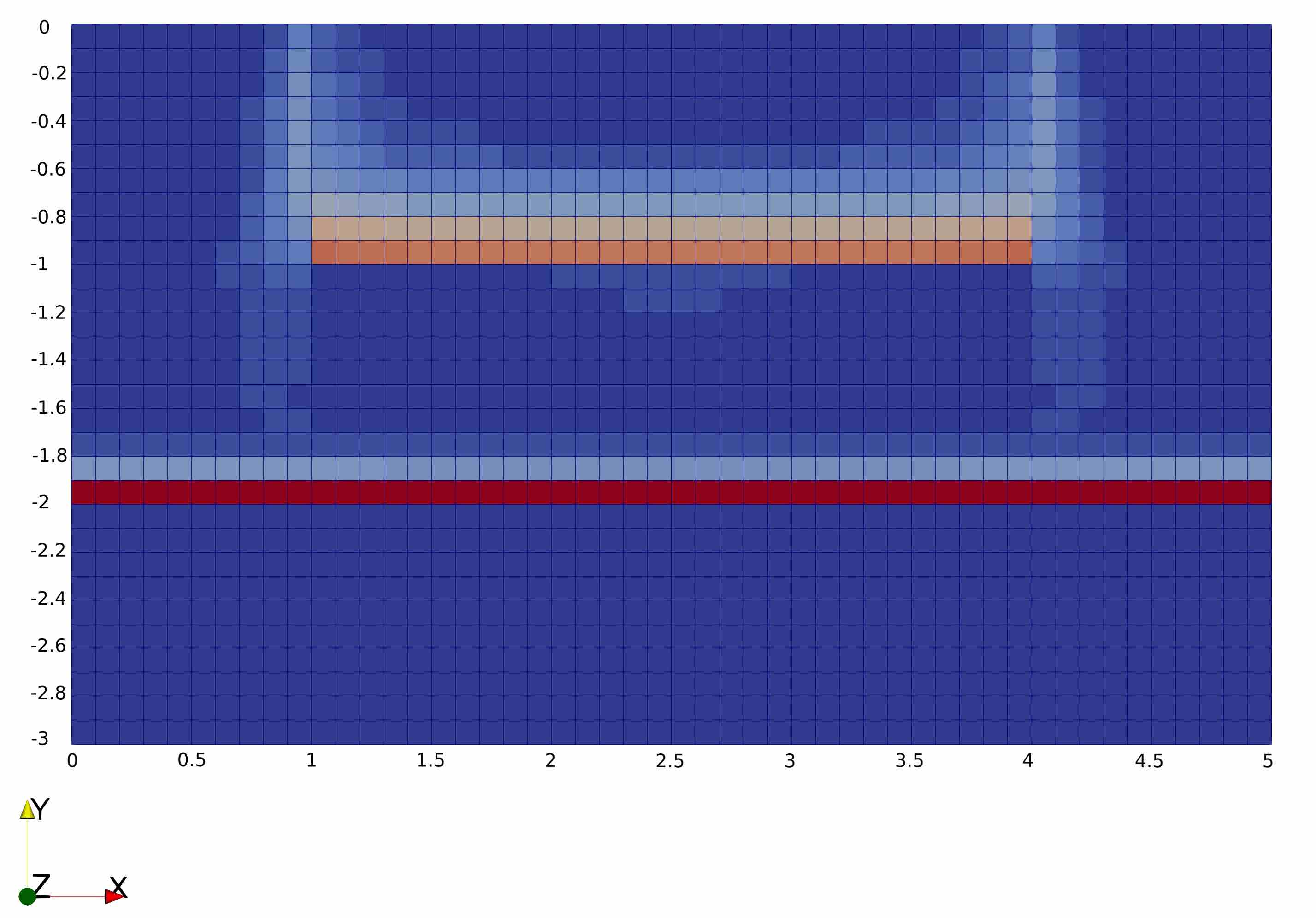}
        \includegraphics[width=4.4cm]{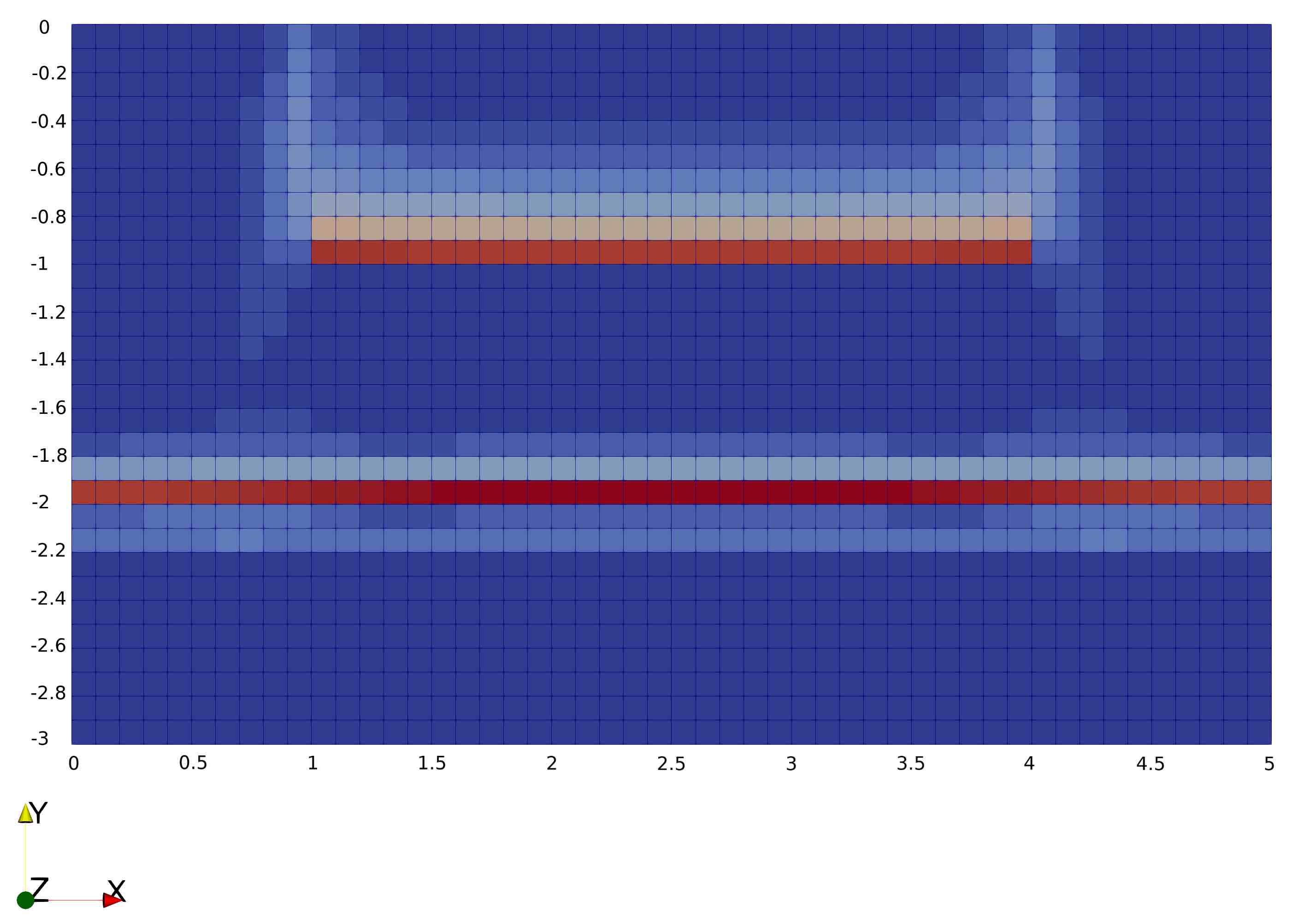}
        \captionof{figure}{Saturation absolute error between Method A and Method B for the drainage case with Brooks and Corey model at $t\in\{23\cdot 10^4s, 53.2\cdot 10^4s, 105\cdot 10^4s\}$.}
        \label{img:errDryBC}
    \end{center}
    \end{minipage}

\subsubsection{Brooks-Corey model: filling case}
For the filling case with the Brooks-Corey model, the convergence error is given in Figure \ref{img:convBCfill}. As for the previous case, Method A enables to recover a linear convergence rate. Except for the first two meshes where the error obtained with Method A is slightly larger, for all other meshes, this error is smaller than the one obtained with method B.
The total, average and maximal number of Newton iterations are given in Table \ref{table:newtonBCfill}. The algorithm behaves here in the same way as before.

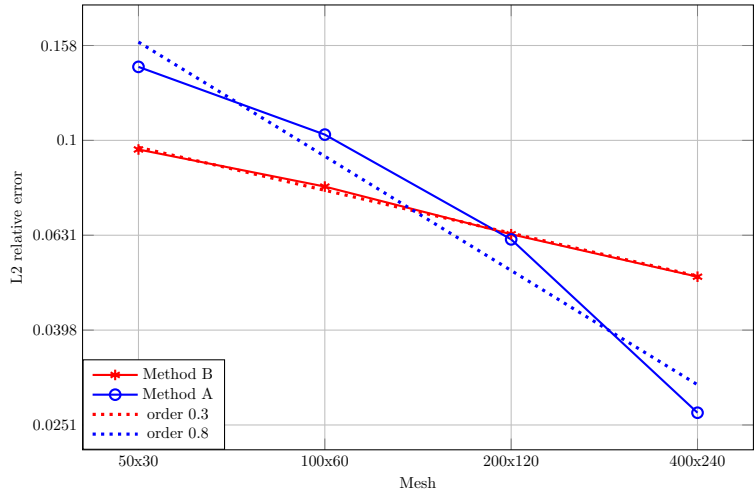
\begin{figure}[htp]
\centering
\begin{tikzpicture}
\begin{axis}[grid=major,
width=0.7\textwidth,
height=0.5\textwidth,
xlabel={Mesh},
ylabel={L2 relative error},
ymode=log,
xmode=log,
xticklabels={50x30, 100x60, 200x120, 400x240},
xtick={1500,6000,24000, 96000},
cycle list name=MyCyclelist1,
legend style={at={(0,0)},anchor=south west},nodes={scale=0.6, transform shape}]
\addplot table[x=mesh, y=L2]{errorBCfill.txt};
\addlegendentry{Method B };
\addplot table[x=mesh, y=L2]{errorBCfillv2.txt};
\addlegendentry{Method A };
\addplot[red,dotted,domain=15e+2:96e+3,line width= 1.2pt] {0.29025352/x^0.150333304};
\addlegendentry{order $0.3$};
\addplot[blue,dotted,domain=15e+2:96e+3,line width= 1.2pt] {3.0039635/x^0.399990833};
\addlegendentry{order $0.8$};
\end{axis}
\end{tikzpicture}
\caption{$L^2(Q_T)$ relative error in saturation for the filling case using Brooks and Corey model.}\label{img:convBCfill}
\end{figure}
\begin{table}[htp!]
\centering
\begin{tabular}{p{1.6cm}p{1.4cm}p{1.4cm}p{1.2cm}}
\hline
 & $\sharp$ total  & $\sharp$ avg  & $\sharp$ max  \\
\hline
Method A & $788$ & $9$ & $32$ \\
Method B & $659$ & $7$ & $31$ \\
\hline\\
\end{tabular}
\caption{Newton's iterations for the mesh $200 \times 120$ for the filling case using Brooks and Corey model.}
\label{table:newtonBCfill}
\end{table}

    Let us now evaluate the saturation absolute error between results obtained with Method A and Method B. In Figure \ref{img:errFillBC} we plot the absolute-error distribution over the domain at three different times: when water crosses the {interface} between $\Omega_1$ and $\Omega_3$, when cells around this interface are almost saturated and at final time. \newline
       \begin{minipage}{0.03\textwidth}
    \vspace{-1.2cm}
        \includegraphics[height = 3.5cm]{./img/satAbsErrorFillBC/leg.jpg}
    \end{minipage}
    \begin{minipage}{0.9\textwidth}
    \begin{center}
        \includegraphics[width=4.4cm]{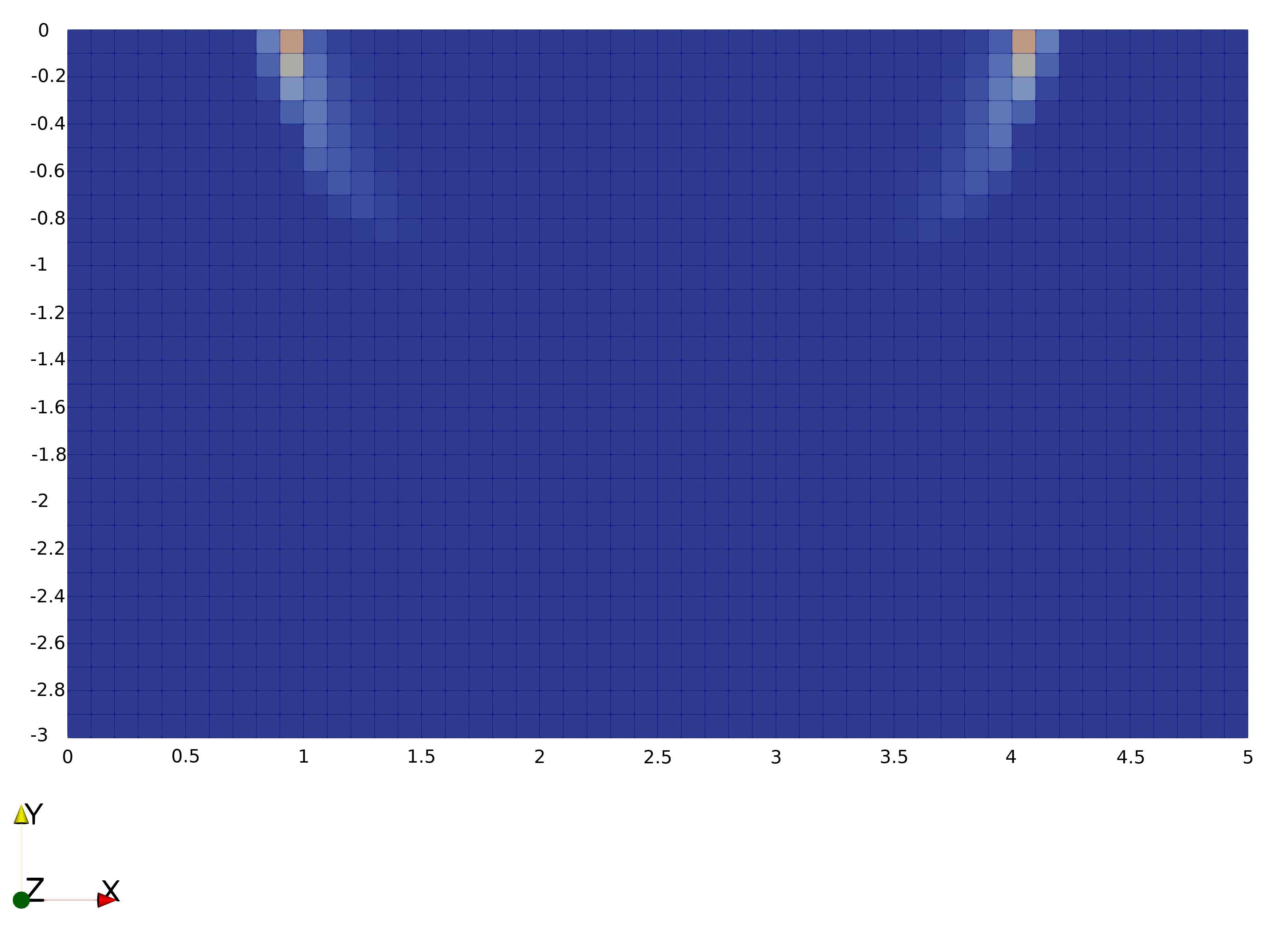}
        \includegraphics[width=4.4cm]{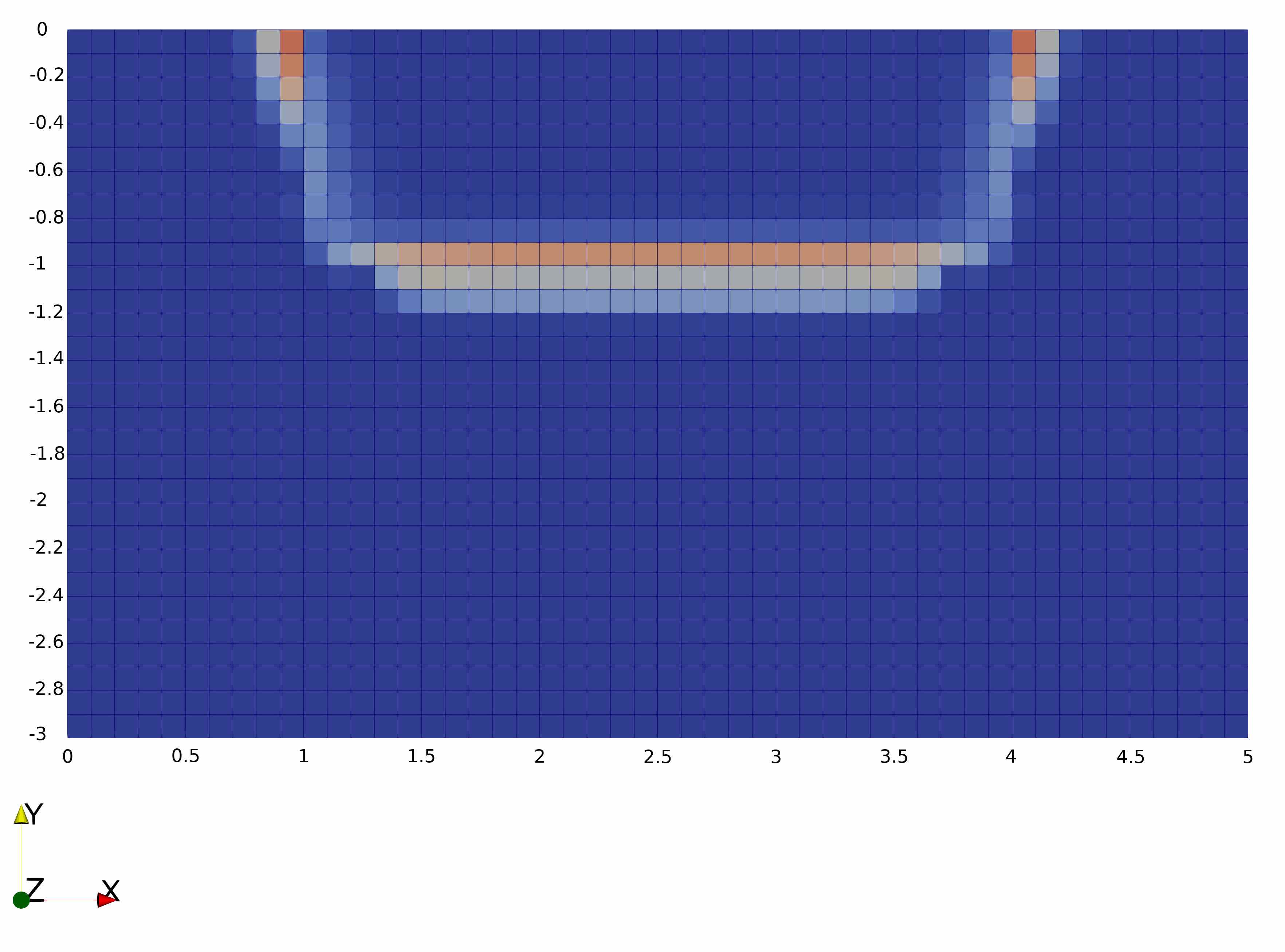}
        \includegraphics[width=4.4cm]{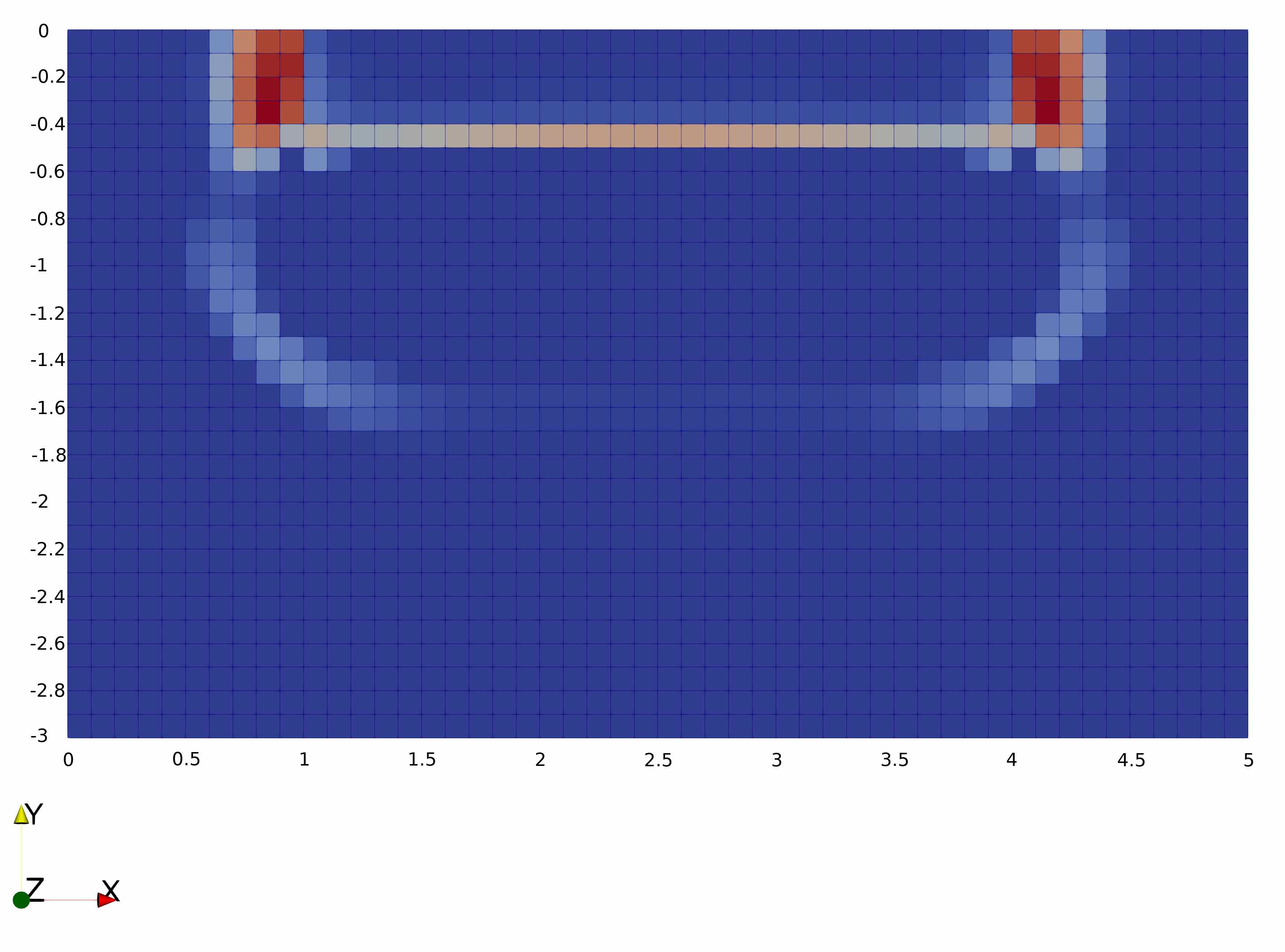}
        \captionof{figure}{Saturation absolute error between Method A and Method B for the filling case with Brooks and Corey model at $t\in\{20\cdot 10^3s, 30\cdot 10^3s, 86\cdot 10^3s\}$.}
        \label{img:errFillBC}
        \end{center}
    \end{minipage}

\subsubsection{Van Genuchten-Mualem model: filling case}
For the filling case with the Van Genuchten model, the convergence error is given in Figure \ref{img:convVGfill}. Both methods exhibit a linear rate of convergence.
On the other hand, the error is slightly larger with method A than with method B.
The total, average and maximal number of Newton iterations are given in Table \ref{table:newtonVGfill}.

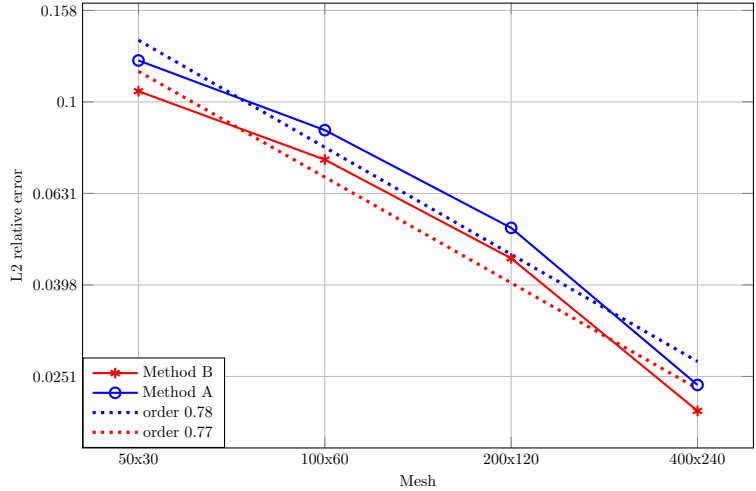
\begin{figure}[htp]
\centering
\begin{tikzpicture}
\begin{axis}[grid=major,
width=0.7\textwidth,
height=0.5\textwidth,
xlabel={Mesh},
ylabel={L2 relative error},
ymode=log,
xmode=log,
xticklabels={50x30, 100x60, 200x120, 400x240},
xtick={1500,6000,24000, 96000},
cycle list name=MyCyclelist1,
legend style={at={(0,0)},anchor=south west},nodes={scale=0.6, transform shape}]
\addplot table[x=mesh, y=L2]{errorVGfill.txt};
\addlegendentry{Method B };
\addplot table[x=mesh, y=L2]{errorVGfillv2.txt};
\addlegendentry{Method A };
\addplot[blue,dotted,domain=15e+2:96e+3,line width= 1.2pt] {2.3431041/(x)^0.388801424};
\addlegendentry{order $0.78$}; 
\addplot[red,dotted,domain=15e+2:96e+3,line width= 1.2pt] {1.9330889/(x)^0.383939001};
\addlegendentry{order $0.77$}; 
\end{axis}
\end{tikzpicture}
\caption{$L^2(Q_T)$ relative error in saturation for the filling case using the Van Genuchten model.}\label{img:convVGfill}
\end{figure}

\begin{table}[htp!]
\centering
\begin{tabular}{p{1.6cm}p{1.4cm}p{1.4cm}p{1.2cm}}
\hline
 & $\sharp$ total  & $\sharp$ avg  & $\sharp$ max  \\
\hline
Method A & $959$ & $5$ & $15$ \\
Method B & $782$ & $4$ & $15$ \\
\hline\\
\end{tabular}
\caption{Newton's iterations for the mesh $200 \times 120$ for the filling case using the Van Genuchten model.}
\label{table:newtonVGfill}
\end{table}

{Figure~\ref{img:errFillVG} shows the localization of the difference for the numerical solutions provided by the two methods A and B. Unsurprisingly, the difference is located in the neighborhood of the interfaces. Moreover, as suggested by Figures~\ref{img:convBCfill} and~\ref{img:convVGfill}, the influence of the introduction of additional interface unknowns (method A) has a lower impact for van Genuchten-Mualem nonlinearities than for Brook-Corey nonlinearities.}

   \begin{minipage}{0.03\textwidth}
    \vspace{-1.2cm}
        \includegraphics[height=3cm]{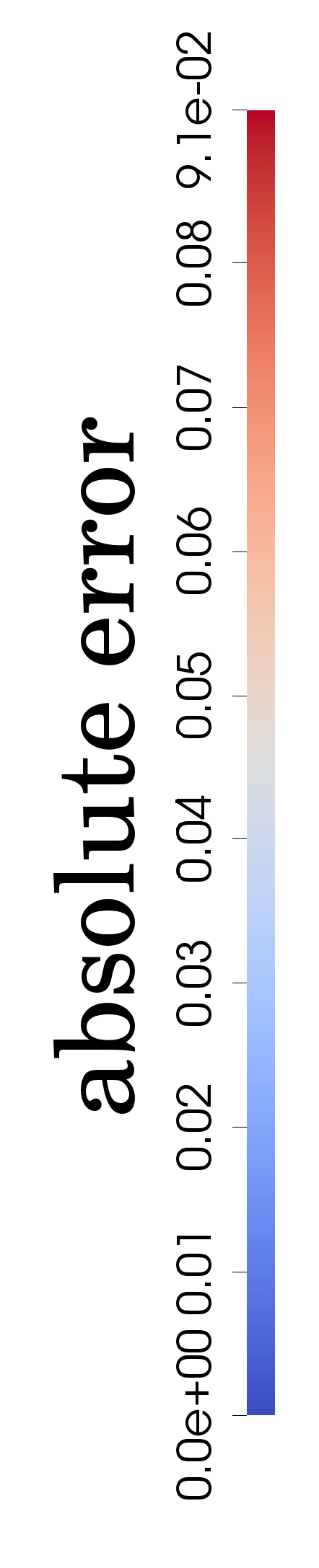}
    \end{minipage}
    \begin{minipage}{0.9\textwidth}
    \begin{center}
        \includegraphics[width=4.4cm]{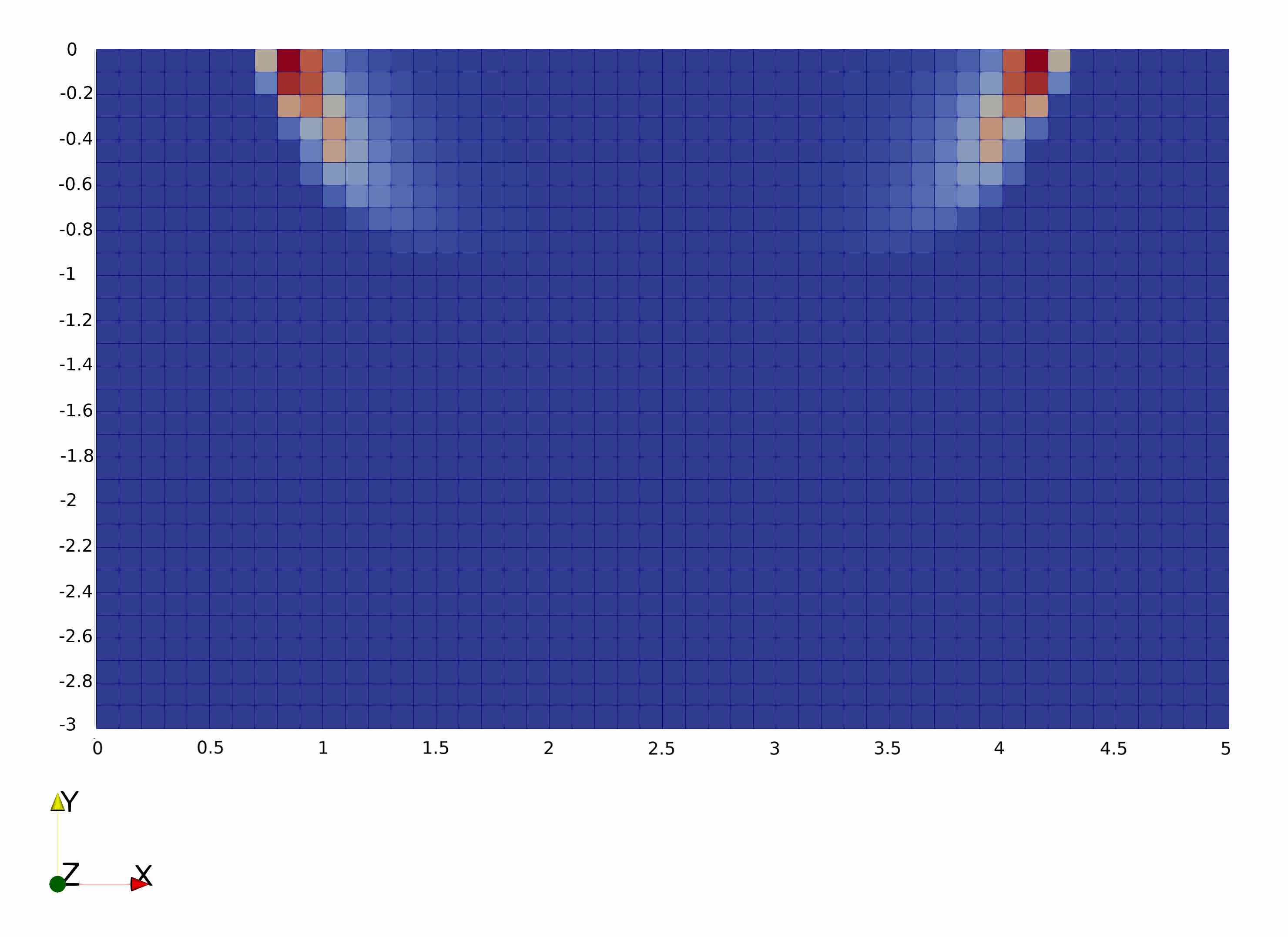}
        \includegraphics[width=4.4cm]{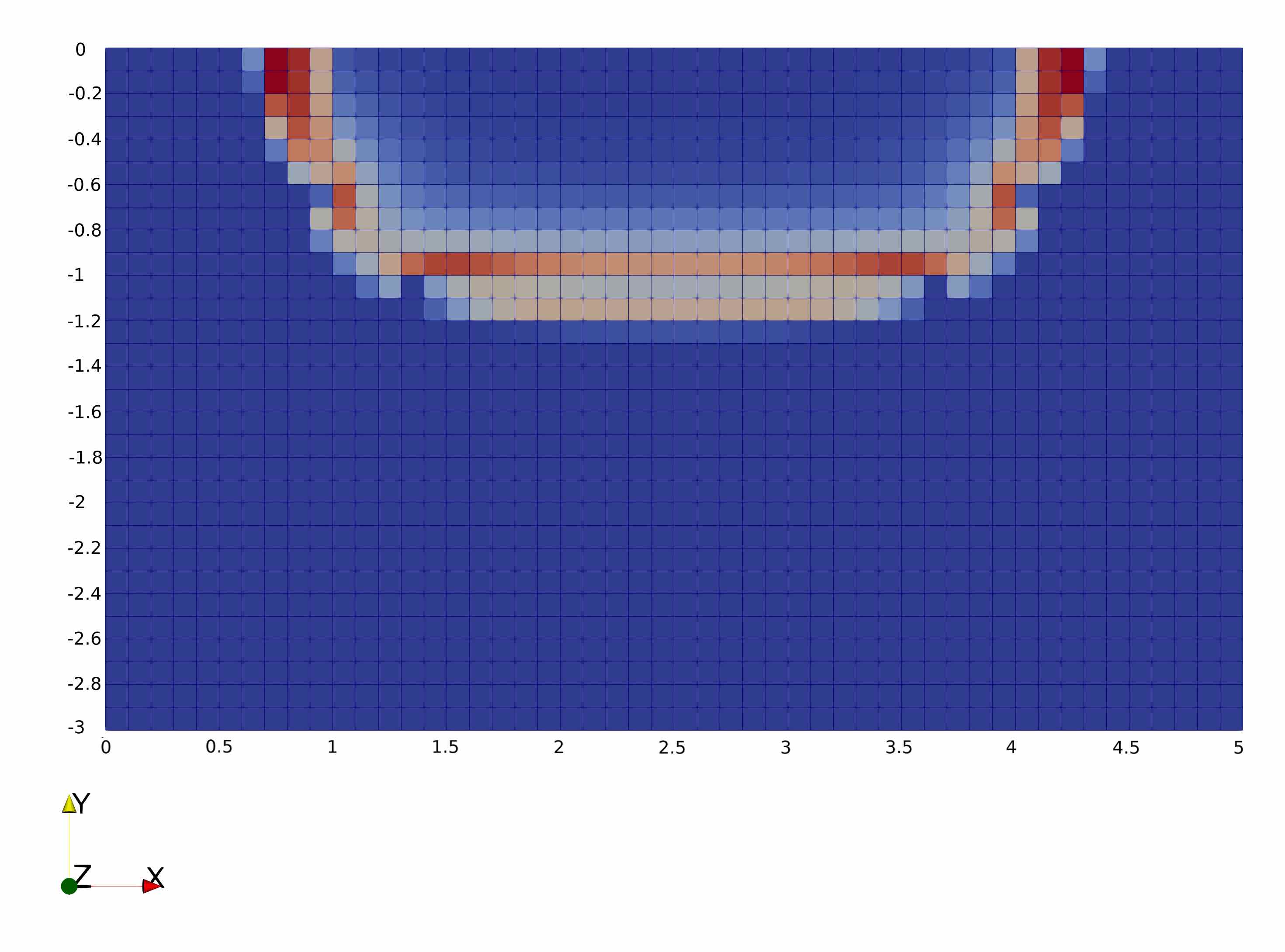}
        \includegraphics[width=4.4cm]{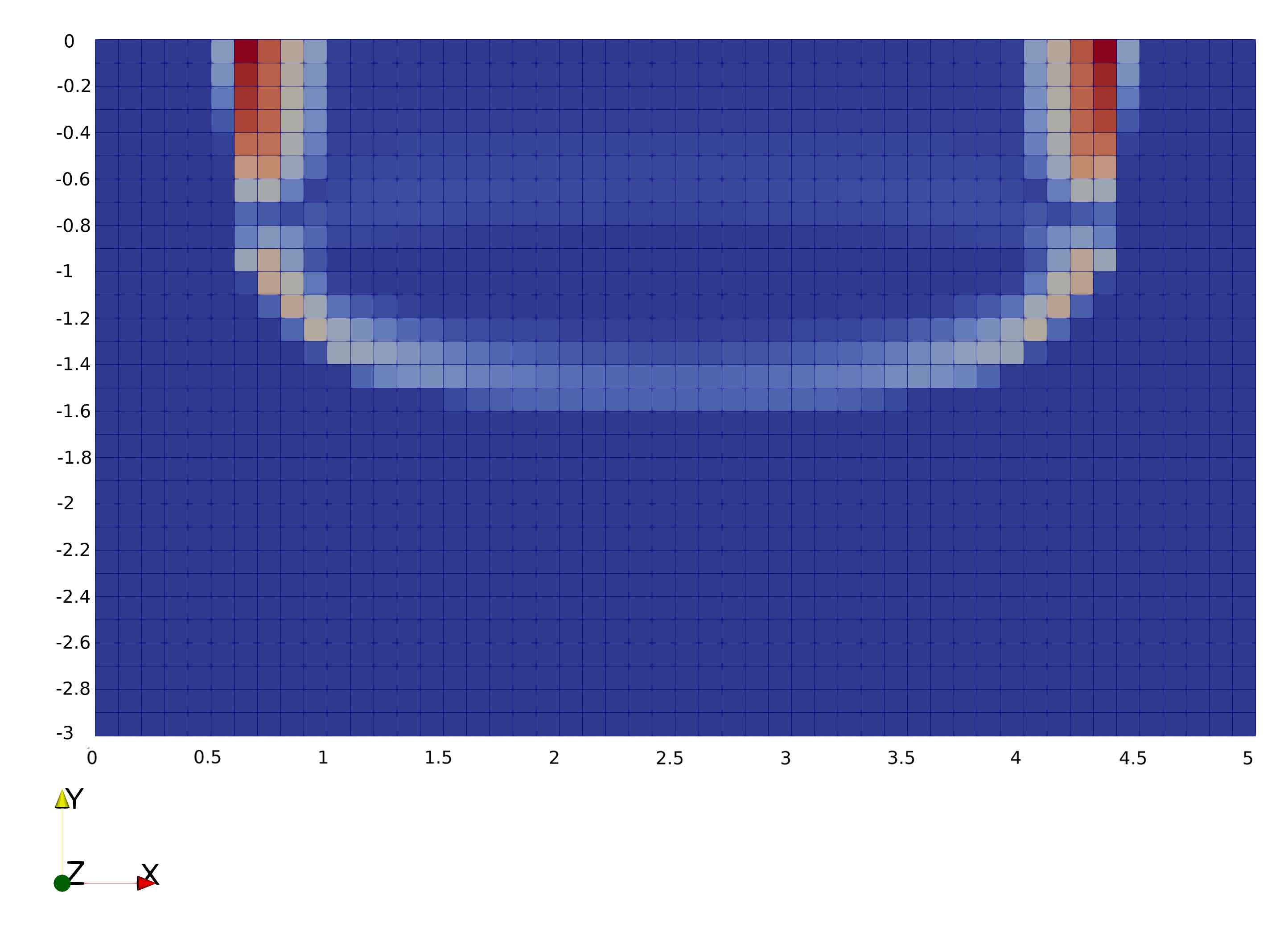}
        \captionof{figure}{Saturation absolute error between Method A and Method B for the filling case with the Van Genuchten model at $t\in\{27.5\cdot 10^3s, 45\cdot 10^3s, 86\cdot 10^3s\}$.}
        \label{img:errFillVG}
        \end{center}
    \end{minipage}

\subsubsection{Van Genuchten-Mualem model: drainage case}
For the drainage case with the Van Genuchten model, the convergence error is given in Figure \ref{img:convVGdry}. Both methods exhibit a linear rate of convergence.
Moreover, the error is slightly larger with method B than with method A.
The total, average and maximal number of Newton iterations are given in Table \ref{table:newtonVGdry}.

\begin{figure}[htp]
\centering
\begin{tikzpicture}
\begin{axis}[grid=major,
width=0.7\textwidth,
height=0.5\textwidth,
xlabel={Mesh},
ylabel={L2 relative error},
ymode=log,
xmode=log,
xticklabels={50x30, 100x60, 200x120, 400x240},
xtick={1500,6000,24000, 96000},
cycle list name=MyCyclelist1,
legend style={at={(0,0)},anchor=south west},nodes={scale=0.6, transform shape}]
\addplot table[x=mesh, y=L2]{errorVGdry.txt};
\addlegendentry{Method B };
\addplot table[x=mesh, y=L2]{errorVGdryv2.txt};
\addlegendentry{Method A };
\addplot[blue,dotted,domain=15e+2:96e+3,line width= 1.2pt] {0.74683451/(x)^0.637271218};
\addlegendentry{order $1.27$}; 
\addplot[red,dotted,domain=15e+2:96e+3,line width= 1.2pt] {0.72299086/(x)^0.528197755};
\addlegendentry{order $1.05$}; 
\end{axis}
\end{tikzpicture}
\caption{$L^2(Q_T)$ relative error in saturation for the drainage case using the Van Genuchten model.}\label{img:convVGdry}
\end{figure}
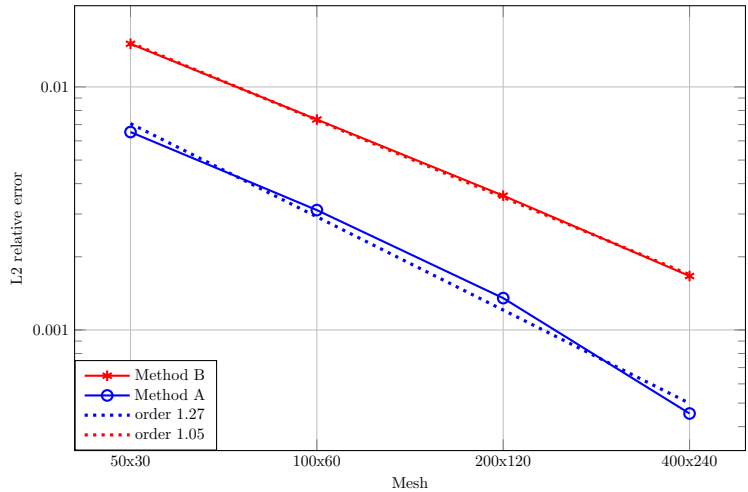

\begin{table}[htp!]
\centering
\begin{tabular}{p{1.6cm}p{1.4cm}p{1.4cm}p{1.2cm}}
\hline
 & $\sharp$ total  & $\sharp$ avg  & $\sharp$ max  \\
\hline
Method A & $3523$ & $2$ & $20$ \\
Method B & $2845$ & $2$ & $29$ \\
\hline\\
\end{tabular}
\caption{Newton's iterations for the mesh $200 \times 120$ for the filling case using the Van Genuchten model.}
\label{table:newtonVGdry}
\end{table}

    Let us now evaluate the saturation absolute error between results obtained with Method A and Method B. In Figure \ref{img:errDryVG} we plot the absolute-error distribution over the domain at three different times: when the cells line in $\Omega_1$ above the interface between $\Omega_1$ and $\Omega_3$ starts drying,  when the cells line in $\Omega_2$ below the interface between $\Omega_3$ and $\Omega_2$ starts drying and at final time. \newline
    \begin{figure}
    \centering
        \includegraphics[width=4.4cm]{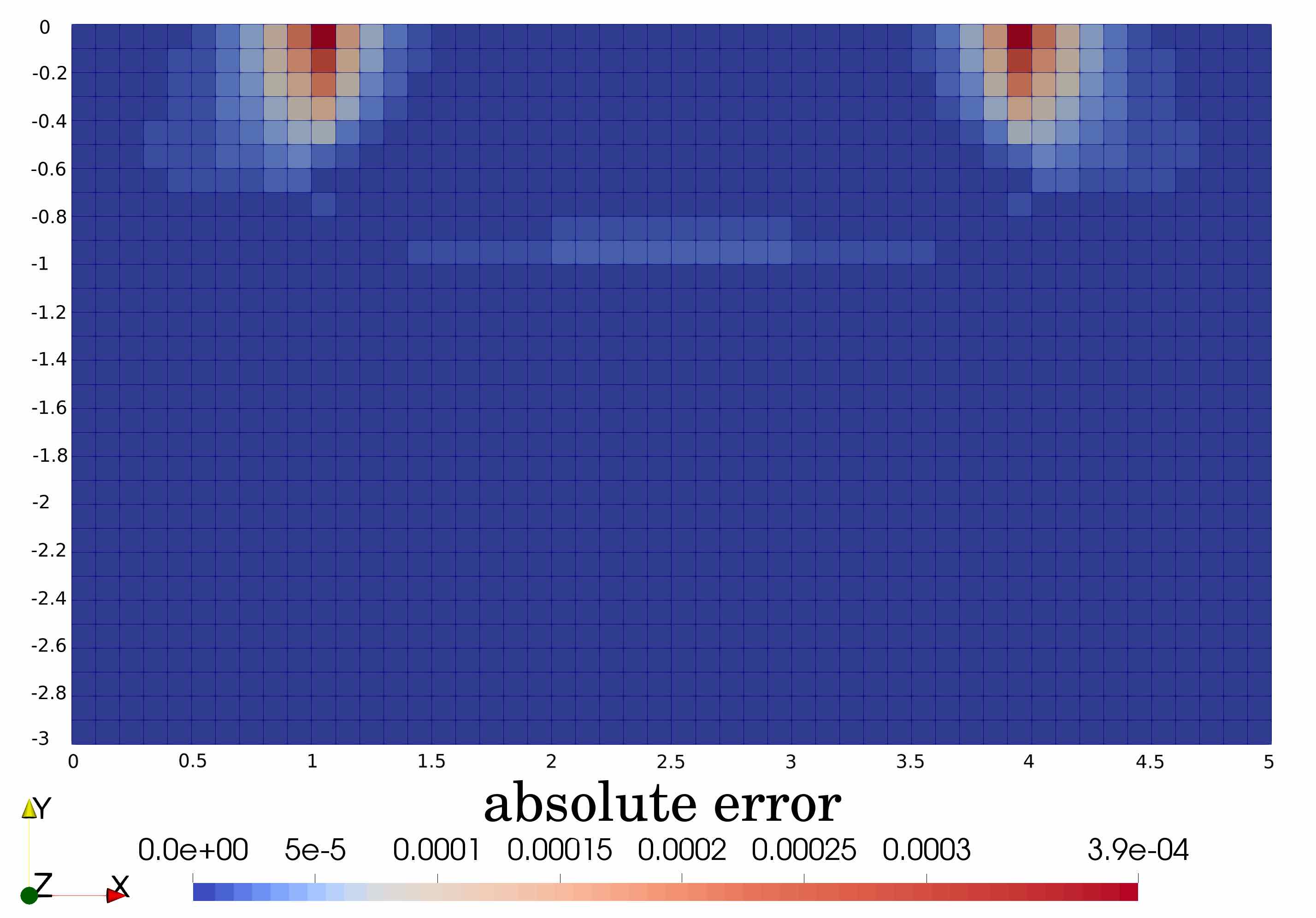}
        \includegraphics[width=4.4cm]{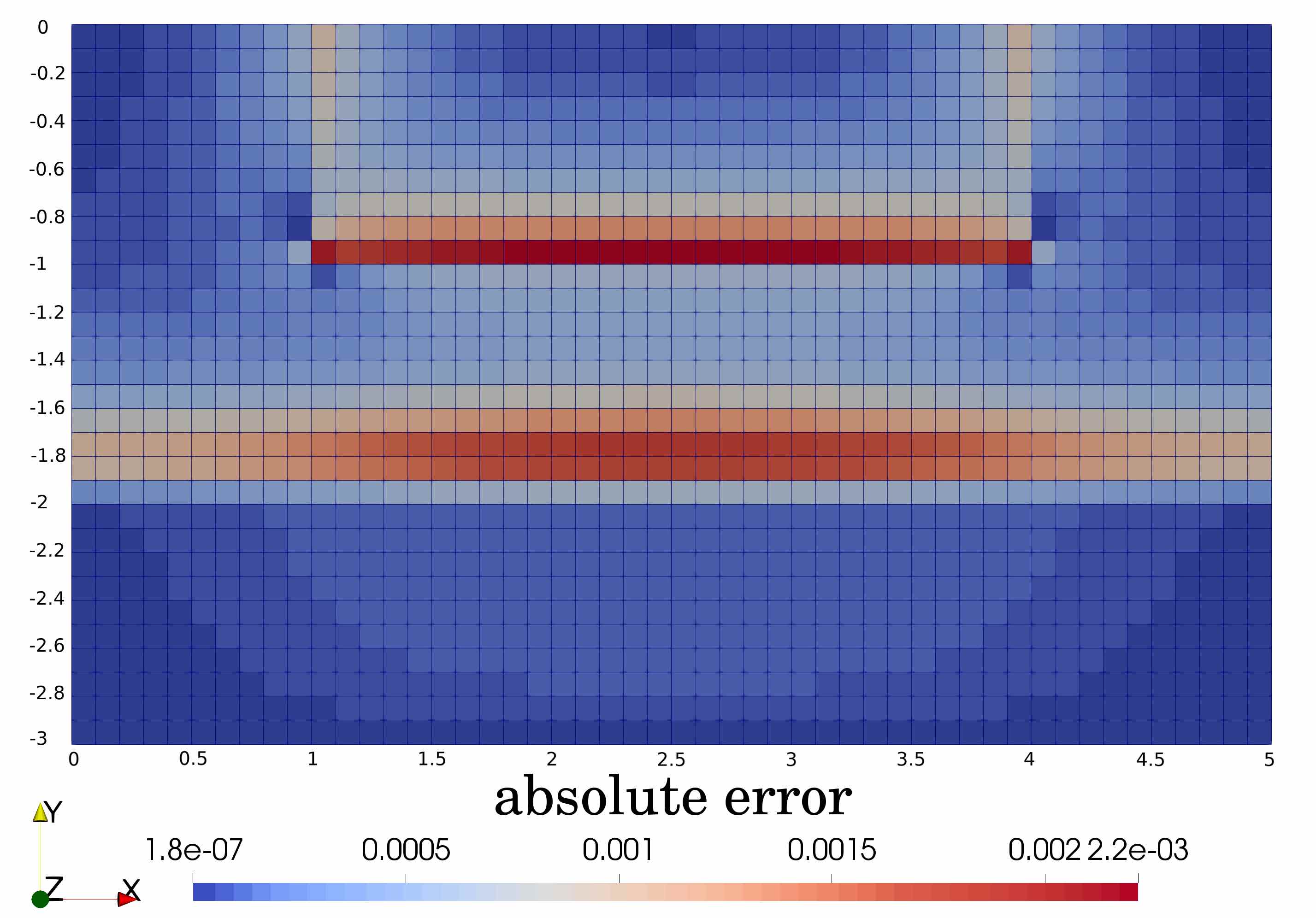}
        \includegraphics[width=4.4cm]{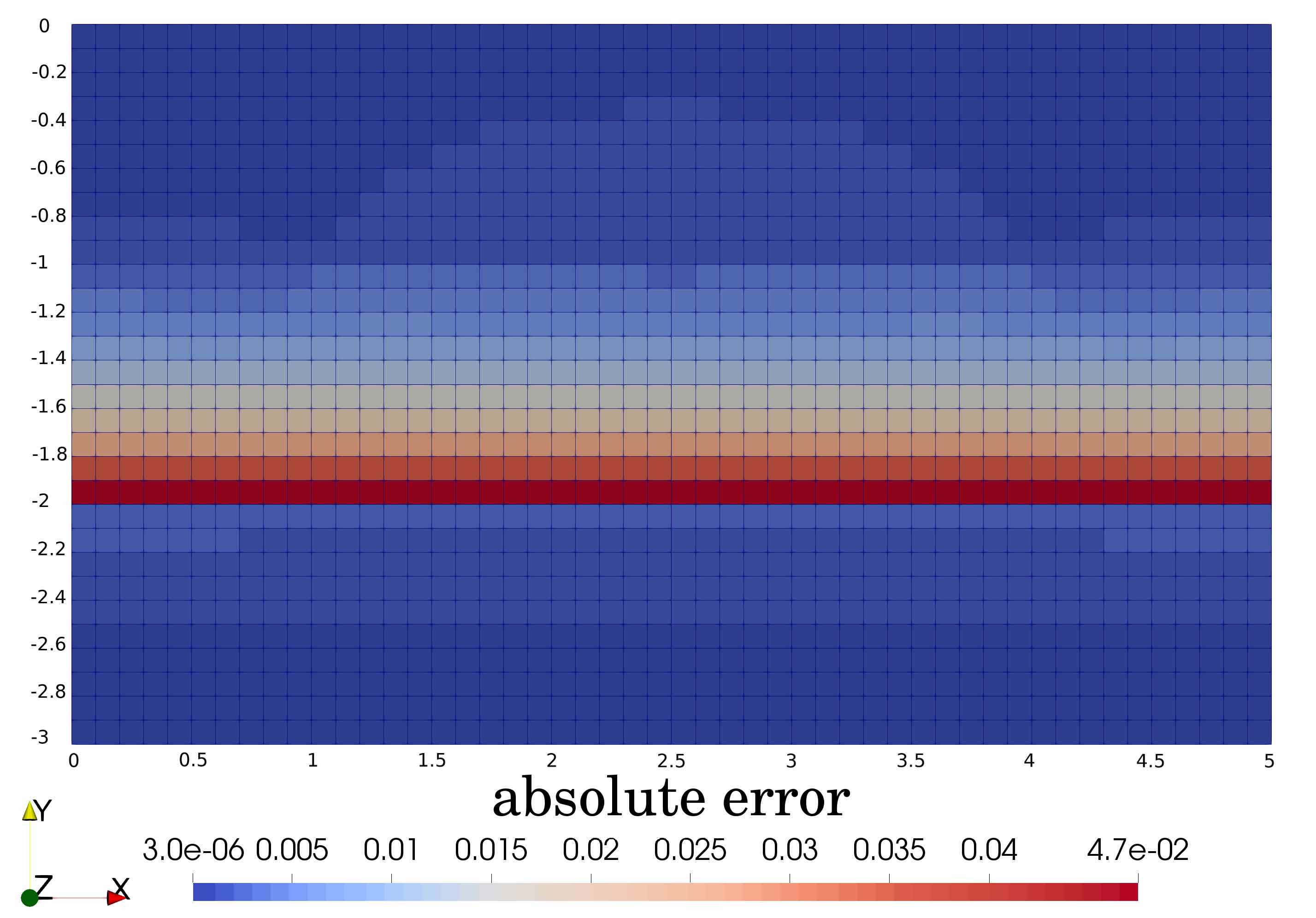}
        \captionof{figure}{Saturation absolute error between Method A and Method B for the drainage case with the van Genuchten-Mualem model at $t\in\{1.12\cdot 10^4s, 10.56\cdot 10^4s, 105\cdot 10^4s\}$.}
        \label{img:errDryVG}
    \end{figure}

\subsubsection{{Influence of the parameter $\delta$}}

Let us now analyze how the thickness of the thin cells employed in Method A affects the accuracy of the solution obtained with this method. We consider the filling and drainage cases along with the Brooks and Corey model and evaluate the relative {$L^2(Q_T)$} error between the solution obtained on the $200\times120$ cells mesh using $\delta\in\{10^{-2}m,10^{-4}m,10^{-6}m\}$ with respect to the reference solution obtained on the $800\times480$ cells mesh using $\delta_{ref}=10^{-6}m$. {As shown in Figure \ref{img:convVirtualCellBC},  the value of $\delta$ does not have a significant influence on the overall error as soon as $\delta$ is small enough. We also observe a moderate influence on the robustness of the non-linear solver for the values considered here.}

\begin{figure}[htp]
\centering
\begin{tikzpicture}
\begin{axis}[grid=major,
width=0.7\textwidth,
height=0.3\textwidth,
xlabel={Thickness of thin cells, Method A},
ylabel={L2 relative error},
ymode=log,
xmode=log,
xticklabels={1e-6, 1e-4, 1e-2},
xtick={1e-6, 1e-4, 1e-2},
cycle list name=MyCyclelist1,
legend style={at={(1,0)},anchor=south east},nodes={scale=0.6, transform shape}]
\addplot table[x=delta, y=L2]{compErrorBCfill.txt};
\addlegendentry{Filling case};
\addplot table[x=delta, y=L2]{compErrBCdry.txt};
\addlegendentry{drainage case};
\end{axis}
\end{tikzpicture}
\caption{$L^2(Q_T)$ relative error in saturation as a function of the thickness $\delta$ of the thin cells with Method A using the $200\times120$ cells mesh.}\label{img:convVirtualCellBC}
\end{figure}
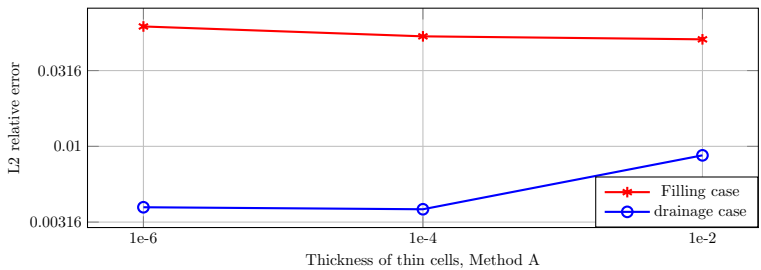

\section{Conclusions and perspectives}

{This article aimed at proving that standard upstream mobility finite volume schemes for variable saturated porous media flows still converge in highly heterogeneous contexts without any specific treatment of the rock type discontinuities. The scheme is indeed shown to satisfy some energy stability which provides enough a priori estimates to carry out its numerical analysis. First, the existence of a unique solution to the nonlinear system stemming from the scheme is established thanks to a topological degree argument and from the monotonicity of the scheme. Besides, a rigorous mathematical convergence proof is conducted, based on compactness arguments. No error estimate can then be deduced from our analysis.}

{Because of the choice of a backward Euler in time discretization and from the upwind choice of the mobilities, a first order in time and space accuracy is expected in the case of homogeneous computational domains. We show in numerical experiments that without any particular treatment of the interfaces at rock discontinuities, this first order accuracy can be lost, especially in the case of Brooks-Corey nonlinearities. This motivates the introduction of a specific treatment of the interfaces. The approach we propose here is based on the introduction of additional unknowns located in fictitious small additional cells on both sides of each interface. Even though the rigorous convergence proof of this approach is not provided here in the multidimensional setting -- such a proof can for instance be done by writing the scheme with the specific treatment of the interface (method A) as a perturbation of the scheme without any particular treatment of the interface (method B)--, the numerical experiments show that it allows to recover the first order accuracy without having major impacts on the implementation and on the behavior of the numerical solver.}

For future researches, we suggest to test the so-called method A on a two-phase flow test {and to compare it to the approaches presented in~\cite{BDMQ_HAL}}. Moreover, in a forthcoming work, we propose two other methods to really impose the pressure continuity condition at interfaces. A comparison between all methods will be shown.

\end{document}